\documentclass[12pt]{amsart}
\usepackage{mathrsfs}
\usepackage{cases}
\usepackage{epic}
\usepackage{amsfonts}
\usepackage{graphicx}
\usepackage{amsmath}
\usepackage{amssymb, upgreek}
\usepackage{bm}
\usepackage{latexsym,todonotes}
\usepackage{pdflscape}
\usepackage[all]{xypic}
\usepackage[all]{xy}
\usepackage{color}
\usepackage{colordvi}
\usepackage{multicol}
\usepackage[normalem]{ulem}
\usepackage[linktocpage=true]{hyperref}
\hypersetup{colorlinks,linkcolor=blue,urlcolor=cyan,citecolor=blue}
\def \bvs{{\boldsymbol{\varsigma}}}
\newcommand{\vs}{\varsigma}
\allowdisplaybreaks

\begin{document}
\input xy
\xyoption{all}

\newtheorem{innercustomthm}{{\bf Theorem}}
\newenvironment{customthm}[1]
  {\renewcommand\theinnercustomthm{#1}\innercustomthm}
  {\endinnercustomthm}

  \newtheorem{innercustomcor}{{\bf Corollary}}
\newenvironment{customcor}[1]
  {\renewcommand\theinnercustomcor{#1}\innercustomcor}
  {\endinnercustomthm}

  \newtheorem{innercustomprop}{{\bf Proposition}}
\newenvironment{customprop}[1]
  {\renewcommand\theinnercustomprop{#1}\innercustomprop}
  {\endinnercustomthm}

\newcommand{\iadd}{\operatorname{iadd}\nolimits}

\renewcommand{\mod}{\operatorname{mod}\nolimits}
\newcommand{\proj}{\operatorname{proj}\nolimits}
\newcommand{\inj}{\operatorname{inj.}\nolimits}
\newcommand{\rad}{\operatorname{rad}\nolimits}
\newcommand{\Span}{\operatorname{Span}\nolimits}
\newcommand{\soc}{\operatorname{soc}\nolimits}
\newcommand{\ind}{\operatorname{inj.dim}\nolimits}
\newcommand{\Ginj}{\operatorname{Ginj}\nolimits}
\newcommand{\res}{\operatorname{res}\nolimits}
\newcommand{\np}{\operatorname{np}\nolimits}
\newcommand{\Fac}{\operatorname{Fac}\nolimits}
\newcommand{\Aut}{\operatorname{Aut}\nolimits}
\newcommand{\DTr}{\operatorname{DTr}\nolimits}
\newcommand{\TrD}{\operatorname{TrD}\nolimits}
\newcommand{\Gr}{\operatorname{Gr}\nolimits}

\newcommand{\Mod}{\operatorname{Mod}\nolimits}
\newcommand{\R}{\operatorname{R}\nolimits}
\newcommand{\End}{\operatorname{End}\nolimits}
\newcommand{\lf}{\operatorname{l.f.}\nolimits}
\newcommand{\Iso}{\operatorname{Iso}\nolimits}
\newcommand{\aut}{\operatorname{Aut}\nolimits}
\newcommand{\Ui}{{\mathbf U}^\imath}
\newcommand{\UU}{{\mathbf U}\otimes {\mathbf U}}
\newcommand{\UUi}{(\UU)^\imath}
\newcommand{\tUU}{{\tU}\otimes {\tU}}
\newcommand{\tUUi}{(\tUU)^\imath}
\newcommand{\tUi}{\widetilde{{\mathbf U}}^\imath}
\newcommand{\sqq}{{\bf v}}
\newcommand{\sqvs}{\sqrt{\vs}}
\newcommand{\dbl}{\operatorname{dbl}\nolimits}
\newcommand{\swa}{\operatorname{swap}\nolimits}
\newcommand{\Gp}{\operatorname{Gp}\nolimits}

\newcommand{\U}{{\mathbf U}}
\newcommand{\tU}{\widetilde{\mathbf U}}

\newcommand{\ov}{\overline}
\newcommand{\tk}{\widetilde{k}}
\newcommand{\tK}{\widetilde{K}}

\newcommand{\tH}{\widetilde{\ch}}

\newcommand{\LL}{\texttt{L}}
\newcommand{\RR}{\texttt{R}}

\newcommand{\utM}{\operatorname{\cs\cd\ch}\nolimits}
\newcommand{\tM}{\operatorname{\cs\cd\widetilde{\ch}}\nolimits}
\newcommand{\rM}{\operatorname{\cs\cd\ch_{\rm{red}}}\nolimits}
\newcommand{\utMH}{\operatorname{\cs\cd\ch(\Lambda^\imath)}\nolimits}
\newcommand{\tMH}{\operatorname{\cs\cd\widetilde{\ch}(\Lambda^\imath)}\nolimits}
\newcommand{\rMH}{\operatorname{\cs\cd\ch_{\rm{red}}(\Lambda^\imath)}\nolimits}
\newcommand{\utMHg}{\operatorname{\ch(Q,\btau)}\nolimits}
\newcommand{\tMHg}{\operatorname{\widetilde{\ch}(Q,\btau)}\nolimits}
\newcommand{\rMHg}{\operatorname{\ch_{\rm{red}}(Q,\btau)}\nolimits}

\newcommand{\rMHd}{\operatorname{\cs\cd\ch_{\rm{red}}(\Lambda^\imath)_{\bvsd}}\nolimits}
\newcommand{\tMHd}{\operatorname{\cs\cd\widetilde{\ch}(\Lambda^\imath)_{\bvsd}}\nolimits}

\newcommand{\tMHl}{\cs\cd\widetilde{\ch}({\bs}_\ell\Lambda^\imath)}
\newcommand{\rMHl}{\cs\cd\ch_{\rm{red}}({\bs}_\ell\Lambda^\imath)_{\bvsd}}
\newcommand{\tMHi}{\cs\cd\widetilde{\ch}({\bs}_i\Lambda^\imath)}
\newcommand{\rMHi}{\cs\cd\ch_{\rm{red}}({\bs}_i\Lambda^\imath)_{\bvsd}}
\newcommand{\tMHgi}{\widetilde{\ch}({\bs}_i Q,\btau)}

\newcommand{\utGpg}{\operatorname{\ch^{\rm Gp}(Q,\btau)}\nolimits}
\newcommand{\tGpg}{\operatorname{\widetilde{\ch}^{\rm Gp}(Q,\btau)}\nolimits}
\newcommand{\rGpg}{\operatorname{\ch_{red}^{\rm Gp}(Q,\btau)}\nolimits}

\newcommand{\colim}{\operatorname{colim}\nolimits}
\newcommand{\gldim}{\operatorname{gl.dim}\nolimits}
\newcommand{\cone}{\operatorname{cone}\nolimits}
\newcommand{\rep}{\operatorname{rep}\nolimits}
\newcommand{\Ext}{\operatorname{Ext}\nolimits}
\newcommand{\Tor}{\operatorname{Tor}\nolimits}
\newcommand{\Hom}{\operatorname{Hom}\nolimits}
\newcommand{\Top}{\operatorname{top}\nolimits}
\newcommand{\Coker}{\operatorname{Coker}\nolimits}
\newcommand{\thick}{\operatorname{thick}\nolimits}
\newcommand{\rank}{\operatorname{rank}\nolimits}
\newcommand{\Gproj}{\operatorname{Gproj}\nolimits}
\newcommand{\Len}{\operatorname{Length}\nolimits}
\newcommand{\RHom}{\operatorname{RHom}\nolimits}
\renewcommand{\deg}{\operatorname{deg}\nolimits}
\renewcommand{\Im}{\operatorname{Im}\nolimits}
\newcommand{\Ker}{\operatorname{Ker}\nolimits}
\newcommand{\Coh}{\operatorname{Coh}\nolimits}
\newcommand{\Id}{\operatorname{Id}\nolimits}
\newcommand{\Qcoh}{\operatorname{Qch}\nolimits}
\newcommand{\CM}{\operatorname{CM}\nolimits}
\newcommand{\sgn}{\operatorname{sgn}\nolimits}
\newcommand{\Gdim}{\operatorname{G.dim}\nolimits}
\newcommand{\fpr}{\operatorname{\mathcal{P}^{\leq1}}\nolimits}

\newcommand{\For}{\operatorname{{\bf F}or}\nolimits}
\newcommand{\coker}{\operatorname{Coker}\nolimits}
\renewcommand{\dim}{\operatorname{dim}\nolimits}
\newcommand{\rankv}{\operatorname{\underline{rank}}\nolimits}
\newcommand{\dimv}{{\operatorname{\underline{dim}}\nolimits}}
\newcommand{\diag}{{\operatorname{diag}\nolimits}}
\newcommand{\qbinom}[2]{\begin{bmatrix} #1\\#2 \end{bmatrix} }

\renewcommand{\Vec}{{\operatorname{Vec}\nolimits}}
\newcommand{\pd}{\operatorname{proj.dim}\nolimits}
\newcommand{\gr}{\operatorname{gr}\nolimits}
\newcommand{\id}{\operatorname{Id}\nolimits}
\newcommand{\Res}{\operatorname{Res}\nolimits}

\newcommand{\pdim}{\operatorname{proj.dim}\nolimits}
\newcommand{\idim}{\operatorname{inj.dim}\nolimits}
\newcommand{\Gd}{\operatorname{G.dim}\nolimits}
\newcommand{\Ind}{\operatorname{Ind}\nolimits}
\newcommand{\add}{\operatorname{add}\nolimits}
\newcommand{\pr}{\operatorname{pr}\nolimits}
\newcommand{\oR}{\operatorname{R}\nolimits}
\newcommand{\oL}{\operatorname{L}\nolimits}
\newcommand{\ext}{{ \mathfrak{Ext}}}
\newcommand{\Perf}{{\mathfrak Perf}}
\def\scrP{\mathscr{P}}
\newcommand{\bk}{{\mathbb K}}
\newcommand{\cc}{{\mathcal C}}
\newcommand{\gc}{{\mathcal GC}}
\newcommand{\dg}{{\rm dg}}
\newcommand{\ce}{{\mathcal E}}
\newcommand{\cs}{{\mathcal S}}
\newcommand{\cl}{{\mathcal L}}
\newcommand{\cf}{{\mathcal F}}
\newcommand{\cx}{{\mathcal X}}
\newcommand{\cy}{{\mathcal Y}}
\newcommand{\ct}{{\mathcal T}}
\newcommand{\cu}{{\mathcal U}}
\newcommand{\cv}{{\mathcal V}}
\newcommand{\cn}{{\mathcal N}}
\newcommand{\mcr}{{\mathcal R}}
\newcommand{\ch}{{\mathcal H}}
\newcommand{\ca}{{\mathcal A}}
\newcommand{\cb}{{\mathcal B}}
\newcommand{\ci}{{\I}_{\btau}}
\newcommand{\cj}{{\mathcal J}}
\newcommand{\cm}{{\mathcal M}}
\newcommand{\cp}{{\mathcal P}}
\newcommand{\cg}{{\mathcal G}}
\newcommand{\cw}{{\mathcal W}}
\newcommand{\co}{{\mathcal O}}
\newcommand{\cq}{{Q^{\rm dbl}}}
\newcommand{\cd}{{\mathcal D}}
\newcommand{\cz}{{\mathcal Z}}
\newcommand{\ck}{\widetilde{\mathcal K}}
\newcommand{\calr}{{\mathcal R}}
\newcommand{\La}{\Lambda}
\newcommand{\ol}{\overline}
\newcommand{\ul}{\underline}
\newcommand{\st}{[1]}
\newcommand{\ow}{\widetilde}
\renewcommand{\P}{\mathbf{P}}
\newcommand{\pic}{\operatorname{Pic}\nolimits}
\newcommand{\Spec}{\operatorname{Spec}\nolimits}

\newtheorem{theorem}{Theorem}[section]
\newtheorem{acknowledgement}[theorem]{Acknowledgement}
\newtheorem{algorithm}[theorem]{Algorithm}
\newtheorem{axiom}[theorem]{Axiom}
\newtheorem{case}[theorem]{Case}
\newtheorem{claim}[theorem]{Claim}
\newtheorem{conclusion}[theorem]{Conclusion}
\newtheorem{condition}[theorem]{Condition}
\newtheorem{conjecture}[theorem]{Conjecture}
\newtheorem{construction}[theorem]{Construction}
\newtheorem{corollary}[theorem]{Corollary}
\newtheorem{criterion}[theorem]{Criterion}
\newtheorem{definition}[theorem]{Definition}
\newtheorem{example}[theorem]{Example}
\newtheorem{exercise}[theorem]{Exercise}
\newtheorem{lemma}[theorem]{Lemma}
\newtheorem{notation}[theorem]{Notation}
\newtheorem{problem}[theorem]{Problem}
\newtheorem{proposition}[theorem]{Proposition}
\newtheorem{solution}[theorem]{Solution}
\newtheorem{summary}[theorem]{Summary}
\numberwithin{equation}{section}

\theoremstyle{remark}
\newtheorem{remark}[theorem]{Remark}
\newcommand{\Pd}{\pi_*}
\def \bvs{{\boldsymbol{\varsigma}}}
\def \bvsd{{\boldsymbol{\varsigma}_{\diamond}}}
\def \btau{{{\uptau}}}

\def \bp{{\mathbf p}}
\def \bq{{\bm q}}
\def \bv{{v}}
\def \bs{{\bm s}}

\def \bK{\mathbf{K}}

\newcommand{\bfv}{\mathbf{v}}
\def \bA{{\mathbf A}}
\def \ba{{\mathbf a}}
\def \bL{{\mathbf L}}
\def \bF{{\mathbf F}}
\def \bS{{\mathbf S}}
\def \bC{{\mathbf C}}
\def \bU{{\mathbf U}}
\def \bc{{\mathbf c}}
\def \fpi{\mathfrak{P}^\imath}
\def \Ni{N^\imath}
\def \fp{\mathfrak{P}}
\def \fg{\mathfrak{g}}
\def \fk{\fg^\theta}  
\def \p{p}
\def \fn{\mathfrak{n}}
\def \fh{\mathfrak{h}}
\def \fu{\mathfrak{u}}
\def \fv{\mathfrak{v}}
\def \fa{\mathfrak{a}}
\def \Z{{\Bbb Z}}
\def \F{{\Bbb F}}
\def \D{{\Bbb D}}
\def \C{{\Bbb C}}
\def \N{{\Bbb N}}
\def \Q{{\Bbb Q}}
\def \G{{\Bbb G}}
\def \P{{\Bbb P}}
\def \K{{k}}
\def \E{{\Bbb K}}
\def \A{{\Bbb A}}
\def \L{{\Bbb L}}
\def \I{{\Bbb I}}
\def \BH{{\Bbb H}}
\def \T{{\mathcal T}}
\def \tT{\widetilde{\mathcal T}}
\def \tTL{\tT(\Lambda^\imath)}
\newcommand {\lu}[1]{\textcolor{red}{$\clubsuit$: #1}}

\newcommand{\nc}{\newcommand}
\newcommand{\browntext}[1]{\textcolor{brown}{#1}}
\newcommand{\greentext}[1]{\textcolor{green}{#1}}
\newcommand{\redtext}[1]{\textcolor{red}{#1}}
\newcommand{\bluetext}[1]{\textcolor{blue}{#1}}
\newcommand{\brown}[1]{\browntext{ #1}}
\newcommand{\green}[1]{\greentext{ #1}}
\newcommand{\red}[1]{\redtext{ #1}}
\newcommand{\blue}[1]{\bluetext{ #1}}

\newcommand{\wtodo}{\todo[inline,color=orange!20, caption={}]}
\newcommand{\lutodo}{\todo[inline,color=green!20, caption={}]}

\title[Hall algebras and quantum symmetric pairs I: foundations]{Hall algebras and quantum symmetric pairs I: foundations}

\author[Ming Lu]{Ming Lu}
\address{Department of Mathematics, Sichuan University, Chengdu 610064, P.R.China}
\email{luming@scu.edu.cn}

\author[Weiqiang Wang]{Weiqiang Wang}
\address{Department of Mathematics, University of Virginia, Charlottesville, VA 22904, USA}
\email{ww9c@virginia.edu}


\begin{abstract}
A quantum symmetric pair consists of a quantum group $\mathbf U$ and its coideal subalgebra ${\mathbf U}^{\imath}_{\boldsymbol{\varsigma}}$ with parameters $\boldsymbol{\varsigma}$ (called an $\imath$quantum group). We initiate a Hall algebra approach for the categorification of $\imath$quantum groups. A universal $\imath$quantum group $\widetilde{\mathbf U}^{\imath}$ is introduced and ${\mathbf U}^{\imath}_{\boldsymbol{\varsigma}}$ is recovered by a central reduction of $\widetilde{\mathbf U}^{\imath}$. The semi-derived Ringel-Hall algebras of the first author and Peng, which are closely related to semi-derived Hall algebras of Gorsky and motivated by Bridgeland's work, are extended to the setting of 1-Gorenstein algebras, as shown in Appendix A by the first author. A new class of 1-Gorenstein algebras (called $\imath$quiver algebras) arising from acyclic quivers with involutions is introduced. The semi-derived Ringel-Hall algebras for the Dynkin $\imath$quiver algebras are shown to be isomorphic to the universal quasi-split $\imath$quantum groups of finite type. Monomial bases and PBW bases for these Hall algebras and $\imath$quantum groups are constructed.
\end{abstract}

\maketitle
 \tableofcontents

\section{Introduction}
\subsection{Background}
\subsubsection{Hall algebras and quantum groups}

Ringel \cite{Rin} in 1990 constructed a Hall algebra associated to a Dynkin quiver $Q=(Q_0,Q_1)$ with a vertex set $Q_0=\I$ over a finite field $\mathbb F_q$, and identified its generic version with half a quantum group $\U^+ =\U^+_v(\fg)$, where $\fg$ is the Lie algebra with the underlying graph of $Q$ as its Dynkin diagram; see Green \cite{Gr} for an extension to acyclic quivers. Ringel's construction has led to a geometric Hall algebra realization of $\U^+$ by Lusztig, who in addition constructed its canonical basis \cite{L90}. These constructions can be regarded as earliest examples of categorifications of halves of quantum groups.

It took some time before a Hall algebra construction of the (whole) quantum groups was found; see \cite{Kap98, PX} for some earlier attempts on realizations of Kac-Moody algebras and quantum groups, and see \cite{T06, XX08} for constructions of derived Hall algebras. Bridgeland \cite{Br} in 2013 succeeded in using a Hall algebra of complexes to realize the quantum group $\U$. Actually Bridgleland's construction naturally produces the Drinfeld double $\widetilde \U$, a variant of $\U$ with the Cartan subalgebra doubled (with generators $K_i, K_i'$, for $i \in \I$). A  reduced version, which is the quotient of $\widetilde \U$ by the ideal generated by the central elements $K_i K_i'-1$, is then identified with $\U$.

Bridgeland's version of Hall algebras has found further generalizations and improvements which allow more flexibilities. M.~Gorsky \cite{Gor1} constructed semi-derived Hall algebras using $\Z/2$-graded complexes of an exact category. More recently, motivated by the works of Bridgeland and Gorsky, the first author and Peng \cite{LP} formulated the {\em semi-derived Ringel-Hall algebras} for hereditary abelian categories. There is another geometric approach toward Bridgeland's Hall algebra developed by Qin \cite{Qin16}; cf. Scherotzke-Sibilla \cite{SS16}.

\subsubsection{$\imath$Quantum groups}

 As a quantization of symmetric pairs $(\fg, \fg^\theta)$, the quantum symmetric pairs $(\U, \Ui)$ were formulated by Letzter \cite{Let99, Let02} (also cf. \cite{Ko14}) with Satake diagrams as inputs. The symmetric pairs are in bijection with the real forms of complex simple Lie algebras, according to \'E. Cartan. By definition, $\Ui =\Ui_\bvs$ is a coideal subalgebra of $\U$ depending on parameters $\bvs =(\vs_i)_{i\in \I}$ (subject to some compatibility conditions) and will be referred to as an $\imath$quantum group in this paper. As suggested in \cite{BW18a}, most of the fundamental constructions in the theory of quantum groups should admit generalizations in the setting of $\imath$quantum groups; see \cite{BW18a, BK19, BW18b} for generalizations of (quasi) R-matrix and canonical bases, and also see \cite{BKLW} (and \cite{Li19}) for a geometric realization and
  \cite{BSWW} for KLR type categorification of a class of (modified) $\Ui$.

Following terminologies in real group literature, we call an $\imath$quantum group {\em quasi-split} if the underlying Satake diagram does not contain any black node. In other words, the involution $\theta$ on $\fg$ is given by $\theta =\omega \circ \btau$, where $\omega$ is the Chevalley involution and $\btau$ is a diagram involution which is allowed to be $\Id$. If in addition $\btau=\Id$, $\Ui$ is called {\em split}. A quantum group can be viewed as a quasi-split $\imath$quantum group associated to the symmetric pair of diagonal type, and thus it is instructive to view $\imath$quantum groups as generalizations of quantum groups which may not admit a triangular decomposition.

\subsection{Goal}

This is the first of a series of papers in our program devoted to developing a new Hall algebra approach to $\imath$quantum groups, a vast generalization of Bridgeland's work. In this paper we initiate a Hall algebra construction associated with $\imath$quivers (aka quivers with involutions), and use it to realize the {\em universal} quasi-split $\imath$quantum groups $\tUi$ which we introduce in this paper; the usual $\imath$quantum groups $\Ui_\bvs$ are reproduced by central reductions of $\tUi$. Consequently, we construct PBW bases for $\Ui$ for the first time. In case of {\em $\imath$quivers of diagonal type}, our approach amounts to a reformulation of Bridgeland's construction.

It is our hope that this work will stimulate further interactions between communities on Hall algebras and on quantum symmetric pairs. On one hand, motivated by quantum symmetric pairs, we supply a new and natural family of  finite-dimensional algebras which affords rich representation theory. On the other hand, we bring in conceptual constructions and tools from quivers to shed new light on old constructions and to uncover new algebraic structures on $\imath$quantum groups.

This paper is arranged into 2 parts and an appendix.
The framework for the semi-derived Ringel-Hall algebras can be extended from hereditary abelian categories \cite{LP} to 1-Gorenstein algebras (or the {\em weakly $1$-Gorenstein} exact categories defined in \S\ref{subsection:Def of MRH}); see Appendix~\ref{app:A} by the first author. A reader might find it helpful to browse Appendix~\ref{app:A} first (which can be read independently from \cite{LP} if one accepts several statements without proofs), as general properties of semi-derived Ringel-Hall algebras formulated therein will be often used in Parts~\ref{part:1}--\ref{part:2}.

Part~\ref{part:1}, which consists of Sections~\ref{sec:quivers}--\ref{sec:HallPBW},  introduces the notion of $\imath$quiver algebras (which form a new class of 1-Gorenstein algebras) and formulates the $\imath$Hall algebras (which are twisted versions of semi-derived Ringel-Hall algebras for the $\imath$quiver algebras). Part~\ref{part:2}, which consists of Sections~\ref{sec:prelim2}--\ref{sec:generic}, establishes isomorphisms between $\imath$Hall algebras and $\imath$quantum groups and constructs new bases of these algebras.

\subsection{An overview of Part ~\ref{part:1} and Appendix~\ref{app:A}}

\subsubsection{$\imath$Quiver algebras}

Let $k$ be a field. Let $(Q, \btau)$ be an $\imath$quiver (that is, $\btau$ is an involutive automorphism of a quiver $Q$ respecting the arrows; we allow $\btau=\Id$). Associated to $(Q, \btau)$, we define an algebra
\begin{equation}   \label{eq:La}
\La =k Q \otimes_k R_2
\end{equation}
where $R_2$ is is the radical square zero of the path algebra of $\xymatrix{1 \ar@<0.5ex>[r]^{\varepsilon} & 1' \ar@<0.5ex>[l]^{\varepsilon'}}$. The involution $\btau$ induces an involution ${\btau}^{\sharp}$ on $\La$. We define the {\em $\imath$quiver algebra} of $(Q, \btau)$ to be
\begin{equation*}
\La^\imath    = \{x\in \Lambda\mid {\btau}^{\sharp}(x) =x\}.
\end{equation*}

\begin{customprop}{{\bf A}}  [Propositions~\ref{prop:invariant subalgebra}, \ref{prop:tensor algebra} and \ref{proposition of 1-Gorenstein}]
  \label{TA}
The $\imath$quiver algebra $\La^\imath$ can be described in terms of a bound quiver as $\Lambda^\imath \cong k \ov Q/\ov{I}$.
Moreover, $\La^\imath$ is a tensor algebra as well as a $1$-Gorenstein algebra.
\end{customprop}

We illustrate by 2 examples of rank two $\imath$quivers. The $\imath$quiver $Q=(\xymatrix{ 1\ar[r]^{\alpha} &2})$ with $\btau=\Id$ gives rise to an enhanced quiver $\ov Q$ with relations as follows:
\begin{center}\setlength{\unitlength}{0.7mm}
 \begin{picture}(50,13)(0,0)
\put(0,-2){$1$}
\put(4,0){\vector(1,0){14}}
\put(10,0){$^{\alpha}$}
\put(20,-2){$2$}
\color{purple}
\put(1,10){$\varepsilon_1$}
\put(21,10){$\varepsilon_2$}
\qbezier(-1,1)(-3,3)(-2,5.5)
\qbezier(-2,5.5)(1,9)(4,5.5)
\qbezier(4,5.5)(5,3)(3,1)
\put(3.1,1.4){\vector(-1,-1){0.3}}
\qbezier(19,1)(17,3)(18,5.5)
\qbezier(18,5.5)(21,9)(24,5.5)
\qbezier(24,5.5)(25,3)(23,1)
\put(23.1,1.4){\vector(-1,-1){0.3}}
\end{picture}
\vspace{0.2cm}
\end{center}
\begin{equation}  \label{eq:rk2split}
\varepsilon_1^2=0=\varepsilon_2^2, \quad \varepsilon_2 \alpha=\alpha\varepsilon_1.
\end{equation}
The $\imath$quiver algebra $\La^\imath$ associated to a split $\imath$quiver $(Q, \Id)$ is isomorphic to $k Q \otimes k [\varepsilon]/(\varepsilon^2)$, and its representation theory was studied by Ringel-Zhang \cite{RZ}. More general quivers with relations were also studied by Geiss-Leclerc-Schr\"{o}er \cite{GLS}. Our motivation of considering $\La^\imath$ is totally different.

On the other hand, the $\imath$quiver $Q=(\xymatrix{ 1\ar[r]^{\alpha} &2 & 3\ar[l]_{\beta} })$ with $\btau \neq \Id$ gives rise to the following enhanced quiver $\ov Q$ with relations:
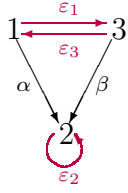
\begin{center}\setlength{\unitlength}{0.7mm}
 \begin{picture}(50,20)(0,-10)
\put(0,-2){$1$}
\put(20,-2){$3$}
\put(2,-11){$_{\alpha}$}
\put(17,-11){$_{\beta}$}
\put(2,-2){\vector(1,-2){8}}
\put(20,-2){\vector(-1,-2){8}}
\put(10,-22){$2$}
\color{purple}
\put(3,1){\vector(1,0){16}}
\put(19,-1){\vector(-1,0){16}}
\put(10,1){$^{\varepsilon_1}$}
\put(10,-4){$_{\varepsilon_3}$}
\put(10,-28){$_{\varepsilon_2}$}
\begin{picture}(50,23)(-10,19)
\color{purple}
\qbezier(-1,-1)(-3,-3)(-2,-5.5)
\qbezier(-2,-5.5)(1,-9)(4,-5.5)
\qbezier(4,-5.5)(5,-3)(3,-1)
\put(3.1,-1.4){\vector(-1,1){0.3}}
\end{picture}
\end{picture}
\vspace{1.4cm}
\end{center}
\begin{equation}  \label{eq:rk2nonsplit}
\varepsilon_1\varepsilon_3=0=\varepsilon_3\varepsilon_1,
\quad
 \varepsilon_2^2=0,
 \quad
 \varepsilon_2 \beta=\alpha\varepsilon_3,
 \quad
 \varepsilon_2 \alpha=\beta\varepsilon_1.
\end{equation}

Denote by $\underline{\Gproj}(\Lambda^{\imath})$ the stable category of the category of Gorenstein projective $\La^\imath$-modules, and denote by $D_{sg}(\mod(\Lambda^{\imath}))$ the singularity category. Let $\Sigma$ be the shift functor of the derived category $D^b(k Q)$. The involution $\btau$ induces a triangulated auto-equivalence $\widehat{\btau}$ of $D^b(\K Q)$.

\begin{customthm}{{\bf B}}  [Theorem \ref{thm:sigma}]
Let $(Q,\btau)$ be an $\imath$quiver. Then the following equivalences of categories hold:
\[
\underline{\Gproj}(\Lambda^{\imath}) \simeq D_{sg}(\mod(\Lambda^{\imath})) \simeq D^b(\K Q)/\Sigma \circ \widehat{\btau}.
\]
\end{customthm}
Note the first equivalence above is a well-known theorem of Buchweitz-Happel (see Theorem \ref{thm:Buchweitz-Happel}), but it is convenient to keep it together with the second equivalence.

\subsubsection{Generalities on semi-derived Ringel-Hall algebras}

The main constructions of semi-derived Ringel-Hall algebras for 1-Gorenstein algebras $A$ from Appendix \ref{app:A} by the first author form an extension of  \cite{LP} (which worked in the setting of Example~\ref{example 1}, generalizing \cite{Br, Gor1}).

Given an exact category $\ca$ with favorable properties as in (E-a)--(E-d) in Appendix~\ref{app:A}, its Ringel-Hall algebra $\ch(\ca)$ is defined. Then we define the semi-derived Ringel-Hall algebra $\utM(\ca)$ to be the localization $(\ch(\ca)/I) [\cs_\ca^{-1}]$, where $\cs_\ca$ is the multiplicatively closed set generated by iso-classes $[K]$ with $K$ of finite projective dimension
(see \eqref{eq:Sca}) and the ideal $I$ is generated by   all differences $[L]-[K\oplus M]$ if there is a short exact sequence
\begin{equation*}
 0 \longrightarrow K \longrightarrow L \longrightarrow M \longrightarrow 0
\end{equation*}
in $\ca$ with $K$ of finite projective dimension (see \eqref{eq:ideal}).
Given a 1-Gorenstein algebra $A$ over a finite field $k=\mathbb F_q$, the module category $\mod(A)$ satisfies the properties (E-a)--(E-d) and so we can define the semi-derived Ringel-Hall algebra  $\utM(A) := \utM(\mod(A))$.
On the other hand, since the subcategory $\Gproj(A)$ of $\mod(A)$ which consists of Gorenstein projective $A$-modules (see \S\ref{subsec:Gproj}) is a Frobenius category, the semi-derived Hall algebra $\cs\cd\ch(\Gproj(A))$ is also defined  \cite{Gor2}.

\begin{customthm}{{\bf C}}
   [Theorems~\ref{theorem isomorphism of algebras}, \ref{theorem basis of MRH 1-Gor}]
Let $A$ be a finite-dimensional $1$-Gorenstein algebra over $\K$. Then
\begin{enumerate}
\item
there is an algebra isomorphism $\cs\cd\ch(A) \cong \cs\cd\ch(\Gproj(A))$;
\item
$\cs\cd\ch(A)$ has a basis given by
$[M]\diamond \E_\alpha$, where $[M]\in\Iso \big(D_{sg}(\mod(A))\big)$, and $\alpha\in K_0(\mod(A))$.
\end{enumerate}
\end{customthm}
In comparison to Gorsky's version of semi-derived Hall algebra, the semi-derived Ringel-Hall algebra $\utM(\ca)$ in our sense is easier for computational purpose. In addition, we shall see that the generators for $\Ui$ have simple interpretations in a twisted version of $\utMH$.

\begin{customthm}{\bf D}    [Theorem~\ref{theorem:derived equivalence of MRH}]
Let $A$ be a $1$-Gorenstein algebra with a tilting module $T$. If $\Gamma=\End_{A}(T)^{op}$ is a $1$-Gorenstein algebra, then there is an algebra isomorphism $\cs\cd\ch(A) \cong  \cs\cd\ch(\Gamma).$
\end{customthm}

\subsubsection{Bases for $\imath$Hall algebras}

The algebra $\La^\imath$ admits a non-negative grading $\La^\imath =\La^\imath_0 \oplus \La^\imath_1$, where $\La^\imath_0 = k Q$. Thus, $\La^\imath$ naturally has the path algebra $k Q$ as its subalgebra and quotient algebra.
This allows us to formulate conceptually an Euler form. This leads to the twisted semi-derived Ringel-Hall algebra $\tMH$, which is $\utMH$ with a twisted Hall multiplication; we shall refer to this as the Hall algebra associated to the $\imath$quiver (or $\imath$Hall algebra for short).

 The quotient morphism $\La^\imath \rightarrow kQ$ induces a pullback functor
$\iota:\mod(\K Q)\longrightarrow\mod(\Lambda^{\imath})$ in \eqref{eqn:rigt adjoint}. Hence we have a natural inclusion $\Iso(\mod(\K Q)) \subseteq \Iso(\mod(\Lambda^{\imath}))$.
In the setting of $\imath$quiver algebra $\La^\imath$, the basis in Theorem~{\bf C} takes a more concrete form, thanks to Theorem~{\bf B}.

\begin{customthm}{\bf E} [Theorem~\ref{thm:utMHbasis}]
The $\imath$Hall algebra $\tMH$ has a (Hall) basis given by
\begin{equation*}
\big\{ [X]\diamond \E_\alpha ~\big |~ [X]\in\Iso(\mod(\K Q)), \alpha\in K_0(\mod(\K Q)) \big\}.
\end{equation*}
\end{customthm}
The span of $\E_\alpha$, for $\alpha\in K_0(\mod(\K Q))$, is a subalgebra of $\tMH$ called a twisted quantum torus and denoted by $\tTL$. It follows by Lemma~\ref{lem:torus} that $\tTL$ is a Laurent polynomial algebra with generators $\E_i$, for $i\in \I$.

Assume now $(Q, \btau)$ is a Dynkin $\imath$quiver.
The indecomposable modules over the path algebra $kQ$ are parameterized by the positive roots for $\fg$, and they are used in \cite{Rin3} to construct a PBW basis for the Hall algebra $\ch(kQ)$; cf. \cite{DDPW}.
As $kQ$ is a quotient algebra of $\La^\imath$, we can regard the indecomposable $kQ$-modules as modules over $\La^\imath$ via pullback.

The algebra $\utMH$ is endowed with a filtered algebra structure by a partial order induced by degeneration; cf. \cite{Rie86}. We relate the associated graded $\cs\cd\ch^{\gr}(\Lambda^{\imath})$ to the Hall algebra $\ch(\K Q)$  of the path algebra $kQ$. This allows us to transfer a monomial basis (and respectively, PBW basis) for $\ch(\K Q)$ to a monomial basis (and respectively, PBW basis) for $\utMH$ (or $\tMH$) over its (twisted) quantum torus.

\begin{customthm}{{\bf F}}
   [Monomial basis Theorem~\ref{thm:monomial}, PBW basis Theorem~\ref{thm:HallPBW}]
Let $(Q,\btau)$ be a Dynkin $\imath$quiver.
There exist monomial bases as well as PBW bases for the $\imath$Hall algebra $\tMH$ as a right $\tTL$-module.
\end{customthm}

The algebra $\tMH$ contains various central elements, cf. Proposition \ref{Prop:centralMH}.
As in \cite{Br}, we shall define in Definition~\ref{def:reducedHall} a reduced version of $\imath$Hall algebra, $\rMH$, to be the quotient algebra of $\tMH$ by an ideal generated by certain central elements.

\subsection{An overview of Part~\ref{part:2}}

Recall $\tU$ is a version of quantum group $\U$ with an enlarged Cartan subalgebra. In this paper we introduce a universal  $\imath$quantum group $\tUi$ whose Cartan subalgebra is generated by $\tk_i$, for $i\in \I$. One readily checks that $\tUi$ is a right coideal subalgebra of $\tU$, and $(\widetilde{\bU}, \widetilde{\bU}^\imath)$ forms a quantum symmetric pair. The algebra $\tUi$ contains central elements $\tk_i$ for $i$ with $i =\btau i$ and $\tk_i \tk_{\btau i}$ with $i \neq \btau i$, and $\Ui=\Ui_\bvs$ is recovered by a central reduction from $\tUi$. More precisely, Proposition~\ref{prop:QSP12} states that:

{\em
The algebra $\Ui$ is isomorphic to the quotient of $\tUi$ by the ideal generated by
$\tk_i - \vs_i \; (\text{for } i =\btau i)$ and $\tk_i \tk_{\btau i} - \vs_i\vs_{\btau i}  \;(\text{for } i \neq \btau i).$
}

In this way, these central elements in $\tUi$ provide a conceptual way of explaining the parameters for the $\imath$quantum groups, and $(\tU, \tUi)$ can be regarded as a {\em universal family} of quantum symmetric pairs.
Set
\[
\sqq=\sqrt{q}.
\]
Denote by $\tUi_{|v={\sqq}}$ the specialization at $v=\sqq$ of the algebra $\tUi$, and so on.

From now on, we mostly restrict ourselves to Dynkin $\imath$quivers (for which the Serre type relations in the $\imath$Hall algebras can be verified readily).  A list of quasi-split symmetric pairs $(\fg, \fg^\theta)$ which are covered by our constructions can be found in Table~\ref{table:values}.

\begin{customthm}{{\bf G}}  [Theorem~\ref{theorem isomorphism of Ui and MRH}]
Let $(Q,\btau)$ be a Dynkin $\imath$quiver. Then we have the following isomorphisms of $\Q(\sqq)$-algebras:
\begin{align*}
\widetilde{\psi}: \tUi_{|v={\sqq}}\stackrel{\simeq}{\longrightarrow} \tMH,\qquad
\psi: \Ui_{|v={\sqq}}\stackrel{\simeq}{\longrightarrow} \rMH.
\end{align*}
\end{customthm}
The twisted quantum torus is mapped by $\widetilde\psi$ to the Cartan subalgebra of $\tUi$.
See Theorem~\ref{theorem isomorphism of Ui and MRH} for the explicit formulas for $\widetilde{\psi}$, which sends the  generators of $\tUi$ to the simple modules up to scalar multiples. (This is one of the advantages of using the formalism of semi-derived Ringel-Hall algebras over the semi-derived Hall algebras defined in \cite{Gor2}.) The proof that $\widetilde\psi$ is a homomorphism is reduced to the computations for the rank 2 $\imath$quivers. 
It is instructive to see how the inhomogeneous Serre relations for $\tUi$ (or $\Ui$) emerge from the $\imath$Hall algebra computation; cf. Propositions~\ref{prop:A2}, \ref{prop:iA3} and \ref{prop:cartan}.

Associated to an acyclic quiver $Q$, we formulate an $\imath$quiver of diagonal type in Example~\ref{example diagonal}, whose $\imath$quiver algebra is $\La$; cf. \eqref{eq:La}. In this setting, a variant of Theorem~{\bf G} provides the following theorem, which can be viewed as a version of Bridgeland's construction \cite{Br} thanks to Theorem~{\bf C}(1); compare \cite{Gor1}.

\begin{customthm}{\bf H}
  [Bridgeland's theorem reformulated, Theorem~\ref{thm:bridgeland2} and Proposition~\ref{prop:bridgeland}]
There exist  injective algebra homomorphisms $\tU_{|v={\sqq}}  \longrightarrow \tM(\La)$ and $\bU_{|v={\sqq}} \rightarrow \rM(\Lambda)$.
\end{customthm}
For $Q$ of Dynkin type, the above homomorphism is an isomorphism.

For Dynkin $\imath$quivers, we show that the structure constants of the $\imath$Hall algebras are Laurent polynomials in $\sqq$. This allows us to define the generic $\imath$Hall algebras $\tMHg$ and $\rMHg$ over the field $\Q(v)$.

\begin{customthm}{\bf I}  [Theorem~\ref{generic U MRH}]
Let $(Q, \btau)$ be a Dynkin $\imath$quiver. Then we have a $\Q(\bv)$-algebra isomorphism
\begin{align*}
\widetilde{\psi}: \tUi\stackrel{\simeq}{\longrightarrow}& \tMHg,
\qquad \quad
\psi: \Ui\stackrel{\simeq}{\longrightarrow} \rMHg.
 \end{align*}
\end{customthm}
Via the isomorphism $\widetilde \psi$ in Theorem~{\bf I}, a monomial basis for $\Ui$ in \cite{Let02} can be matched with a monomial basis for $\rMHg$ (which is a generic version of a corresponding monomial basis for $\rMH$ in Theorem~{\bf F}). The geometric degeneration partial order provides a natural interpretation for the fundamental filtration on $\Ui$ in \cite{Let02} via the isomorphism $\widetilde \psi$.

\subsection{Future works}

In a second paper of this series \cite{LW19b}, we formulate BGP-type reflection functors for $\imath$Hall algebras, which are shown to satisfy the braid group relations for the restricted Weyl group of the symmetric pair $(\fg, \fg^\theta)$. Via the isomorphism $\widetilde{\psi}$ this leads to automorphisms on the universal $\imath$quantum group $\tUi$ which satisfy braid group relations.

In another sequel \cite{LW19c} we shall study Nakajima-Keller-Scherotzke categories for $\imath$quivers and use it to provide a geometric realization of $\imath$quantum groups; cf. \cite{Qin16,SS16}. It will be interesting to examine if this geometric approach provides links to Y.~ Li's $\imath$quiver variety \cite{Li19} and to (dual) $\imath$canonical bases \cite{BW18b}.

There are several ways to extend the connections between $\imath$Hall algebras and $\imath$quantum groups to more general settings as explained below, and we plan to return to these in future works.

A major next step in our program is to formulate the Hall algebras for general (i.e., beyond quasi-split) $\imath$quantum groups. The $\imath$quiver algebras in this paper can be viewed as a first foundation for such an extension. This might eventually lead to general geometric framework for quantum symmetric pairs and the (dual) $\imath$canonical bases.

We expect (see Conjecture~\ref{conj:mono}) that there is an injective homomorphism from the quasi-split $\imath$quantum groups of symmetric Kac-Moody type to the Hall algebras associated to the acyclic $\imath$quivers introduced in Part~\ref{part:1} of this paper. As $\imath$quantum groups have sophisticated Serre type relations, it takes serious work to complete this.
We also expect that our work can be extended to cover $\imath$quantum groups of non-ADE Dynkin types and symmetrizable  Kac-Moody types using {\em valued $\imath$quivers}. In all these settings BGP-type reflection functors and braid groups actions shall be available.

\subsection{Organization}

The materials for $\imath$quiver algebras and $\imath$Hall algebras in Sections~\ref{sec:quivers}--\ref{sec:Hall}  are valid for arbitrary acyclic $\imath$quivers. The quivers are assumed to be Dynkin when we develop PBW bases and isomorphism theorems between $\imath$Hall algebras and $\imath$quantum groups. The field $k$ in Sections~\ref{sec:quivers}--\ref{sec:iLambda} on $\imath$quiver algebras can be arbitrary, while we take $k=\mathbb F_q$ when dealing with Hall algebras in Sections~\ref{sec:Hall}--\ref{sec:generic}.

The paper is organized as follows.
In  Appendix~\ref{app:A} by the first author, the notion of semi-derived Ringel-Hall algebras is formulated in a framework of weakly 1-Gorenstein exact categories,  including the module categories of 1-Gorenstein algebras. The basic properties of the Ringel-Hall algebras are established, including the tilting invariance and an isomorphism with Gorsky's semi-derived Hall algebras for 1-Gorenstein algebras.

In Section~\ref{sec:quivers}, the basic notion of $\imath$quiver algebras $\La$ and $\La^\imath$ is introduced. Then we provide bound quiver descriptions for these algebras, show that $\La^\imath$ is a tensor algebra, and develop the representation theory of modulated graphs of $\imath$quivers.

We study the homological properties of the $\imath$quiver algebra $\La^\imath$ in Section~\ref{sec:iLambda}, by first establishing that it is 1-Gorenstein. The Gorenstein projective modules, indecomposable projective modules and also modules with finite projective dimensions are completely described for $\Lambda^\imath$. We set up the connections among singularity category of $\La^\imath$, the stable category of Gorenstein projective $\La^\imath$-modules, and an orbit triangulated category of $D^b(kQ)$.

We apply in Section~\ref{sec:Hall} the general machinery in Appendix \ref{app:A} to formulate the semi-derived Ringel-Hall algebra $\utMH$ for the 1-Gorenstein algebra $\La^\imath$ and a twisted version $\tMH$. We give a rather explicit Hall basis of $\utMH$. We show that the embedding of a $\imath$subquiver in an $\imath$quiver leads to an inclusion of $\imath$Hall algebras.

In Section~\ref{sec:HallPBW}, for Dynkin $\imath$quivers, we construct monomial bases and PBW bases for $\tMH$ over its quantum torus. This is based on a filtered algebra structure on $\tMH$ which allows us to relate to the usual Hall algebra $\tH(Q)$.

\medskip

In Section~\ref{sec:prelim2} (the first section of Part~\ref{part:2}), we review the quantum symmetric pair $(\U, \Ui)$ and introduce a new quantum symmetric pair $(\tU, \tUi)$. We show that $\Ui$ for various parameters $\bvs$ are central reductions of $\tUi$.

In Section~\ref{sec:HallQSP}, we establish the algebra isomorphism $\widetilde{\psi}: \tUi_{|v={\sqq}}\stackrel{\simeq}{\rightarrow} \tMH$ and a reduced variant. The proof that $\widetilde{\psi}$ is a homomorphism is reduced to rank 2 $\imath$quiver computations. A monomial basis of $\tMH$ is used to show that $\widetilde{\psi}$ is an isomorphism.

In case of $\imath$quivers of diagonal type, the $\imath$quantum groups become the usual quantum groups $\widetilde{\U}$ and $\U$. The $\imath$Hall algebra in this case provides in Section~\ref{sec:Bridgeland} a reformulation of Bridgeland's Hall algebra construction.

In Section~\ref{sec:generic},  we show the structure constants of the $\imath$Hall algebras for Dynkin $\imath$quivers are Laurent polynomials in $\sqq$. This allows us to define the generic $\imath$Hall algebras $\tMHg$ and $\rMHg$; we finally prove that $\tMHg \cong \tUi$ and $\rMHg \cong \Ui$.

\subsection{Acknowledgments}
ML thanks University of Virginia for hospitality during his visit when this project was initiated. We thank Institute of Mathematics at Academia Sinica (Taipei) for hospitality and support which has greatly facilitated the completion of this work. WW is partially supported by NSF grant DMS-1702254 and DMS-2001351. We thank the referee for his/her careful reading and helpful suggestions.

\vspace{2mm}
\noindent {\bf Notes added.}
Conjecture \ref{conj:mono} (in a generalized form) has been proved in \cite{LW20}, where the $\imath$quiver algebras are defined for arbitrary (not necessarily acyclic) $\imath$quivers. %
\subsection{Notations}

We list the notations which are often used throughout the paper.
\smallskip

$\triangleright$ $\N,\Z,\Q$, $\C$ -- sets of nonnegative integers, integers, rational  and complex numbers.

\medskip

For a finite-dimensional $k$-algebra $A$, we denote

$\triangleright$ $\mod(A)$ -- category of finite-dimensional left $A$-modules,


$\triangleright$ $D=\Hom_\K(-,\K)$ -- the standard duality,

$\triangleright$ $D^b(A)$ -- bounded derived category of finite-dimensional $A$-modules,

$\triangleright$ $\Sigma$ -- the shift functor.

\medskip

Let $Q=(Q_0,Q_1)$ be an acyclic quiver.
For $i \in Q_0$, we denote

$\triangleright$ $e_i$ -- the primitive idempotent of $\K Q$,

$\triangleright$ $S_i$ -- the simple module supported at $i$,

$\triangleright$ $P_i$ -- the projective cover of $S_i$,

$\triangleright$ $I_i$ --  the injective hull of $S_i$,

$\triangleright$ $\rep_\K(Q)$ -- category of representations of $Q$ over $\K$, identified with $\mod(\K Q)$.

\medskip

For an additive category $\ca$ and $M\in \ca$, we denote

$\triangleright$ $\add M$ -- subcategory of $\ca$ whose objects are the direct summands of finite direct sums of copies of $M$,

$\triangleright$ $\Ind (\ca)$ --  set of the isoclasses of indecomposable objects in $\ca$,

$\triangleright$ $\Iso(\ca)$ -- set of the isoclasses of objects in $\ca$.

\medskip

For an exact category $\ca$, we denote

$\triangleright$ $K_0(\ca)$ --  Grothendieck group of $\ca$,

$\triangleright$ $\widehat{A}$ -- the class in $K_0(\ca)$ of $A \in \ca$.

\medskip

For various Hall algebras, we denote

$\triangleright$ $\ch (Q)$ -- Ringel-Hall algebra of the path algebra $kQ$,

$\triangleright$ $\tH (Q)$ -- twisted Ringel-Hall algebra of the path algebra $kQ$,

\smallskip
Associated to the $\imath$quiver $(Q,\btau)$ (aka quiver with involution), we denote

$\triangleright$ $\La^\imath$ -- $\imath$quiver algebra, 

$\triangleright$ $\utMH$ -- semi-derived Ringel-Hall algebra for $\La^\imath$ (i.e., for $\mod (\La^\imath)$),

$\triangleright$ $\T (\La^\imath)$ -- quantum torus, 

$\triangleright$ $\tMH$ -- $\imath$Hall algebra 
 (=twisited semi-derived Ringel-Hall algebra for $\La^\imath$),

$\triangleright$ $\tTL$ -- twisted quantum torus, 

$\triangleright$ $\rMH$ -- reduced $\imath$Hall algebra 

$\triangleright$ $\tMHg$ -- generic $\imath$Hall algebra 

$\triangleright$ $\rMHg$ -- generic reduced $\imath$Hall algebra 

\medskip
For quantum algebras, we denote

$\triangleright$ $\U$ -- quantum group,

$\triangleright$ $\tU$ -- Drinfeld double (a variant of $\U$ with doubled Cartan subalgebra),

$\triangleright$ $\Ui=\Ui_{\bvs}$ -- a right coideal subalgebra of $\U$, depending on parameter $\bvs \in (\Q(v)^\times)^\I$,

$\triangleright$ $(\U, \Ui)$ -- quantum symmetric pair,

$\triangleright$ $(\tU, \tUi)$ -- a universal quantum symmetric pair, such that $\Ui$ is obtained from $\tUi$ by a central reduction.

\part{$\imath$Quiver Algebras and $\imath$Hall Algebras}
  \label{part:1}

\section{Finite-dimensional algebras arising from $\imath$quivers}
  \label{sec:quivers}

In this section, starting with an acyclic $\imath$quiver (i.e., a quiver with involution), we construct a finite-dimensional algebra $\Lambda^\imath$.  We describe the bound quiver for $\Lambda^\imath$ and show that $\La^\imath$ is a tensor algebra. We develop the representation theory of modulated graphs for $\imath$quivers.

\subsection{The $\imath$quivers and doubles}

Let $\K$ be a field.
Let $Q=(Q_0,Q_1)$ be an acyclic quiver. Throughout the paper, we shall identify
\[
\I=Q_0.
\]
An {\em involution} of $Q$ is defined to be an automorphism $\btau$ of the quiver $Q$ such that $\btau^2=\Id$. In particular, we allow the {\em trivial} involution $\Id:Q\rightarrow Q$. An involution $\btau$ of $Q$ induces an involution of the path algebra $\K Q$, again denoted by $\btau$.
A quiver together with a specified involution $\btau$, $(Q, \btau)$, will be called an {\em $\imath$quiver}.

Let $Z_m$ be the quiver with $m$ vertices and $m$ arrows which forms an oriented cycle. The vertex set of $Z_m$ is $\{0,1,\dots,m-1\}$. Let $R_m$ be the radical square zero selfinjective Nakayama algebra of $Z_m$, i.e.,  $R_m:=kZ_m/J$, where $J$ denotes the ideal of $kZ_m$ generated by all paths of length two. In particular,

$\triangleright$ $R_1$ is isomorphic to the truncated polynomial algebra $\K[\varepsilon]/(\varepsilon^2)$;

$\triangleright$ $R_2$ is the radical square zero of the path algebra of $\xymatrix{1 \ar@<0.5ex>[r]^{\varepsilon} & 1' \ar@<0.5ex>[l]^{\varepsilon'}}$, i.e., $\varepsilon' \varepsilon =0 =\varepsilon\varepsilon '$.

Let $\cc_{\Z/m}(\mod(kQ))$ be the category of the $\Z/m$-graded complexes over $\mod(kQ)$, for any $m\geq 1$, see \cite{Br, CD, LP}. The following lemma is well known.

\begin{lemma}  \label{lem:Rm}
We have $\cc_{\Z/m}(\mod(kQ))\cong \mod(kQ\otimes_k R_m)$, for any $m \geq 1$.
\end{lemma}

Define a $\K$-algebra
\begin{equation}
  \label{eq:La}
\Lambda=\K Q\otimes_\K R_2.
\end{equation}

Let $Q^{\sharp}$ be the quiver such that
\begin{itemize}
\item the vertex set of $Q^{\sharp}$ consists of 2 copies of the vertex set $Q_0$, $\{i,i'\mid i\in Q_0\}$;
\item the arrow set of $Q^{\sharp}$ is
\[
\{\alpha: i\rightarrow j,\alpha': i'\rightarrow j'\mid(\alpha:i\rightarrow j)\in Q_1\}\cup\{ \varepsilon_i: i\rightarrow i' ,\varepsilon'_i: i'\rightarrow i\mid i\in Q_0 \}.
\]
\end{itemize}
We call $Q^{\sharp}$ the {\em double framed quiver} associated to the quiver $Q$.

The involution $\btau$ of a quiver $Q$ induces an involution ${\btau}^{\sharp}$ of $Q^{\sharp}$ defined by
\begin{itemize}
\item ${\btau}^{\sharp}(i)=(\btau i)'$, ${\btau}^{\sharp}(i') =\btau i$ for any $i\in Q_0$;
\item ${\btau}^{\sharp}(\varepsilon_i)= \varepsilon_{\btau i}'$, ${\btau}^{\sharp}(\varepsilon_i')= \varepsilon_{\btau i}$ for any $i\in Q_0$;
\item ${\btau}^{\sharp}(\alpha)= (\btau\alpha)'$, ${\btau}^{\sharp}(\alpha')=\btau\alpha$ for any $\alpha\in Q_1$.
\end{itemize}
So starting  from the $\imath$quiver $(Q, \btau)$ we have constructed a new $\imath$quiver $(Q^{\sharp}, {\btau}^{\sharp})$.

\subsection{A bound quiver description of $\La$}

The algebra $\Lambda$ can be described in terms of a quiver with relations. Let $I^{\sharp}$ be the admissible ideal of $\K Q^{\sharp}$ generated by
\begin{itemize}
\item
(Nilpotent relations) $\varepsilon_i \varepsilon_i'$, $\varepsilon_i'\varepsilon_i$ for any $i\in Q_0$;
\item
(Commutative relations) $\varepsilon_j' \alpha' -\alpha\varepsilon_i'$, $\varepsilon_j \alpha -\alpha'\varepsilon_i$ for any $(\alpha:i\rightarrow j)\in Q_1$.
\end{itemize}
Then the algebra $\La$ can be realized as
\begin{equation}  \label{eq:La=QI}
\Lambda\cong \K Q^{\sharp} \big/ I^{\sharp}.
\end{equation}
Let $Q$ (respectively, $Q'$) be the full subquiver of $Q^{\sharp}$ formed by all vertices $i$ (respectively, $i'$) for $i\in Q_0$. Then $Q\sqcup Q'$ is a subquiver of $Q^{\sharp}$.

\begin{example}
   \label{example 2-cyclic complex of An}
(a)  Let $Q=(\xymatrix{ 1\ar[r]^{\alpha} &2})$. Then its double framed quiver $Q^{\sharp}$ is
\begin{center}\setlength{\unitlength}{0.7mm}
\vspace{-2cm}
\begin{equation*}
\begin{picture}(100,40)(0,20)
\put(49,8){\small $1'$}
\put(50,31){\small $1$}
\put(72,8){\small $2'$}
\put(72,31){\small $2$}

\put(53,10){\vector(1,0){18.5}}
\put(53,32.5){\vector(1,0){18.5}}

\put(60,12.5){$_{\alpha'}$}
\put(60,35){$_\alpha$}
\color{purple}
\put(50,13){\vector(0,1){17}}
\put(52,29.5){\vector(0,-1){17}}
\put(72,13){\vector(0,1){17}}
\put(74,29.5){\vector(0,-1){17}}

\put(45,20){\small $\varepsilon_1'$}
\put(53,20){\small $\varepsilon_1$}
\put(67,20){\small $\varepsilon_2'$}
\put(75,20){\small $\varepsilon_2$}
\end{picture}
\end{equation*}
\vspace{0.2cm}
\end{center}
and $I^{\sharp}$ is generated by all possible quadratic relations
\begin{eqnarray*}
&&\varepsilon_1\varepsilon_1', \,\,\,\,\varepsilon_1'\varepsilon_1, \,\,\,\,\varepsilon_2'\varepsilon_2, \,\,\,\,\varepsilon_2\varepsilon_2',\,\,\,\, \alpha' \varepsilon_1 -\varepsilon_2 \alpha,\,\,\,\, \alpha \varepsilon_1'- \varepsilon_2' \alpha'.
\end{eqnarray*}

(b) Let $Q=(\xymatrix{ 1\ar[r]^{\alpha} &2 & 3\ar[l]_{\beta} })$. Then its double framed quiver $Q^{\sharp}$ is
\begin{center}\setlength{\unitlength}{0.7mm}
\vspace{-2cm}
\begin{equation*}
\begin{picture}(100,40)(0,20)
\put(49,8){\small $1'$}
\put(50,31){\small $1$}
\put(72,8){\small $2'$}
\put(72,31){\small $2$}
\put(95,8){\small $3'$}
\put(95,31){\small $3$}

\put(53,10){\vector(1,0){18.5}}
\put(53,32.5){\vector(1,0){18.5}}

\put(94,10){\vector(-1,0){18}}
\put(94,32.5){\vector(-1,0){18.5}}

\put(60,12.5){$_{\alpha'}$}
\put(60,35){$_\alpha$}
\put(82,12.5){$_{\beta'}$}
\put(82,35){$_\beta$}
\color{purple}
\put(50,13){\vector(0,1){17}}
\put(52,29.5){\vector(0,-1){17}}
\put(72,13){\vector(0,1){17}}
\put(74,29.5){\vector(0,-1){17}}
\put(95,13){\vector(0,1){17}}
\put(97,29.5){\vector(0,-1){17}}

\put(45,20){\small $\varepsilon_1'$}
\put(53,20){\small $\varepsilon_1$}
\put(67,20){\small $\varepsilon_2'$}
\put(75,20){\small $\varepsilon_2$}
\put(90,20){\small $\varepsilon_3'$}
\put(98,20){\small $\varepsilon_3$}
\end{picture}
\end{equation*}
\vspace{0.2cm}
\end{center}
and $I^{\sharp}$ is generated by all possible quadratic relations
\begin{eqnarray*}
&&\varepsilon_i'\varepsilon_i, \,\,\,\,\varepsilon_i\varepsilon_i', \quad \forall 1\leq i\leq 3,\\
&&\alpha' \varepsilon_1 -\varepsilon_2 \alpha,\,\,\,\, \alpha\varepsilon_1'- \varepsilon_2' \alpha',\,\,\,\,\beta \varepsilon_3' -\varepsilon_2' \beta',\,\,\,\, \beta' \varepsilon_3- \varepsilon_2 \beta.
\end{eqnarray*}

(c) Let $Q$ be the quiver such that $Q_0= \{1,2\}$, and $Q_1=\emptyset$. Then $\Lambda\cong R_2\times R_2$.
\end{example}

\begin{example}
\label{example diagonal}
Let $Q=(Q_0,Q_1)$ be an acyclic quiver, $Q^{\sharp}$ be its double framed quiver. Let $Q^{\dbl} =Q\sqcup  Q^{\diamond}$,  where $Q^{\diamond}$ is an identical copy of $Q$ with a vertex set $\{i^{\diamond} \mid i\in Q_0\}$ and an arrow set $\{ \alpha^{\diamond} \mid \alpha \in Q_1\}$.
Let $\Lambda=kQ\otimes_k R_2$, $\Lambda^{\diamond} =k Q^{\diamond} \otimes_k R_2$, and $\Lambda^{\dbl}:=kQ^{\dbl}\otimes R_2$. Then the double framed quiver $(Q^{\dbl})^{\sharp}$ of $Q^{\dbl}$ is $Q^{\sharp}\sqcup (Q^{\sharp})^{\diamond}$, and
$\Lambda^{\dbl}=\Lambda\times \Lambda^{\diamond} \cong \Lambda\times \Lambda$.
\end{example}
\subsection{The $\imath$quiver algebra $\La^\imath$}
    \label{subsec:iLa}

The following can be verified from the definitions.
\begin{lemma}
  \label{lem:tildetheta}
The action of ${\btau}^{\sharp}$ preserves $I^{\sharp}$. Hence ${\btau}^{\sharp}$ induces an involution, again denoted by ${\btau}^{\sharp}$, of the algebra $\Lambda$.
\end{lemma}

\begin{definition}
The fixed point subalgebra of $\Lambda$ under ${\btau}^{\sharp}$,
\begin{equation}
   \label{eq:iLa}
\Lambda^\imath
= \{x\in \Lambda\mid {\btau}^{\sharp}(x) =x\},
\end{equation}
 is called the {\rm $\imath$quiver algebra} of $(Q, \btau)$.
 \end{definition}

\begin{remark}
\label{rem:action Lambda}
Since $\Lambda=\K Q\otimes_k R_2$, it has a basis $\{p\otimes e_0,p\otimes e_1,p\otimes \varepsilon,p\otimes \varepsilon' \text{ for all paths } p \text{ of } Q\}$.
Then the action of ${\btau}^{\sharp}$ on $\Lambda$ is given by
\begin{align*}
{\btau}^{\sharp}(p\otimes e_0)=\btau p\otimes e_1,& \qquad {\btau}^{\sharp}(p\otimes e_1)=\btau p\otimes e_0\\
{\btau}^{\sharp}(p\otimes \varepsilon)=\btau p\otimes \varepsilon',&\qquad {\btau}^{\sharp}(p\otimes \varepsilon')=\btau p\otimes \varepsilon.
\end{align*}
\end{remark}

We describe $\Lambda^\imath$ in terms of a quiver $\ov Q$ and its admissible ideal $\ov{I}$ as follows.

\begin{proposition}
  \label{prop:invariant subalgebra}
We have $\Lambda^\imath \cong \K \ov{Q} / \ov{I}$,
where
\begin{itemize}
\item[(i)] $\ov{Q}$ is constructed from $Q$ by adding a loop $\varepsilon_i$ at the vertex $i\in Q_0$ if $\btau i=i$, and adding an arrow $\varepsilon_i: i\rightarrow \btau i$ for each $i\in Q_0$ if $\btau i\neq i$;
\item[(ii)] $\ov{I}$ is generated by
\begin{itemize}
\item[(1)] (Nilpotent relations) $\varepsilon_{i}\varepsilon_{\btau i}$ for any $i\in\I$;
\item[(2)] (Commutative relations) $\varepsilon_i\alpha-\btau(\alpha)\varepsilon_j$ for any arrow $\alpha:j\rightarrow i$ in $Q_1$.
\end{itemize}
\end{itemize}
\end{proposition}
The quiver $\ov Q$ or $(\ov Q, \ov I)$ is called an {\em enriched (bound) quiver}.

\begin{proof}
The cyclic group $\langle {\btau}^{\sharp} \rangle =\{\Id, {\btau}^{\sharp} \}$ of order 2 acts on $Q^{\sharp}$ freely. So we can define the orbit quiver $Q^{\sharp}/\langle {\btau}^{\sharp} \rangle$.
Then $Q^{\sharp}/\langle {\btau}^{\sharp} \rangle$ coincides with $\ov{Q}$ if by abuse of notation we identify
$i$, $\varepsilon_i$, and $\alpha$ as the orbits of $i \in Q_0$, $\varepsilon_i$, and $\alpha\in Q_1$, respectively. We identify $Q^{\sharp}/\langle {\btau}^{\sharp} \rangle$ with $\ov{Q}$ below. This induces a Galois covering of quivers $\pi:Q^{\sharp}\rightarrow\ov{Q}$.

Note that $\ov{I}=\pi(I^{\sharp})$. Then $\pi$ induces a Galois covering of quivers with relations $(Q^{\sharp}, I^{\sharp})\rightarrow (\ov{Q},\pi(I^{\sharp}))$ \`a la \cite{BG}, also denoted by $\pi$.
As $\Lambda^{\imath}$ is the fixed point subalgebra of $\Lambda$, $\pi$ induces a homomorphism $\phi:\Lambda^{\imath}\rightarrow \K\ov{Q}/\ov{I}$. In fact, $\Lambda^{\imath}$ is spanned by $\{p+\btau^\sharp p\mid p\text{ is a path in }Q^{\sharp} \}$. Note that for any path $p$ in $Q^{\sharp}$, $p\in (I^{\sharp})$ if and only if $\btau^\sharp p\in (I^{\sharp})$, and in this case $\pi(p)\in (\ov{I})$. Then we have $\phi(p+\btau^\sharp p)= \pi(p)$ for any path $p$ in $Q^{\sharp}$.

On the other hand,  there exists a natural homomorphism $\varphi:\K\ov{Q}\rightarrow \Lambda^{\imath}$ induced by mapping any arrow
$\alpha$ in $Q$ (viewed as a subquiver of $\ov{Q}$) to $\alpha+ (\btau \alpha)'$ for each arrow $\alpha\in Q_1\subseteq Q^{\sharp}_1$, and mapping $\varepsilon_i$ to $\varepsilon_i+\varepsilon_{\btau i}'$ for any arrow $\varepsilon_i:i\rightarrow \btau i$. It is routine to prove that $\varphi(\ov{I})=0$ in $\Lambda^\imath$. So $\varphi$ induces a homomorphism $\varphi: \K\ov{Q}/\ov{I}\rightarrow \Lambda^{\imath}$. One checks that $\varphi \phi=\Id$ and $\phi\varphi=\Id$, and therefore $ \Lambda^{\imath}\cong \K\ov{Q}/\ov{I}$.
\end{proof}

We have the following corollary to Proposition \ref{prop:invariant subalgebra}.
The algebra $\K Q\otimes R_1$ was considered  in \cite{RZ} with a different motivation.

\begin{corollary}\label{cor: 1-peoriodic}
If the involution $\btau$ of $Q$ is trivial (i.e., $\btau=\Id$), then $\Lambda^\imath \cong \K Q\otimes R_1$.
\end{corollary}

\begin{example}\label{example 2}
(a) We continue Example \ref{example 2-cyclic complex of An}(a), with $\btau=\Id$. Then $\Lambda^\imath$ is isomorphic to the path algebra of the following
quiver $\ov{Q}$ with relations:
\begin{center}\setlength{\unitlength}{0.7mm}
 \begin{picture}(50,13)(0,0)
\put(0,-2){$1$}
\put(4,0){\vector(1,0){14}}
\put(10,0){$^{\alpha}$}
\put(20,-2){$2$}
\color{purple}
\put(0,9){\small $\varepsilon_1$}
\put(20,9){\small $\varepsilon_2$}
\qbezier(-1,1)(-3,3)(-2,5.5)
\qbezier(-2,5.5)(1,9)(4,5.5)
\qbezier(4,5.5)(5,3)(3,1)
\put(3.1,1.4){\vector(-1,-1){0.3}}
\qbezier(19,1)(17,3)(18,5.5)
\qbezier(18,5.5)(21,9)(24,5.5)
\qbezier(24,5.5)(25,3)(23,1)
\put(23.1,1.4){\vector(-1,-1){0.3}}
\end{picture}
\vspace{0.2cm}
\end{center}
\[
\varepsilon_1^2=0=\varepsilon_2^2, \quad \varepsilon_2 \alpha=\alpha\varepsilon_1.
\]

(b) We continue Example \ref{example 2-cyclic complex of An}(b), with $\btau$ being the nontrivial involution of $Q$.
Then $\Lambda^\imath$ is isomorphic to the path algebra of the following
quiver $\ov{Q}$ with relations:
\begin{center}\setlength{\unitlength}{0.7mm}
 \begin{picture}(50,20)(0,-10)
\put(0,-2){$1$}
\put(20,-2){$3$}
\put(2,-11){$_{\alpha}$}
\put(17,-11){$_{\beta}$}
\put(2,-2){\vector(1,-2){8}}
\put(20,-2){\vector(-1,-2){8}}
\put(9.5,-22.5){$2$}
\color{purple}
\put(3,1){\vector(1,0){16}}
\put(19,-1){\vector(-1,0){16}}
\put(10,3){\small ${\varepsilon_1}$}
\put(10,-5){\small ${\varepsilon_3}$}
\put(10,-30){\small ${\varepsilon_2}$}
\begin{picture}(50,23)(-10,19)
\color{purple}
\qbezier(-1,-1)(-3,-3)(-2,-5.5)
\qbezier(-2,-5.5)(1,-9)(4,-5.5)
\qbezier(4,-5.5)(5,-3)(3,-1)
\put(3.1,-1.4){\vector(-1,1){0.3}}
\end{picture}
\end{picture}
\vspace{1.4cm}
\end{center}
\[
\varepsilon_1\varepsilon_3=0=\varepsilon_3\varepsilon_1,
\quad
 \varepsilon_2^2=0,
 \quad
 \varepsilon_2 \beta=\alpha\varepsilon_3,
 \quad
 \varepsilon_2 \alpha=\beta\varepsilon_1.
\]

(c) We continue Example \ref{example 2-cyclic complex of An}(c) with $\btau =\Id$. Then, $\Lambda^\imath\cong R_1\times R_1$.
\end{example}

\begin{example}
    \label{example diagonal 2}
{\rm ($\imath$quiver of diagonal type)}
Continuing Example \ref{example diagonal}, we let $\rm{swap}$ be the involution of $Q^{\rm dbl}$ uniquely determined by $\swa(i)=i^\diamond$ for any $i\in Q_0$. 
Then $(\Lambda^{\dbl})^\imath$ is isomorphic to $\Lambda$. Explicitly, let $(\ov{Q}^{\dbl},\ov{I}^{\dbl})$ be the bound quiver of $(\Lambda^{\dbl})^\imath$. Then $(\ov{Q}^{\dbl},\ov{I}^{\dbl})$ coincides with the double $\imath$quiver $(Q^{\sharp},I^{\sharp})$. So we just use $(Q^{\sharp},I^{\sharp})$ as the bound quiver of $(\Lambda^{\dbl})^\imath$ and identify $(\Lambda^{\dbl})^\imath$ with $\Lambda$.
\end{example}

From Proposition \ref{prop:invariant subalgebra}, every $\Lambda^{\imath}$-module corresponds to a representation of the bound quiver $(\ov{Q},\ov{I})$, and $\mod(\Lambda^{\imath})$ is isomorphic to the category $\rep_\K(\ov{Q},\ov{I})$ of representations of $(\ov{Q},\ov{I})$ over $\K$. Throughout this paper, we always identify these two categories.

\begin{remark}
   \label{rem:pushdown functor}
By the proof of Proposition~\ref{prop:invariant subalgebra}, there is a Galois covering $\pi: (Q^{\sharp},I^{\sharp})\rightarrow (\ov{Q},\ov{I})$. We also use $\pi: \Lambda\rightarrow \Lambda^\imath$ to denote this Galois covering.
Hence we obtain a pushdown functor \cite{Ga}
\begin{align}  \label{eq:pi}
\Pd:\mod (\Lambda) \longrightarrow \mod (\Lambda^{\imath}).
\end{align}
In particular, $\Pd$ preserves projective modules (also injective modules) and the almost split sequences. However, $\Pd$ may not be dense in general.

There exists an action of the cyclic group $\langle {\btau}^{\sharp} \rangle$ 
on the module category $\mod(\Lambda)$ induced by the involution ${\btau}^{\sharp}$ of $\Lambda$ in Lemma~\ref{lem:tildetheta}.
Denote by $\mod^{\langle {\btau}^{\sharp} \rangle}(\Lambda)$ the subcategory of $\mod(\Lambda)$
formed by the $\langle {\btau}^{\sharp} \rangle$-invariant modules, see \cite[Page 94]{Ga}.  
Then the pullup functor $\pi^*: \mod(\Lambda^\imath)\rightarrow \mod(\Lambda)$ induces an equivalence $\mod(\Lambda^\imath)\simeq \mod^{\langle {\btau}^{\sharp} \rangle}(\Lambda)$.
\end{remark}

By Proposition \ref{prop:invariant subalgebra}, $\Lambda^{\imath}$ is a non-negatively graded algebra when equipped with a principal grading $|\cdot |$ by
\[
|\varepsilon_i|=1,
\qquad
|\alpha| =0
\]
for each $i\in \I$ and each arrow $\alpha$ in $Q\subseteq \ov{Q}$.
Then $\Lambda^\imath=\bigoplus_{\ell\in\N}\Lambda^\imath_\ell$, where $\Lambda^{\imath}_\ell$ is the degree $\ell$ subspace of $\Lambda^{\imath}$.
From Proposition \ref{prop:invariant subalgebra}, we obtain the following.

\begin{corollary}
  \label{cor:subquotient}
We have $\Lambda^\imath=\Lambda^{\imath}_0 \bigoplus \Lambda^\imath_1$, where $\Lambda^{\imath}_0= \K Q$. In particular $\K Q$ is naturally a subalgebra and also a quotient algebra of $\Lambda^\imath$.
\end{corollary}

Viewing $\K Q$ as a subalgebra of $\Lambda^{\imath}$, we have a restriction functor
\[
\res: \mod (\Lambda^{\imath})\longrightarrow \mod (\K Q);
\]
viewing $\K Q$ as a quotient algebra of $\Lambda^{\imath}$, we obtain a pullback functor
\begin{equation}\label{eqn:rigt adjoint}
\iota:\mod(\K Q)\longrightarrow\mod(\Lambda^{\imath}).
\end{equation}

\begin{lemma}\label{lemma restriction functor}
The restriction functor $\res: \mod (\Lambda^{\imath})\rightarrow \mod (\K Q)$ is faithful and dense.
\end{lemma}

\begin{proof}
It is routine to prove that $\res$ is faithful, which will be omitted here. Moreover, $\res \circ \iota\simeq \Id$, that is, $\res$ is dense.
\end{proof}

We shall always identify $\mod (\K Q)$ with the full subcategory of $\mod(\Lambda^{\imath})$ by applying the natural embedding $\iota$, and denote by $\mod(kQ)\subseteq \mod(\Lambda^\imath)$.

\subsection{$\Lambda^{\imath}$ as a tensor algebra}\label{subsection:tensor algebra}

First, let us recall the definition of tensor algebras. Let $A$ be a $\K$-algebra, $M=\,_AM_A$ be an $A$-bimodule which is finitely generated on both sides. Write $M^{\otimes_A0}=k$ and $M^{\otimes_A (j+1)}=M\otimes_A(M^{\otimes_{_A} j})$ for $j\geq 0$. Denote by
\[
T_A(M)=\bigoplus_{j=0}^\infty M^{\otimes_{A} j}
\]
the tensor algebra. We shall assume $M$ is {\em nilpotent}, that is, there exists $N>0$ such that $M^{\otimes_A j}=0$ for any $j>N$.

Following \cite[Section 1]{BSZ}, Geiss, Leclerc and Schr\"{o}er \cite{GLS} give a criterion for a path algebra to be isomorphic to a tensor algebra, which we shall recall. Let $Q$ be a finite quiver, and let $w:Q_1\rightarrow \{0,1\}$ be a map assigning to each arrow of $Q$ a degree. Then $\K Q$ is a non-negatively graded algebra. Let $r_1,\dots,r_m$ be a set of relations for $\K Q$ which are homogeneous with respect to this grading. Suppose that there is some $1\leq l\leq m$ such that
$\deg(r_i)=0$ for $1\leq i\leq l$ and $\deg(r_j)=1$ for $l+1\leq j\leq m$.
Let $A:=\K Q/I$ where $I$ is the ideal generated by $r_1,\dots,r_m$. By assumption, $A$ is non-negatively graded. The  subspace $A_\ell$ of elements with degree $\ell$ is naturally an $A_0$-bimodule.
\begin{lemma}[\cite{BSZ,GLS}]
   \label{lemma criterion of tensor algebra}
The algebra $A$ is isomorphic to the tensor algebra $T_{A_0}(A_1)$.
\end{lemma}

Let $\mathcal{A}$ be an additive category with an additive endofunctor $F\colon \mathcal{A}\rightarrow \mathcal{A}$. By a \emph{representation} of $F$, we mean a pair $(X, u)$ with $X$ an object and $u\colon F(X)\rightarrow X$ a morphism in $\mathcal{A}$. A morphism $f\colon (X, u)\rightarrow (Y, v)$ between two representations is a morphism $f\colon X\rightarrow Y$  in $\mathcal{A}$ satisfying $f\circ u=v\circ F(f)$. This defines the category ${\rm rep}(F)$ of representations of $F$.

There is an isomorphism of categories
\begin{align}\label{equ:iso-cat}
{\rm rep}(M\otimes_A-)\stackrel{\simeq}\longrightarrow \mod(T_A(M)),
\end{align}
 which identifies a representation $(X, u)$ of $M\otimes_A-$ with  a left $T_A(M)$-module $X$ such that $m\cdot x=u(m\otimes x)$ for $m\in M$ and $x\in X$. We always identify these two categories in the following.

Now back to our setting, let $(Q, \btau)$ be an $\imath$quiver, and $\Lambda^\imath=\K\ov{Q}/\ov{I}$ with $(\ov{Q}, \ov{I})$ being defined in Proposition \ref{prop:invariant subalgebra}.
For each $i\in Q_0$, define a $k$-algebra
\begin{align}\label{dfn:Hi}
\BH _i:=\left\{ \begin{array}{cc}  \K[\varepsilon_i]/(\varepsilon_i^2) & \text{ if }\btau i=i,
 \\
\K(\xymatrix{i \ar@<0.5ex>[r]^{\varepsilon_i} & \btau i \ar@<0.5ex>[l]^{\varepsilon_{\btau i}}})/( \varepsilon_i\varepsilon_{\btau i},\varepsilon_{\btau i}\varepsilon_i)  &\text{ if } \btau i \neq i .\end{array}\right.
\end{align}
Note that $\BH _i=\BH _{\btau i}$ for any $i\in Q_0$. 
Choose one representative for each $\btau$-orbit on $\I$, and let
\begin{align}   \label{eq:ci}
\ci = \{ \text{the chosen representatives of $\btau$-orbits in $\I$} \}.
\end{align}
Define the following subalgebra of $\Lambda^{\imath}$:
\begin{equation}  \label{eq:H}
\BH =\bigoplus_{i\in \ci }\BH _i.
\end{equation}
Note that $\BH $ is a radical square zero selfinjective algebra. Denote by
\begin{align}
\res_\BH :\mod(\Lambda^\imath)\longrightarrow \mod(\BH )
\end{align}
the natural restriction functor.

Define
\[
\Omega:=\Omega(Q) =\{(i,j) \in Q_0\times Q_0\mid  \exists (\alpha:i\rightarrow j)\in Q_1\}.
\]
Then $\Omega$ represents the orientation of $Q$. Since $Q$ is acyclic, if $(i,j)\in\Omega$, then $(j,i)\notin \Omega$.
We also use $\Omega(i,-)$ to denote the subset $\{j\in Q_0 \mid \exists (\alpha:i\rightarrow j)\in Q_1\}$, and $\Omega(-,i)$ is defined similarly.

For any $(i,j)\in \Omega$, we define
\begin{equation}  \label{eq:jHi}
{}{}_j\BH_i := \BH _j\Span_k\{\alpha,\btau\alpha\mid(\alpha:i\rightarrow j)\in Q_1\text{ or } (\alpha:i\rightarrow \btau j)\in Q_1\}\BH _i.
\end{equation}
Note that ${}{}_j\BH_i ={}_{\btau j} \BH _{\btau i}={}_{j} \BH _{\btau i}={}_{\btau j} \BH _{i}$ for any $(i,j)\in \Omega$.

We describe a basis of ${}_j\BH_i $ (as $\K$-linear space) for each $(i,j)\in\Omega$ by separating into 2 cases (i)-(ii):
\begin{itemize}
\item[(i)] $\btau i=i$ and $\btau j=j$. Then $\{\alpha,\alpha\varepsilon_i\mid (\alpha:i\rightarrow j)\in Q_1\}$ forms a basis of ${}_j\BH_i $.

\item[(ii)]  $\btau i\neq i$ or $\btau j\neq j$. Then
\begin{align*}
&\{\alpha,\btau\alpha,\varepsilon_j\alpha =\btau\alpha \varepsilon_i, \varepsilon_{\btau j} \btau\alpha =\alpha \varepsilon_{\btau i}\mid (\alpha:i\rightarrow j)\in Q_1\}
\\ \cup &
\{\alpha,\btau\alpha,\varepsilon_{\btau j}\alpha =\btau\alpha \varepsilon_i, \varepsilon_{j} \btau\alpha =\alpha \varepsilon_{\btau i}\mid (\alpha:i\rightarrow \btau j)\in Q_1\}
\end{align*}
forms a basis of ${}_j\BH_i $.

\end{itemize}
So ${}_j\BH_i $ is an $\BH _j\mbox{-}\BH _i$-bimodule, which is free as a left $\BH _j$-module and free as a right $\BH _i$-module. In particular, for $(i,j)\in \Omega$,
define
\begin{eqnarray}\label{basis of Hij left}
_j\LL_i&=&\left\{ \begin{array}{cc}
\{\alpha \mid (\alpha:i\rightarrow j)\in Q_1\} & \text{ if }\btau i=i, \btau j=j,\\
\{\alpha+\btau\alpha \mid (\alpha:i\rightarrow j)\in Q_1\} & \text{ if }\btau i=i, \btau j\neq j,\\
\{\alpha,\btau \alpha \mid (\alpha:i\rightarrow j)\in Q_1\} & \text{ if }\btau i\neq i,\btau j=j,\\
\{\alpha+\btau \alpha \mid (\alpha:i\rightarrow j) \text{ or }(\alpha:i\rightarrow \btau j)\in Q_1\} & \text{ if }\btau i\neq i,\btau j\neq j;\label{eqn:basis of L}
\end{array}\right.\\
_j\RR_i&=& \left\{ \begin{array}{cc}\label{basis of Hij right}
\{\alpha \mid (\alpha:i\rightarrow j)\in Q_1\} & \text{ if }\btau i=i,  \btau j=j,\\
\{\alpha,\btau \alpha \mid (\alpha:i\rightarrow j)\in Q_1\} & \text{ if }\btau i=i, \btau j\neq j,\\
\{\alpha+\btau \alpha \mid (\alpha:i\rightarrow j)\in Q_1\} & \text{ if }\btau i\neq i,\btau j=j,\\
\{\alpha+\btau \alpha \mid (\alpha:i\rightarrow j) \text{ or }(\alpha:i\rightarrow \btau j)\in Q_1\} & \text{ if }\btau i\neq i,\btau j\neq j.\end{array}\right. \label{eqn:basis of R}
\end{eqnarray}
Then $_j\LL_i$ (respectively, $_j\RR_i$) is a basis of ${}_j\BH_i $ as a left $\BH _j$-modules (respectively, as a right $\BH _i$-modules).

Denote
\begin{equation}  \label{eq:ovOmega}
\overline{\Omega}:=\{(i,j)\in \ci \times \ci \mid (i,j)\in\Omega\text{ or }(i,\btau j)\in\Omega\},
\end{equation}
and define the following $\BH $-$\BH $-bimodule
\begin{equation}
   \label{eq:MHH}
   M:=\bigoplus_{(i,j)\in\ov{\Omega}} {}_j\BH_i .
\end{equation}

\begin{proposition}
    \label{prop:tensor algebra}
The algebra $\La^\imath$ is isomorphic to the tensor algebra $T_\BH (M)$: $\La^{\imath}\cong T_\BH (M)$.
\end{proposition}

\begin{proof}
In this proof, we shall consider another non-negative grading of $\Lambda^\imath$ different from the one in Corollary~ \ref{cor:subquotient}. It follows by Proposition \ref{prop:invariant subalgebra} that the algebra $\Lambda^{\imath}$ admits a new non-negative grading by setting $\deg(\varepsilon_i)=0$ for each arrow $\varepsilon_i$ in $\ov{Q}$ and $\deg(\alpha)=1$ for each arrow $\alpha$ in $Q$ (viewed as a subquiver of $\ov{Q}$).

It follows by Proposition~ \ref{prop:invariant subalgebra} that the generators of $\ov{I}$ is homogeneous, and $\BH $ is the subalgebra of elements of degree zero, and $M$ is the subspace of elements of degree $1$. Now the assertion follows from Lemma \ref{lemma criterion of tensor algebra}.
\end{proof}

\subsection{Modulated graphs for $\imath$quivers}

In this section we generalize the notion of representation theory of \emph{modulated graphs} to the setting of $\imath$quivers. See  \cite{DR, Li, GLS} for details about representations of modulated graphs.

Recall from \eqref{eq:ci} that $\ci$ is a (fixed) subset of $Q_0$ formed by the representatives of all $\btau$-orbits.
The tuple $\{(\BH _i,{}_j\BH_i)\mid {i\in\ci ,(i,j)\in\ov{\Omega}}\}$ as defined in \eqref{dfn:Hi} and \eqref{eq:jHi} is called a \emph{modulation} of $(Q, \btau)$ and is denoted by $\cm(Q, \btau)$.

A representation $(N_i,N_{ji}):=(N_i,N_{ji})_{i\in\ci ,(i,j)\in\ov{\Omega}}$ of $\cm(Q, \btau)$ is defined by assigning to each $i\in \ci$ a finite-dimensional $\BH _i$-module $N_i$ and to each $(i,j)\in \overline{\Omega}$ an $\BH _j$-morphism
$N_{ji}:{}_j\BH_i \otimes_{\BH _i} N_i\rightarrow N_j$. A morphism $f:L\rightarrow N$ between representations $L=(L_i,L_{ji})$ and $N=(N_i,N_{ji})$ of $\cm(Q, \btau)$ is a tuple $f=(f_i)_{i\in \ci}$ of $\BH _i$-morphisms $f_i:L_i\rightarrow N_i$ such that the following diagram is commutative for each $(i,j)\in\overline{\Omega}$:
\[\xymatrix{_j\BH_i\otimes_{\BH _i} L_i \ar[r]^{1\otimes f_i} \ar[d]^{L_{ij}}&  _j\BH_i \otimes_{\BH _i} N_i\ar[d]^{N_{ij}}\\
 L_j\ar[r]^{f_j} & N_j}\]

Similar to the one-point extension of algebras, the representations of $\cm(Q, \btau)$ form an abelian category $\rep(\cm(Q, \btau))$.
Similar to \cite[Proposition 5.1]{GLS} (see also \cite[Theorem~ 3.2]{Li}), we shall show that $\rep(\cm(Q, \btau))$ is isomorphic to $\rep(\ov{Q},\ov{I})$.

For $(N_i,N_{ij})\in \rep(\cm(Q, \btau))$, we define a representation $(X_j,X(\alpha),X(\varepsilon_j))_{j\in Q_0,\alpha\in Q_1}$
of $\Lambda^{\imath}$ as follows. Since $N_i$ is an $\BH _i$-module, for any $j\in Q_0$, let $X_j=e_j N_j$ if $j\in\ci $ or $X_{j}=e_{j}N_{\btau j}$ otherwise.
Define a $\K$-linear map $X(\varepsilon_i):X_i\rightarrow X_j$ by letting
$X(\varepsilon_i)(x)=\varepsilon_i\cdot x$.
For any $(\alpha:i\rightarrow j)\in Q_1$,  then $(i,j)\in\Omega$. By the basis of $_j\BH_i $ described in \eqref{basis of Hij left}--\eqref{basis of Hij right}, we define a $\K$-linear map $X(\alpha):X_i\rightarrow X_j$ by
\[
X(\alpha)(x) = N_{ij}(\alpha\otimes x).
\]
One checks that $(X_j,X(\alpha),X(\varepsilon_j))_{j\in Q_0,\alpha\in Q_1}$ is actually a representation of $(\ov{Q},\ov{I})$.

Conversely, let $(X_i,X(\alpha),X(\varepsilon_i))\in\rep(\ov{Q},\ov{I})$. If $\btau i=i$, then $N_i=(X_i,X(\varepsilon_i))$ is naturally an $\BH _i$-module; if $\btau i\neq i$, then $N_i=(X_i,X_{\btau i},X(\varepsilon_i),X(\varepsilon_{\btau i}))$ is an $\BH _i$-module. Note that $N_i=N_{\btau i}$. For $(i,j)\in \ov{\Omega}$, then either $(i,j)\in \Omega$ or $(i,\btau j)\in \Omega$, and there exists an $\BH _j$-morphism
\[
N_{ji}: \,_j \BH _i\otimes_{\BH _i} N_i\longrightarrow N_j
\]
determined by
$N_{ji}(\alpha\otimes x):= X(\alpha)(x)$ for any arrow $\alpha:i\rightarrow j$ and $\alpha:\btau i\rightarrow \btau j$ if $(i,j)\in\Omega$ and for any $\alpha:i\rightarrow \btau j$ or $\alpha:\btau i\rightarrow j$ if $(i,\btau j)\in \Omega$.
Then $(N_i,N_{ji})\in \rep(\cm(Q, \btau))$.

A direct computation shows that the above functors between $\rep(\cm(Q, \btau))$ and $\rep(\ov{Q},\ov{I})$ are mutual inverses. Then we have established the following.

\begin{proposition}
   \label{prop:modulated representation}
The categories $\rep(\cm(Q, \btau))$ and $\rep(\ov{Q},\ov{I})$ are isomorphic.
\end{proposition}

\begin{remark}
The materials in this section will play a basic role in the constructions of BGP-type reflection functors associated to $\imath$quivers in \cite{LW19b}.
\end{remark}

\section{Homological properties of the algebra $\La^\imath$}
 \label{sec:iLambda}

In this section, we shall study the homological properties of the algebra $\La^\imath$, such as Gorenstein homological properties and singularity categories.

\subsection{Gorenstein projective modules}
  \label{subsec:Gproj}

In this subsection, we review briefly and set up notations for Gorenstein algebras and Gorenstein projective modules. Let $\K$ be a field. Let $A$ be a finite-dimensional $\K$-algebra.
A complex
\[
P^\bullet:\cdots\longrightarrow P^{-1}\longrightarrow P^0\stackrel{d^0}{\longrightarrow} P^1\longrightarrow \cdots
\]
of finitely generated projective $A$-modules is said to be \emph{totally acyclic} provided it is acyclic and the Hom complex $\Hom_A(P^\bullet,A)$ is also acyclic. 
An $A$-module $M$ is said to be (finitely generated) \emph{Gorenstein projective} provided that there is a totally acyclic complex $P^\bullet$ of projective $A$-modules such that $M\cong \Ker d^0$ \cite{EJ}. We denote by $\Gproj(A)$ the full subcategory of $\mod(A)$ consisting of Gorenstein projective modules.

A $k$-algebra $A$ is called a {\em Gorenstein algebra} \cite{EJ, Ha3}  if $\ind\,_A A<\infty$ and $\ind A_A <\infty$. It is known that a $k$-algebra $A$ is Gorenstein if and only if $\ind {}_AA<\infty$ and $\pd D(A_A)<\infty$. For a Gorenstein algebra $A$, by Zaks' lemma we have $\ind {}_AA=\ind A_A$, and the common value is denoted by $\Gd A$. If $\Gd A\leq d$, we say that $A$ is a \emph{$d$-Gorenstein} algebra.




For a module $M$ take a short exact sequence $0\rightarrow\Omega(M)\rightarrow P\rightarrow M\rightarrow0$ with $P$ projective. The module $\Omega(M)$ is called a \emph{syzygy module} of $M$. Syzygy modules of $M$ are not uniquely determined, while they are naturally isomorphic to each other in the stable category $\underline{\mod}(A)$. For each $d \ge 1$ denote by $\Omega^d(\mod(A))$ the subcategory of modules of the form $\Omega^d(M)$ for an $A$-module $M$.

\begin{theorem}[\text{\cite[Theorem 3.2]{AM}}]
   \label{theorem characterize of gorenstein property}
Let $A$ be an algebra and let $d\geq0$. Then the following statements are equivalent:
\begin{enumerate}
\item[(a)]
the algebra $A$ is $d$-Gorenstein;
\item[(b)]
$\Gproj(A)= \Omega^d(\mod(A))$.
\end{enumerate}
In this case, an $A$-module $G$ is Gorenstein projective if and only if there is an exact sequence $0\rightarrow G\rightarrow P^0\rightarrow P^1\rightarrow \cdots$ with each
$P^\ell$ projective.
\end{theorem}

The following lemma is standard.

\begin{lemma}
\label{lemma perpendicular of P GP}
For each Gorenstein projective $A$-module $M$, we have
$\Ext^p_A(M,U)=0$ for all $p>0$ and all $U$ of finite projective dimension.
\end{lemma}

Let $\cp^{\leq d}(A)$ be the subcategory of $\mod(A)$ which consists of $A$-modules of projective dimension less than or equal to $d$, for $d\in\N$.

\begin{lemma}[\text{\cite[Theorem 1.1]{AB}}]
   \label{lemma approx of GP}
Let $A$ be a $d$-Gorenstein algebra. Then for each $M\in \mod(A)$, there are short exact sequences
\begin{eqnarray*}
&&0\longrightarrow H_M\longrightarrow G_M\longrightarrow M\longrightarrow0,  \\
&&0\longrightarrow M\longrightarrow H^M\longrightarrow G^M\longrightarrow0,
\end{eqnarray*}
where $H_M\in\cp^{<d}(A),H^M\in\cp^{\leq d}(A)$, $G_M,G^M\in\Gproj(A)$.
\end{lemma}

In fact, for a $d$-Gorenstein algebra $A$, $(\Gproj(A),\cp^{\leq d}(\mod(A)))$ is a cotorsion pair in $\mod(A)$, see \cite{EJ}. 

Let $A$ be a finite-dimensional algebra. The {\em singularity category} of $A$ is defined (cf. \cite{Ha3}) to be the Verdier localization
\[
 D_{sg}(\mod(A)):=D^b(\mod(A))/K^b(\proj(A)).
\]

\begin{theorem}[Buchweitz-Happel's Theorem; see \cite{Ha3}]
\label{thm:Buchweitz-Happel}
Let $A$ be a finite-dimensional algebra. Then $\Gproj(A)$ is a Frobenius category with the projective modules as the projective-injective objects. There is an exact embedding $\Phi:\underline{\Gproj}A\rightarrow D_{sg}(\mod(A))$ given by $\Phi(M)=M$, where the second $M$ is the corresponding stalk complex at degree $0$, and $\Phi$ is an equivalence if and only if $A$ is Gorenstein.
\end{theorem}

\subsection{$\Lambda^\imath$ as a $1$-Gorenstein algebra}
  \label{subsec:Gorenstein}

In this subsection, we shall prove that $\Lambda^\imath$ is $1$-Gorenstein and then give a characterization of Gorenstein projective $\Lambda^{\imath}$-modules.

The following result might be known for experts; we include a proof here as we cannot find a suitable reference.

\begin{lemma}\label{lemma equivalent of projective module}
Let $Q$ be an acyclic quiver. A representation $N=(N_i,N(\alpha))\in\rep_\K(Q)$ is projective if and only if the map
$$(N(\alpha))_\alpha:\bigoplus_{(\alpha:j\rightarrow i)\in Q_1} N_j \longrightarrow N_i$$
is injective for each $i\in Q_0$.
\end{lemma}

\begin{proof}
Set $A=\K Q$.
We prove it by induction on the cardinality of $Q_0$.
First, if $|Q_0|=1$, then $A \cong \K$, and the result is trivial.

For $|Q_0|=n>1$, as $Q$ is an acyclic quiver, there exists at least one sink vertex, denoted by $0$. Let $B=(1-e_0)A(1-e_0)$, where $e_0$ is the idempotent corresponding to the vertex $0$. Then $B$ is also a hereditary algebra, and $A$ is a one-point coextension of $B$. So
$A=\left( \begin{array}{cc} \K&  M \\ 0& B\end{array}\right)$
with $M_B$ a right projective $B$-module. In this way, any $N=(N_i,N(\alpha))\in\rep_\K(Q)$ can be identified with a triple
$\left( \begin{array}{cc} X \\ Y \end{array} \right)_\phi$,
where $X\in\mod (\K)$, $Y\in\mod (B)$ and $\phi:M\otimes_B Y\rightarrow X$ is a linear map. In fact $X=N_0$, and $\phi$ coincides with
\begin{equation}\label{eqn:mor1}
(N(\alpha))_\alpha:\bigoplus_{(\alpha:i\rightarrow 0)\in Q_1} N_i\rightarrow N_0.
\end{equation}

For the `if' part, assume $\phi$ is injective. Let $Z =\coker (\phi)$. Then
\[
\left( \begin{array}{cc} X \\ Y \end{array} \right)_\phi\cong \left( \begin{array}{cc} M\otimes_B Y\\ Y \end{array} \right)_{\Id}\oplus \left( \begin{array}{cc} Z \\ 0 \end{array} \right)_0.
\]
Clearly, $ \left( \begin{array}{cc} Z \\ 0 \end{array} \right)_0$ is a projective $A$-module. By the inductive assumption, $Y$ is a projective $B$-module. Then  $\left( \begin{array}{cc} M\otimes_B Y\\ Y \end{array} \right)_{\Id}$ is projective, and so $N$ is a projective $A$-module.

For the `only if' part, by the inductive assumption,
$(N(\alpha))_\alpha:\bigoplus_{(\alpha:j\rightarrow i)\in Q_1} N_j\rightarrow N_i$
is injective for $i\neq 0$. For $i=0$, since the desired map \eqref{eqn:mor1} coincides with $\phi$, from the description of projective modules in $\rep_k(Q)$  \cite[Chapter III.2, Lemma 2.4]{ASS}, we obtain that $\phi$ is injective.

The lemma is proved.
\end{proof}

\begin{proposition}
  \label{proposition of 1-Gorenstein}
Let $(Q, \btau)$ be an $\imath$quiver. Then
\begin{itemize}
\item[(1)] $\Lambda^{\imath}$ is a $1$-Gorenstein algebra;
\item[(2)] for any $X\in\mod(\Lambda^{\imath})$, we have $X\in \Gproj(\Lambda^{\imath})$ if and only if $\res(X)$ is a projective $\K Q$-module.
\end{itemize}
\end{proposition}

\begin{proof}
(1) It is well known that every projective (respectively, injective) $\Lambda^\imath$-module is gradable, i.e. in the image of the pushdown functor $\Pd$ in \eqref{eq:pi} (up to isomorphism). In particular, there exist a projective $\Lambda$-module $P$ and an injective $\Lambda$-module $I$ such that
$\Pd(P)\cong {}_{\Lambda^\imath}(\Lambda^\imath)$ and $\Pd(I)\cong D((\Lambda^\imath)_{\Lambda^\imath})$.
 As $\Lambda=\K Q\otimes R_2$ is $1$-Gorenstein, let
\begin{align*}
&0 \rightarrow P \longrightarrow I^0 \longrightarrow I^1 \longrightarrow0, \qquad
 0\longrightarrow P^1\longrightarrow P^0 \longrightarrow I \longrightarrow 0
\end{align*}
be an injective resolution of $P$ and a projective resolution of $I$ respectively.
By applying $\Pd$ to these two resolutions and noting that $\Pd$ is an exact functor preserving projectives and injectives, we conclude that $\Lambda^\imath$ is $1$-Gorenstein.

(2)
Let $X=(X_i,X(\alpha),X(\varepsilon_i))_{i\in\ov{Q}_0,\alpha\in \ov{Q}_1}\in\mod(\Lambda^\imath)$ correspond under the category equivalence in Proposition~\ref{prop:modulated representation} to $(N_i,N_{ji})\in\rep(\cm(Q, \btau))$. In particular, $N_i$ is the restriction of $X$ to $\BH _i$, i.e.,
the underlying space of $N_i$ is $X_i$ if $\btau i=i$, and is $X_i\oplus X_{\btau i}$ if $\btau i\neq i$.

By construction, an $\BH $-bimodule $M$ is projective as a left (and also right) $\BH $-module. Then $M$ is perfect in the sense of \cite[Definition 4.4]{CL}. The following statements are equivalent:
\begin{enumerate}
\item[(i)]
$X\in \Gproj(\Lambda^{\imath})$;
\item[(ii)]
$\bigoplus_{j:(i,j)\in\ov{\Omega}}\,_j\BH_i \otimes_{\BH _i} N_i\longrightarrow N_j$ is injective, for any $i\in Q_0$; \item[(iii)]
$\bigoplus_{(\alpha:i\rightarrow j)\in Q_1} X_i\xrightarrow{(X(\alpha))_\alpha} X_j$ is injective, for each $i\in Q_0$.
\end{enumerate}
The equivalence between (i) and (ii) follows by \cite[Proposition 4.5, Theorem 3.9]{CL}, and the equivalence of (ii) and (iii) follows by a direct computation using \eqref{eqn:basis of R}.

Now the assertion (2) follows from Lemma \ref{lemma equivalent of projective module}.
\end{proof}

Denote by $\cc_{\Z/2}(\proj(\K Q))$ the category of  $\Z/2$-graded complexes over $\proj(\K Q)$ \cite{PX, Br}. Recall $\Lambda$ itself is an algebra arising from an $\imath$quiver of diagonal type by Example~\ref{example diagonal 2}. In this case  Proposition~ \ref{proposition of 1-Gorenstein} gives us the following equivalence
\begin{equation}
  \label{eq:LaZ2}
\Gproj(\Lambda)\cong \cc_{\Z/2}(\proj(\K Q)).
\end{equation}
Another corollary of  Proposition~ \ref{proposition of 1-Gorenstein} was known before. Denote by $\cc_{\Z/1}(\proj(\K Q))$ the category of $\Z/1$-graded complexes over $\proj(\K Q)$.
\begin{corollary} [\cite{RZ}]
We have
$\Gproj(\Lambda^\imath)\cong\cc_{\Z/1}(\proj(kQ))$, if $\btau=\Id$.
\end{corollary}

\subsection{$\La^\imath$-modules with finite projective dimensions}

In this subsection, we shall characterize the objects in $\cp^{\leq 1}(\Lambda^{\imath})$ in terms of simple objects.

Recall $\BH $ from \eqref{eq:H}. Denote by $\res_\BH : \mod(\Lambda^{\imath})\rightarrow \mod(\BH )$ the restriction functor.
On the other hand, as $\BH $ is a quotient algebra of $\Lambda^\imath$, every $\BH $-module can be viewed as a $\Lambda^\imath$-module.

Recall the algebra $\BH _i$ for $i \in \ci$ from \eqref{dfn:Hi}. For $i\in Q_0 =\I$, define the indecomposable module over $\BH _i$ (if $i\in \ci$) or over $\BH_{\btau i}$ (if $i\not \in \ci$)
\begin{align}
  \label{eq:E}
\E_i =\begin{cases}
k[\varepsilon_i]/(\varepsilon_i^2), & \text{ if }\btau i=i;
\\
\xymatrix{\K\ar@<0.5ex>[r]^1 & \K\ar@<0.5ex>[l]^0} \text{ on the quiver } \xymatrix{i\ar@<0.5ex>[r]^{\varepsilon_i} & \btau i\ar@<0.5ex>[l]^{\varepsilon_{\btau i}} }, & \text{ if } \btau i\neq i.
\end{cases}
\end{align}
Then $\E_i$, for $i\in Q_0$, can be viewed as a $\Lambda^\imath$-module and will be called a {\em generalized simple} $\Lambda^\imath$-module.

\begin{lemma}
    \label{lemma locally projective modules}
We have $\pd_{\Lambda^{\imath}} (\E_i)\leq 1$ and $\ind_{\Lambda^{\imath}} (\E_i)\leq 1$, for any $i\in Q_0$.
\end{lemma}

\begin{proof}
We view $\K Q$ as a subalgebra of $\Lambda^{\imath}$. Proposition \ref{proposition of 1-Gorenstein} implies that the left (and respectively, right) regular module $\Lambda^{\imath}$ is projective as a $\K Q$-module.
For any $i\in Q_0$, the simple $\K Q$-module $S_i$ admits the following projective resolution
\[
0\longrightarrow \bigoplus_{(\alpha: i\rightarrow j)\in Q_1} (\K Q)e_j\longrightarrow (\K Q)e_i\longrightarrow S_i\longrightarrow0.
\]
By applying $\Lambda^{\imath}\otimes_{\K Q}-$ to it and noting that $\Lambda^{\imath}\otimes_{\K Q}S_i\cong \E_i$, we obtain the following projective resolution of $\E_i$
\begin{equation}\label{eqn:projective resolution of E}
0\longrightarrow \bigoplus_{(\alpha: i\rightarrow j)\in Q_1} \Lambda^{\imath}\, e_j\longrightarrow \Lambda^{\imath}\,e_i\longrightarrow \E_i \longrightarrow0.
\end{equation}
Thus $\pd_{\Lambda^{\imath}} (\E_i)\leq 1$.
As $\Lambda^{\imath}$ is $1$-Gorenstein, it follows that $\ind_{\Lambda^{\imath}} (\E_i)\leq 1$.
\end{proof}

Dual to (\ref{eqn:projective resolution of E}), we obtain the injective resolution of $\E_i$:
\begin{equation}\label{eqn:injective resolution of E}
0\longrightarrow \E_i \longrightarrow D(e_{\btau i} \Lambda^{\imath})\longrightarrow \bigoplus_{(\alpha:j\rightarrow \btau i)\in Q_1}D(e_{j}  \Lambda^{\imath})\longrightarrow0,
\end{equation}
by considering the right modules, and applying the duality functor $D$.

\begin{proposition}
\label{propostion equivalence of locally projective modules}
For any $M\in\mod(\Lambda^{\imath})$, the following are equivalent:
\begin{itemize}
\item[(i)] $\pd M<\infty$;
\item[(ii)] $\ind M<\infty$;
\item[(iii)] $\pd M\leq1$;
\item[(iv)] $\ind M\leq1$;
\item[(v)] $\res_\BH (M)$ is projective as $\BH $-module.
\end{itemize}
\end{proposition}
\begin{proof}
Proposition \ref{proposition of 1-Gorenstein} states that $\Lambda^{\imath}$ is $1$-Gorenstein, and then the equivalence of (i)--(iv) follows.
By using Lemma \ref{lemma locally projective modules}, the proof of the equivalence ``(i)$\Leftrightarrow$(v)'' is similar to \cite[Proposition~ 3.5]{GLS}, and hence is omitted.
\end{proof}

Recall the generalized simple $\La^\imath$-module $\E_i$ from \eqref{eq:E}.

\begin{corollary}\label{corollary locally projective modules}
A $\Lambda^{\imath}$-module $M$ has finite projective dimension if and only if it has a filtration $0=M_t\subset M_{t-1}\subset \cdots \subset M_{1}\subseteq M_0=M$ such that $M_j/M_{j+1}\cong \E_{i_j}$ for some $i_j\in Q_0$, for $0\leq j\leq t$.
\end{corollary}

\begin{proof}
The ``if" direction follows from Lemma \ref{lemma locally projective modules}.

It remains to prove the ``only if" direction. Assume that $M$ has finite projective dimension. By Proposition~ \ref{propostion equivalence of locally projective modules}, $\res_\BH (M)$ is projective as $\BH $-module. As $Q$ is an acyclic quiver, there exists a vertex $i$ of the quiver $Q$ such that $e_iM\neq 0$ or $e_{\btau i}M\neq 0$, and $e_jM=0$ for all $j\in\Omega(i, -)\cup\Omega(\btau i,-)$. If $\btau i=i$, we have a short exact sequence
\begin{align}\label{eqn:cor projective}
0\longrightarrow e_iM\longrightarrow M\longrightarrow(1-e_i)M\longrightarrow0.
\end{align}
Note $e_i\res_\BH (M)=e_iM$ viewed as $\BH _i$-modules. So $e_iM$ is a projective $\BH _i$-module, and then $e_iM\cong \E_i^{\oplus m_i}$ for some $m_i$. By applying $\res_\BH $ to \eqref{eqn:cor projective}, we see that $\res_\BH ((1-e_i)M)$ is projective as an $\BH $-module.
If $\btau i\neq i$, then
\begin{align*}
0\longrightarrow (e_i+e_{\btau i})M\longrightarrow M\longrightarrow(1-e_i-e_{\btau i})M\longrightarrow0
\end{align*}
is an exact sequence with $\res_{\BH }((e_i+e_{\btau i})M)$  projective as $\BH _i$-modules. Then $(e_i+e_{\btau i})M\cong \E_i^{\oplus m_i} \oplus \E_{\btau i}^{\oplus m_{\btau i}}$ for some $m_i,m_{\btau i}$. A similar argument as above shows that $\res_\BH ((1-e_i -e_{\btau i})M)$ is projective as an $\BH $-module.

Now the ``only if" direction follows by induction on dimension of $M$.
\end{proof}

\subsection{Projective $\La^\imath$-modules}
  \label{subsection: description of projectives}

We describe the indecomposable projective $\Lambda^{\imath}$-modules in this subsection.

For each $i\in Q_0$, we denote by $P_i$ the indecomposable projective $\K Q$-module $(kQ)e_i$.
Recall that $R_2$ is the radical square zero of the path algebra of $\xymatrix{1 \ar@<0.5ex>[r]^{\varepsilon} & 2 \ar@<0.5ex>[l]^{\varepsilon'}}$. We shall identify in this subsection $\mod(\Lambda)\cong \cc_{\Z/2}(\mod(\K Q))$ as in Lemma~\ref{lem:Rm}.

\begin{lemma} [\cite{Br}]
   \label{lem:projectitve of Lambda}
A $\Lambda$-module is indecomposable projective if and only if it is isomorphic to $\xymatrix{ P_i \ar@<0.5ex>[r]^{1}& P_i \ar@<0.5ex>[l]^{0}  }$ or
$\xymatrix{ P_i \ar@<0.5ex>[r]^{0}& P_i \ar@<0.5ex>[l]^{1}  }$ for some indecomposable projective $\K Q$-module $P_i$.
\end{lemma}

\begin{proof}
The ``if'' part follows from \cite[Lemma 3.3]{Br}.

For the ``only if'' part, let $M$ be an  indecomposable projective module. As $\Lambda=\K Q\otimes_k R_2$, the restriction of $M$ to the subquiver $Q\bigcup Q'\subseteq Q^{\sharp}$ is projective, so $M\in\cc_{\Z/2}(\proj(\K Q))$. It follows by the decomposition in \cite[Lemmas 4.2, 3.3]{Br} that $M$ is an acyclic complex.
Then the assertion follows from \cite[Lemma 3.2]{Br}.
\end{proof}

\begin{proposition}\label{prop:projective module of lambdai}
A $\Lambda^{\imath}$-module $X=(X_i,X(\alpha), X(\varepsilon_i))_{i\in Q_0,\alpha\in Q_1}$ is indecomposable projective if and only if the $\K Q$-module $(X_i,X(\alpha))$ is equal to $P_j\oplus P_{\btau j}$ and $X(\varepsilon_j)$ is a linear isomorphism, for some $j\in Q_0$; see \eqref{eqn:rigt adjoint}. In particular, we have a short exact sequence in $\mod(\Lambda^\imath)$:
\begin{align}
   \label{eqn:projective resolution of kQ}
0\longrightarrow P_{\btau j}\longrightarrow (\Lambda^\imath) e_j \longrightarrow P_{j}\longrightarrow0.
\end{align}
\end{proposition}

\begin{proof}
The pushdown functor $\Pd:\mod(\Lambda)\rightarrow \mod (\Lambda^\imath)$ defined in \eqref{eq:pi} is exact and preserves projective modules.
Then the assertion follows by applying $\Pd$ to the indecomposable projective $\La$-modules in Lemma \ref{lem:projectitve of Lambda}.
\end{proof}

\subsection{Singularity category of Gorenstein algebras}
  \label{subsec:tilting}

Let $A=\bigoplus_{\ell\in\N} A_\ell$ be a non-negatively graded finite-dimensional algebra, and $\mod^\Z(A)$ be the category of finitely generated graded $A$-modules.
For $\ell\in \Z$, the \emph{grade shift functor} $(\ell):\mod^\Z(A)\rightarrow \mod^\Z(A)$, $X \mapsto X(\ell)$, is defined by letting $X(\ell) =\oplus_{j\in\Z}X(\ell)_j$, where $X(\ell)_j:=X_{j+\ell}$. Following \cite{Ya}, the truncation functors
\[
(-)_{\geq \ell}:\mod^\Z(A)\longrightarrow\mod^\Z(A),\quad (-)_{\leq \ell}:\mod^\Z(A)\longrightarrow \mod^\Z(A)
\]
are defined as follows. For a $\Z$-graded $A$-module $X=\bigoplus_{\ell\in\Z}X_\ell $, $X_{\geq \ell}$ is a $\Z$-graded $A$-submodule of $X$
defined by
\[
(X_{\geq \ell})_j:=\left\{ \begin{array}{ccc} 0& \mbox{ if }j<\ell \\ X_j& \mbox{ if }j\geq \ell,\end{array} \right.
\]
and $X_{\leq \ell}$ is a $\Z$-graded quotient $A$-module $X/X_{\geq \ell+1}$ of $X$.

Now we define a $\Z$-graded $A$-module by
\[
T:=\bigoplus_{\ell\geq0} A(\ell)_{\leq0}.
\]
For $\ell$ sufficiently large, $A(\ell)_{\leq0}$ is projective since $A$ is finite-dimensional. So we can regard $T$ as an object in $D_{sg}(\mod^\Z(A))$ (by noting that every projective module is zero in $D_{sg}(\mod^\Z(A))$).

Let $\proj^\Z(A)$ be the subcategory of finitely generated graded projective $A$-modules.
Define the {\em graded singularity category} of $A$ to be
\[
D_{sg}(\mod^\Z(A)):=D^b(\mod^\Z(A))/K^b(\proj^\Z(A)).
\]
Similarly, one can define a notion of \emph{graded Gorenstein projective modules}.
We denote by $\Gproj^\Z(A)$ the full subcategory of $\mod^\Z(A)$ formed by all $\Z$-graded Gorenstein projective modules. Note that a graded $A$-module is graded Gorenstein projective if it is Gorenstein projective as an ungraded module, see e.g., \cite{LZ}. The forgetful functor $\mod^\Z(A)\rightarrow \mod(A)$ maps graded Gorenstein projective modules to Gorenstein projective modules. The results stated in Theorem \ref{thm:Buchweitz-Happel}  also holds for graded algebras and graded modules.

For any positive integer $m$, every $\Z$-graded algebra $A=\bigoplus_{\ell\in \Z} A_\ell$ can be viewed naturally as a $\Z/m$-graded algebra $A=\bigoplus_{\bar \ell \in \Z/m}A_{\ov \ell}$ with $A_{\ov \ell} =\bigoplus_{\ell' \equiv \ell \pmod m} A_{\ell'}$; thus  one can also define $\mod^{\Z/m}(A)$, $\Gproj^{\Z/m}(A)$ and $D_{sg}(\mod^{\Z/m}(A))$.

\begin{lemma}[\text{\cite[Proposition 3.4]{LZ}}]\label{corollary for T CM}
Let $A$ be a non-negatively graded Gorenstein algebra such that $A_0$ has finite global dimension. If $T= \bigoplus_{i\geq0} A(i)_{\leq0}$ is graded Gorenstein projective, then $T$ is a tilting object in $D_{sg}(\mod^\Z(A))$.
\end{lemma}




%
%
\subsection{Singularity category for $\imath$quivers}

In this subsection, we shall describe $D_{sg}(\mod(\Lambda^{\imath}))$ (and equivalently, $\underline{\Gproj}(\Lambda^\imath)$) by using the triangulated orbit of $D^b(\K Q)$ \`a la Keller \cite{Ke2}. The main result is Theorem~\ref{thm:sigma} below.

We apply the general considerations in \S\ref{subsec:tilting} to $\Lambda^\imath$ as a non-negatively graded algebra by Corollary~\ref{cor:subquotient}. Proposition \ref{prop:projective module of lambdai} shows that $\Lambda^{\imath}(\ell)_{\leq 0}=\Lambda^{\imath}(\ell)$ for $\ell\geq1$, and thus $T= \bigoplus_{\ell\geq0} \Lambda^{\imath}(\ell)_{\leq0}\cong \Lambda^{\imath}_0$ in $D_{sg}(\mod^\Z(A))$. Denote by $\underline{T}= \Lambda^{\imath}_0$.  By Proposition~\ref{prop:projective module of lambdai} again, we obtain $\underline{T}=\iota ({}_{\K Q}\K Q)= {}_{\Lambda^\imath}(\K Q)$.
It follows from \eqref{eqn:projective resolution of kQ} that $\Omega(\underline{T})\cong \underline{T}(-1)$,
and then $\underline{T}$ is a Gorenstein projective $\Lambda^{\imath}$-module by Theorem \ref{theorem characterize of gorenstein property} by noting that $\Lambda^{\imath}$ is $1$-Gorenstein by Proposition \ref{proposition of 1-Gorenstein}. Also $T$ is Gorenstein projective.

\begin{proposition}
   \label{prop:tilting object}
$\underline{T}= \Lambda^{\imath}_0$ is a tilting object in $D_{sg}(\mod^\Z(\Lambda^{\imath}))$, and its (opposite) endomorphism algebra is isomorphic to $\K Q$. In particular, we have
\[
D_{sg}(\mod^\Z(\Lambda^{\imath}))\simeq D^b(\K Q).
\]
\end{proposition}

\begin{proof}
By Lemma~ \ref{corollary for T CM} and the discussions above, $\underline{T}$ is a tilting object in $D_{sg}(\mod^\Z(\Lambda^{\imath}))$. By Proposition \ref{prop:projective module of lambdai}, $\Lambda^\imath$ has Gorenstein parameter $1$, i.e., $\soc \Lambda^\imath\subset \Lambda^\imath_1$; cf. Corollary \ref{cor:subquotient}.
It follows from \cite[Theorem 3.5(ii)]{LZ} that
\[
\End_{D_{sg}(\mod^\Z(\Lambda^{\imath}))}(\underline{T})^{op}\cong\End_{\mod^\Z (\Lambda^\imath)}(\underline{T})^{op}\cong\End_{\Lambda^{\imath}}(\underline{T})^{op}
\]
and then
\[
\End_{D_{sg}(\mod^\Z(\Lambda^{\imath}))}(\underline{T})^{op}\cong \End_{\Lambda^{\imath}}(\K Q)^{op}\cong \K Q.
\]
The proposition follows now by \cite[Theorem 3.5(iii)]{LZ}. 
\end{proof}

The triangulated equivalence in Proposition~\ref{prop:tilting object}, denoted by $G$, is given by the composition of functors:
\begin{align}\label{eqn:G}
G:D^b(\K Q)\xrightarrow{\underline{T}\otimes_{\K Q}^\L-} D^b(\mod^\Z \Lambda^{\imath})\stackrel{\pi}{\longrightarrow}  D_{sg}(\mod^\Z \Lambda^{\imath}).
\end{align}
On the other hand, $\underline{T}$ is isomorphic to $\K Q$ as a $\Lambda^{\imath}$-$\K Q$-bimodule, so $( \underline{T}\otimes_{\K Q}-)\simeq \iota$, where $\iota$ is defined in \eqref{eqn:rigt adjoint}. So $G$ is equivalent to the composition
\[
D^b(\K Q)\xrightarrow{D^b(\iota)} D^b(\mod^\Z \Lambda^{\imath})\stackrel{\pi}{\longrightarrow}  D_{sg}(\mod^\Z \Lambda^{\imath}),
\]
where $D^b(\iota)$ is the derived functor of $\iota$ since $\iota$ is exact.

The automorphism $\btau$ of $\K Q$ induces an automorphism $\ov{\btau}$ of $\Lambda^\imath$. Then $\ov{\btau}$ induces an automorphism of $\mod(\Lambda^\imath)$.

\begin{remark}\label{rem:action of sigma}
The automorphism $\btau$ of $\K Q$ induces an automorphism, denoted again by $\btau$, of $\mod(\K Q)$.
The restriction of $\ov{\btau}$ to the subcategory $\mod(\K Q)$ of $\mod(\Lambda^\imath)$ coincides with $\btau$, i.e., $\ov{\btau}|_{\mod(\K Q)}=\btau$.
\end{remark}

Set $\Gamma=\K Q$. Then $\ov{\btau}$ induces an automorphism of $\mod(\Lambda^\imath\otimes \Gamma^{op})$, that is, for any $\Lambda^\imath$-$\Gamma$-bimodule $X$, ${}^{\ov{\btau}} X$ is defined to be the $\Lambda^\imath$-$\Gamma$-bimodule with its left $\Lambda^\imath$-module structure twisted by $\ov{\btau}$. Similarly, $\btau$ induces an automorphism of $\mod(\Gamma\otimes \Gamma^{op})$.
%
It is natural to view $\Lambda^{\imath}$ as a $\Lambda^{\imath}$-$\Gamma$-bimodule. Recall that $\underline{T}$ is isomorphic to $\Gamma$ as $\Lambda^{\imath}$-$\Gamma$-bimodule. By \eqref{eqn:projective resolution of kQ}, we have the following exact sequence in $\mod^\Z(\Lambda^{\imath}\otimes \Gamma^{op})$:
\begin{align}\label{eqn:U}
0\longrightarrow U\longrightarrow \Lambda^{\imath} (2)\longrightarrow \underline{T}(2)\longrightarrow0.
\end{align}
In particular, $U$ is isomorphic to $({}^{\ov{\btau}}\Gamma)(1)$ as $\Lambda^{\imath}$-$\Gamma$-bimodule.
Furthermore, we obtain the following exact sequence in $\mod^\Z(\Lambda^{\imath}\otimes \Gamma^{op})$:
\begin{align}
0\longrightarrow \underline{T}\longrightarrow ({{}^{\ov{\btau}} \Lambda^{\imath}})(1)\longrightarrow U\longrightarrow0.
\end{align}

By applying $\res: \mod(\Lambda^\imath)\rightarrow\mod(\Gamma)$, $U$ can be viewed as a $\Gamma$-$\Gamma$-bimodule with the left $\Gamma$-module structure induced by its left $\Lambda^\imath$-module structure. Correspondingly, since $\underline{T}$ is isomorphic to $\Gamma$ as $\Lambda^{\imath}$-$\Gamma$-bimodule, $\underline{T}$ is a $\Gamma$-$\Gamma$-bimodule. Then we obtain the isomorphisms $U\otimes_{\Gamma}^\L \underline{T}\simeq U$ and $\underline{T}\otimes_\Gamma^\L \underline{T}\cong \underline{T}$ in $D^b(\mod^\Z(\Lambda^{\imath}\otimes \Gamma^{op}))$.

From the above observation, similar to \cite[Proposition 3.9, Theorem 3.11]{Lu}, we obtain the following result.
Recall the shift functor $\Sigma$ of $D^b(\K Q)$ and the functor $G$ from \eqref{eqn:G}.

\begin{proposition}
    \label{lemma euivalent of functors}
We have
\[
(2)\circ G\simeq G\circ \Sigma^2.
\]
In particular, $D_{sg}(\mod(\Lambda))\simeq D_{sg}(\mod^{\Z/2}(\Lambda^{\imath}))\simeq D^b(\K Q)/\Sigma^2$ as triangulated categories.
\end{proposition}

\begin{proof}
The proof that $(2)\circ G\simeq G\circ \Sigma^2$ and $D_{sg}(\mod^{\Z/2}(\Lambda^{\imath}))\simeq D^b(\K Q)/\Sigma^2$ is completely similar to that of \cite[Proposition 3.9, Theorem 3.11]{Lu}, and hence is omitted. The equivalence $D_{sg}(\mod(\Lambda))\simeq D_{sg}(\mod^{\Z/2}(\Lambda^{\imath}))$ follows by noting that $\mod(\Lambda)\cong \mod^{\Z/2}(\Lambda^{\imath})$.
\end{proof}

The degree shift functor $(1)$ is a triangulated autoequivalence of $D_{sg}(\mod^\Z(\Lambda^{\imath}))$. Under the equivalence functor $G: D^b(\K Q)\rightarrow D_{sg}(\mod^\Z(\Lambda^{\imath}))$, the shift functor $(1)$ induces a triangulated autoequivalence of $D^b(\K Q)$, which is denoted by $F_{{\btau}^{\sharp}}$. In particular, by definition and Proposition~ \ref{lemma euivalent of functors}, we obtain that $(F_{{\btau}^{\sharp}})^2\simeq \Sigma^2$.

\begin{lemma}\label{lem:triangulated orbit category}
The  orbit category $D^b(\K Q)/F_{{\btau}^{\sharp}}$ is a triangulated orbit category \`a la Keller \cite{Ke2}.
\end{lemma}

\begin{proof}
We make the following observations.

First, for any indecomposable object $Y$ of $D^b(kQ)$, it follows from $(F_{{\btau}^{\sharp}})^2\simeq \Sigma^2$ that only finitely many $(F_{{\btau}^{\sharp}})^i Y$, $i\in\Z$, lie in $\mod(\K Q)$.

Secondly, for the $F_{{\btau}^{\sharp}}$-orbit $\co_V$ of each indecomposable object $V$ in $D^b(\K Q)$, by using  $(F_{{\btau}^{\sharp}})^2\simeq \Sigma^2$ again, there exists some indecomposable object $Y$ in $\mod(\K Q)$ such that $Y\in \co_V$ or $\Sigma Y\in \co_V$.

The assertion follows from these 2 observations, by the theorem in \cite[\S4]{Ke2}.
\end{proof}

Let $\widehat{\btau}$ be the triangulated auto-equivalence of $D^b(\K Q)$ induced by $\btau$; cf. Remark~ \ref{rem:action of sigma}.
Now we can formulate the main result in this subsection.

\begin{theorem}\label{thm:sigma}
Let $(Q, \btau)$ be an $\imath$quiver. Then the following equivalences of categories hold:
\[
\underline{\Gproj}(\Lambda^{\imath})\simeq D_{sg}(\mod(\Lambda^{\imath}))\simeq D^b(\K Q)/\Sigma \circ \widehat{\btau}.
\]
\end{theorem}

\begin{proof}
First, Theorem \ref{thm:Buchweitz-Happel} shows that $\underline{\Gproj}(\Lambda^{\imath})\simeq D_{sg}(\mod(\Lambda^{\imath}))$ since $\Lambda^{\imath}$ is $1$-Gorenstein.
Similar to the proof of \cite[Theorem 6.2]{Ya}, by Lemma \ref{lem:triangulated orbit category}, we have $D_{sg}(\mod^{\Z/1\Z}(\Lambda^{\imath}))\simeq D^b(\K Q)/F_{{\btau}^{\sharp}}$. As $\mod^{\Z/1\Z}(\Lambda^{\imath})\cong \mod(\Lambda^{\imath})$, we obtain that $D_{sg}(\mod(\Lambda^{\imath}))\simeq D^b(\K Q)/F_{{\btau}^{\sharp}}$.

It remains to prove that $\Sigma\circ \widehat{\btau}\simeq F_{{\btau}^{\sharp}}$.
Let $\Gamma=\K Q$. Recall that $U$ is defined in \eqref{eqn:U}. Viewing $U$ as a $\Gamma$-$\Gamma$-bimodule, let
$$F=U\otimes^\L_\Gamma-:D^b(\mod \Gamma)\longrightarrow D^b(\mod \Gamma).$$
Similar to the proof of \cite[Proposition 6.1]{Ya}, we obtain that
$G\circ (\Sigma\circ F)\simeq (1)\circ G$. By definition, $F_{{\btau}^{\sharp}}\simeq \Sigma\circ F$.
Note that $U$ is isomorphic to $({}^{\ov{\btau}}\Gamma)(1)$ as $\Lambda^{\imath}$-$\Gamma$-bimodules. Then
$_\Gamma U_\Gamma$ is isomorphic to ${}^\btau\Gamma$, which implies that $(_\Gamma U\otimes_\Gamma -)\simeq\btau$. Therefore,
$F$ is isomorphic to $\widehat{\btau}$ and so $\Sigma\circ \widehat{\btau}\simeq \Sigma\circ F\simeq F_{{\btau}^{\sharp}}$.
\end{proof}

Recall that the forgetful functor $\mod^\Z(\Lambda^\imath)\rightarrow \mod(\Lambda^\imath)$ is not dense in general. We have the following corollary of Theorem \ref{thm:sigma}.
\begin{corollary}
\label{cor:dense of forgetful}
The forgetful functor
$\Gproj^\Z (\Lambda^{\imath})\longrightarrow \Gproj(\Lambda^{\imath})$ is dense.
\end{corollary}

\begin{proof}
By definition, we have $\underline{\Gproj}^\Z (\Lambda^\imath)/(1)\simeq D_{sg}(\mod^\Z (\Lambda^\imath))/(1)\simeq D^b(\K Q)/F_{{\btau}^{\sharp}}$ as additive categories. It follows from Theorem \ref{thm:sigma} that $D^b(\K Q)/F_{{\btau}^{\sharp}}\simeq \underline{\Gproj} (\Lambda^\imath)$.
So $\underline{\Gproj} (\Lambda^\imath)\simeq \underline{\Gproj}^\Z (\Lambda^\imath)/(1)$, which implies the forgetful functor $\underline{\Gproj}^\Z (\Lambda^{\imath})\rightarrow \underline{\Gproj}(\Lambda^{\imath})$ is dense. The result follows since projective $\Lambda^\imath$-modules are gradable.
\end{proof}

\begin{remark}
\label{rem:bounded}
Using Proposition \ref{proposition of 1-Gorenstein}, we obtain that $M\in \Gproj^\Z (\Lambda^{\imath})$ if and only if its restriction to $\K Q$ is graded projective. So
$\Gproj^\Z(\Lambda^\imath) \cong \cc^b(\proj(\K Q))$ by noting that $\mod^\Z(\Lambda^\imath)\cong \cc^b(\mod(\K Q))$.
\end{remark}

By identifying $\mod^{\Z/2}(\Lambda^\imath)$ with $\mod(\Lambda)$, the pushdown functor $\Pd: \mod(\Lambda)\longrightarrow \mod(\Lambda^\imath)$ is the forgetful functor. From the proof of Corollary \ref{cor:dense of forgetful}, $\Pd$ induces a Galois covering
\begin{align}
\label{FGproj}
\Pd: \Gproj(\Lambda)\longrightarrow \Gproj(\Lambda^\imath)
\end{align}
by noting that $\Gproj(\Lambda)\cong\Gproj^{\Z/2}(\Lambda^\imath)$.

The following corollary will be useful in providing a concrete basis for the semi-derived Ringel-Hall algebras of $\La^\imath$ later (see Theorem~\ref{thm:utMHbasis}).

\begin{corollary}
   \label{corollary for stalk complexes}
For any $M\in D_{sg}(\mod(\Lambda^{\imath}))$, there exists a unique (up to isomorphism) module $N\in \mod(\K Q)\subseteq \mod(\Lambda^{\imath})$ such that
$M\cong N$ in $D_{sg}(\mod(\Lambda^{\imath}))$.
In particular, we have $\Ind (\mod(\K Q))=\Ind D_{sg}(\mod(\Lambda^\imath)).$
\end{corollary}

\begin{proof}
Without loss of generality, we assume that $M$ is indecomposable.

It follows from  Theorem \ref{thm:sigma} that the triangulated equivalence functor
\[
G: D^b(\K Q)\xrightarrow{D^b(\iota)} D^b(\mod^\Z\Lambda^\imath)\stackrel{\pi}{\longrightarrow}  D_{sg}(\mod^\Z (\Lambda^\imath))
\]
induces an equivalence $\tilde{G}:D^b(\K Q)/\Sigma\circ \widehat{\btau}\simeq D_{sg}(\mod(\Lambda^\imath))$. Note that $\mod(\K Q)\subseteq \mod(\Lambda^{\imath})$ is induced by $\iota$. So it is equivalent to proving that there exists a unique $N\in\mod(\K Q)$ such that $\tilde{G}(N)=M$ in $D_{sg}(\mod(\Lambda^\imath))$.

From above, there exists an indecomposable complex $Y\in D^b(\K Q)$ such that $\tilde{G}(Y)=M$. Happel's Theorem shows that $Y$ is a stalk complex, i.e., $Y=\Sigma^\ell X$ for some $X\in\mod(\K Q)$ and $\ell\in\Z$. So we have $N:= (\Sigma\circ\widehat{\btau})^{-\ell}Y \in\mod(\K Q)$, and
\begin{align}
N \cong \left\{ \begin{array}{ccc}
\btau(X),& \text{ if }2\nmid \ell,\\
X,&\text{ if }2\mid \ell.
\end{array}\right.
\end{align}
Clearly, we have $\tilde{G}(N)=M$.

The uniqueness follows by noting that, for any $N\in\mod(\K Q)$, $(\Sigma\circ \widehat{\btau})^\ell(N)\in\mod(\K Q)$ if and only if $\ell=0$.
\end{proof}

\begin{remark}
   \label{rem:sing12}
Recall by Example~\ref{example 2}(c) that $\Lambda$ itself is an algebra arising from an $\imath$quiver. Hence, for any $M\in D_{sg}(\mod(\Lambda))$, there exists a unique (up to isomorphism) module $N\in \mod(\K Q\times \K Q')\subseteq \mod(\Lambda)$ such that
$M\cong N$ in $D_{sg}(\mod(\Lambda))$. In particular, the pullback functor $\iota:\mod(\K Q\times \K Q')\rightarrow \mod(\Lambda)$ in \eqref{eqn:rigt adjoint} induces
$$\Ind (\mod(\K Q))\sqcup\Ind (\mod(\K Q'))=\Ind D_{sg}(\mod(\Lambda)).$$
\end{remark}

\begin{remark}
For $\btau=\Id$, Corollary \ref{corollary for stalk complexes} recovers a result in \cite{RZ}.
\end{remark}

\section{Hall algebras for $\imath$quivers}
  \label{sec:Hall}

We take $\K=\F_q$ in the remainder of this paper.
By applying the procedure from Appendix~ \ref{subsec:MRH for 1-Gor}, we define the semi-derived Ringel-Hall algebra $\utMH$ associated to an acyclic $\imath$quiver, since $\Lambda^\imath$ is a $1$-Gorenstein algebra. We will introduce a twisted version of $\utMH$, denoted by $\tMH$ (which is referred  to as the Hall algebra of an $\imath$quiver, or an $\imath$Hall algebra for short).

\subsection{Euler forms}

Recall that $\cp^{\leq 1}(\Lambda^\imath) =\cp^{<\infty}(\Lambda^\imath)$ is the subcategory of $\Lambda^\imath$-modules of finite projective dimensions.

\begin{lemma}
  \label{lemma:isomorphic of Grothendieck groups}
We have the following isomorphism of abelian groups
\begin{align}\label{dfn:phi}
\phi: K_0(\mod(\K Q))\longrightarrow K_0 \big(\cp^{\leq 1}(\Lambda^\imath) \big),\qquad \widehat{S}_i \mapsto \widehat{\E}_i, \; \forall \,i\in Q_0.
\end{align}
\end{lemma}

\begin{proof}
Recall from Corollary~\ref{cor:subquotient} that $kQ$ is a quotient algebra of $\Lambda^\imath$.
The following are well known:
\begin{itemize}
\item $K_0(\proj (\Lambda^\imath))\cong K_0(\cp^{\leq 1}(\Lambda^\imath))$;
\item $K_0(\mod (\K Q))$ is a free abelian group with $\{\widehat{S}_i\mid i\in Q_0\}$ as a basis.
\end{itemize}
Then $\phi$ is well defined, and $K_0 \big(\cp^{\leq 1}(\Lambda^\imath) \big)$ is a free abelian group with $\{\widehat{(\Lambda^\imath) e_i}\mid i\in Q_0\}$ as a basis. By Corollary \ref{corollary locally projective modules} we conclude that
$\phi$ is surjective, and then it is an isomorphism by comparing the ranks.
\end{proof}

It follows from Proposition \ref{prop:projective module of lambdai} that $(kQ)e_i$ is a Gorenstein projective $\Lambda^\imath$-module; cf. \S\ref{subsec:Gproj}.

\begin{proposition}
   \label{prop:Grothendieck}
$K_0(\Gproj(\Lambda^\imath))$ is a free abelian group of rank $|Q_0|$, with $\{\widehat{\K Qe_i}\mid i\in Q_0\}$ as a basis.
\end{proposition}

\begin{proof}
Set $A=\K Q$, and $n=|Q_0|$ in this proof.

Since $\Lambda^\imath$ is $1$-Gorenstein, by Lemmas \ref{lemma perpendicular of P GP} and \ref{lemma approx of GP}, $(\cp^{\leq 1}(\Lambda^\imath),\Gproj(\Lambda^\imath))$ is a left complete Ext-orthogonal pair \`a la \cite{Wa}.
By \cite[Theorem 2.1]{Wa}, there is a short exact sequence of groups
$$0\longrightarrow K_0(\proj(\Lambda^\imath)) \longrightarrow K_0(\Gproj(\Lambda^{\imath}))\oplus K_0(\cp^{\leq 1}(\Lambda^\imath))\longrightarrow K_0(\mod(\Lambda^\imath))\longrightarrow0.$$
Here we use the fact $\cp^{\leq 1}(\Lambda^\imath)\cap \Gproj(\Lambda^\imath)=\proj(\Lambda^\imath)$.
It is well known that $K_0(\mod(\Lambda^\imath))$ and $K_0(\proj(\Lambda^\imath))$ are free abelian groups of rank $n$. Then
both $K_0(\Gproj(\Lambda^{\imath}))$ and $K_0(\cp^{\leq 1}(\ca))$ are free abelian groups, and the sum of their ranks equals to $2n$. Hence $K_0(\Gproj(\Lambda^\imath))$ is a free abelian group of rank $n$.

It follows by Proposition \ref{prop:tilting object} and its proof that $_{\Lambda^\imath}A$ is a Gorenstein projective $\Lambda^{\imath}$-module. Define $\bK:=\langle \widehat{Ae_i} \mid i\in Q_0\rangle$ to be the subgroup of $K_0(\Gproj(\Lambda^\imath))$. By Proposition \ref{prop:projective module of lambdai}, there exists a short exact sequence
$$0\longrightarrow Ae_{\btau i}\longrightarrow \Lambda^{\imath}e_i\longrightarrow Ae_i\longrightarrow0, \qquad \forall i\in Q_0.$$
So $\widehat{\Lambda^{\imath}\, e_i}\in \bK$ for each $i\in Q_0$.
By Proposition \ref{prop:tilting object}, $_{\Lambda^\imath}A$ is a tilting object in $\underline{\Gproj}^\Z( \Lambda^{\imath})$. Since the forgetful functor $\Gproj^\Z(\Lambda^\imath)\rightarrow \Gproj(\Lambda^\imath)$ is dense by Corollary~ \ref{cor:dense of forgetful}, $\underline{\Gproj}(\Lambda^{\imath})$ is generated by $Ae_i$, $i\in Q_0$. Then $\bK= K_0(\Gproj(\Lambda^\imath))$, and the proposition is proved.
\end{proof}

Following \eqref{left Euler form}--\eqref{right Euler form} in Appendix~\ref{subsection:Def of MRH}, we can define the Euler forms $\langle K,M\rangle =\langle L,M\rangle_{\Lambda^\imath}$ and $\langle M,K\rangle =\langle M,K\rangle_{\Lambda^\imath}$ for any $K\in\cp^{\leq1}(\Lambda^\imath)$, $M\in\mod(\Lambda^\imath)$. These forms descend to bilinear Euler forms on the Grothendieck groups:
\begin{eqnarray*}
\langle\cdot,\cdot\rangle: K_0(\cp^{\leq 1}(\Lambda^\imath))\times K_0(\mod(\Lambda^\imath))\longrightarrow \Z,
\\
\langle\cdot,\cdot\rangle: K_0(\mod(\Lambda^\imath))\times K_0(\cp^{\leq 1}(\Lambda^\imath))\longrightarrow \Z,
\end{eqnarray*}
such that
\begin{equation}
\langle \widehat{K},\widehat{M}\rangle=\langle K,M\rangle,\qquad\langle \widehat{M},\widehat{K}\rangle=\langle M,K\rangle,\qquad\forall\, K\in\cp^{\leq1}(\Lambda^\imath), M\in\mod(\Lambda^\imath).
\end{equation}
Denote by $\langle\cdot,\cdot\rangle_Q$ the Euler form of $\K Q$. Denote by $S_i$ the simple $kQ$-module (respectively, $\Lambda^{\imath}$-module) corresponding to vertex $i\in Q_0$ (respectively, $i\in\ov{Q}_0$).
Since $\btau$ is an involution of the quiver $Q$, we have
\begin{align}
  \label{eq:SSQ}
 \langle S_{\btau i}, S_{\btau j}\rangle_Q =\langle S_i,S_j\rangle_Q,
 \qquad
 \langle S_i,S_{\btau j}\rangle_Q =\langle S_{\btau i}, S_j\rangle_Q.
\end{align}

These 2 Euler forms are related via the restriction functor $\res:\mod(\Lambda^\imath)\rightarrow \mod(\K Q)$ as follows.

\begin{lemma}
   \label{lemma compatible of Euler form}
We have
\begin{itemize}
\item[(i)]
$\langle \E_i, M\rangle = \langle S_i,\res (M) \rangle_Q$ and $\langle M,\E_i\rangle =\langle \res(M), S_{\btau i} \rangle_Q$, for any $i\in Q_0$, $M\in\mod(\Lambda^{\imath})$;
\item[(ii)] $\langle M,N\rangle=\frac{1}{2}\langle \res(M),\res(N)\rangle_Q$, for any $M,N\in\cp^{\leq 1}(\Lambda^\imath)$.
\end{itemize}
\end{lemma}

\begin{proof}
(i) It suffices to prove the formulas for $M=S_j$.
Recall $\E_i$ has a projective resolution (\ref{eqn:projective resolution of E}). If $i\neq j$, then $\langle \E_i,S_j\rangle=-|\{ (\alpha \colon i\rightarrow j)\in Q_1\}|=\langle S_i,S_j\rangle_Q$;
If $i=j$, then $\langle \E_i,S_j\rangle=1=\langle S_i,S_j\rangle_Q$. So
$\langle \E_i, M\rangle = \langle S_i,\res (M) \rangle_Q$.

Dually, by using (\ref{eqn:injective resolution of E}), we obtain that $\langle S_j,\E_i\rangle =\langle S_j, S_{\btau i} \rangle_Q$.

(ii) By Corollary \ref{corollary locally projective modules}, it suffices to prove the result for $M=\E_i$ and $N=\E_j$,  for $i, j \in Q_0$.
Then by \eqref{eq:SSQ} and Part~(i) we have
\begin{align*}
\langle \res(\E_i), & \res(\E_j)\rangle_Q-2\langle \E_i,\E_j\rangle\\
&= \langle S_i\oplus S_{\btau i}, S_j\oplus S_{\btau j}\rangle_Q -2\langle S_i, \res(\E_j)\rangle_Q\\
&= \langle S_i, S_j\oplus S_{\btau j}\rangle_Q +\langle S_{\btau i}, S_j\oplus S_{\btau j}\rangle_Q -2\langle S_i,S_j\oplus S_{\btau j}\rangle_Q\\
&= \langle S_i, S_j\oplus S_{\btau j}\rangle_Q +\langle S_{i}, S_{\btau j}  \oplus S_j\rangle_Q -2\langle S_i,S_j\oplus S_{\btau j}\rangle_Q=0.
\end{align*}
The lemma is proved.
\end{proof}

\subsection{Semi-derived Ringel-Hall algebras for $\La^\imath$}

As $\Lambda^\imath$ is a $1$-Gorenstein algebra by Proposition~\ref{proposition of 1-Gorenstein}, following Appendix \ref{app:A} we define the semi-derived Ringel-Hall algebra
\[
\utMH:= \utM(\mod(\Lambda^\imath))
\]
associated to the category $\mod(\Lambda^\imath)$, in which the multiplication is denoted by $\diamond$. More precisely, $\utM(\Lambda^\imath)$ is defined to be the localization ring $(\ch(\Lambda^\imath)/I) [\cs_{\Lambda^\imath}^{-1}]$, where $\cs_{\Lambda^\imath}=\{a[K]\mid a \in \Q^\times, K \in\cp^{\leq1}(\Lambda^\imath)\}$ and the ideal $I$ is generated by $[L]-[K\oplus M]$ associated to any short exact sequence
\begin{equation*}
 0 \longrightarrow K \longrightarrow L \longrightarrow M \longrightarrow 0
\end{equation*}
in $\mod(\Lambda^\imath)$ with $K\in\cp^{\leq1}(\Lambda^\imath)$.

By Lemma \ref{lem:bimodule of MRH}, for any $M\in \mod(\Lambda^\imath)$ and $K\in\cp^{\leq 1}(\Lambda^\imath)$, we have in $\utMH$
\begin{align}
[K]\diamond[M] &=q^{-\langle K,M\rangle}[K\oplus M],\\
[M]\diamond[K] &=q^{-\langle M,K\rangle}[K\oplus M].
\end{align}

For any $\alpha\in K_0(\mod(\K Q))$, by Lemma \ref{lemma:isomorphic of Grothendieck groups}, there exist $A,B\in\cp^{\leq 1}(\Lambda^\imath)$ such that $\phi(\alpha) =\widehat{A}-\widehat{B} \in K_0 \big(\cp^{\leq 1}(\Lambda^\imath) \big)$. Define
\[
\E_\alpha:=q^{-\langle \phi(\alpha),\widehat{B}\rangle} [A]\diamond [B]^{-1} \in \utMH,
\]
which is independent of choices of $A$ and $B$ (this is similar to \S\ref{subsec:MRH for 1-Gor} or \cite[\S3.4]{LP}). In particular, we have
 \[
 \E_{\widehat{S_i}}=[\E_i], \qquad \forall i\in Q_0.
 \]

We define a partial order $\leq$ on the Grothendieck group $K_0(\mod(\K Q))$, namely,
\[
\alpha\leq \beta \; \text{  if and only if }\; \beta-\alpha\in K_0^+(\mod(\K Q)),
\]
where
$K_0^+(\mod(\K Q))\subseteq K_0(\mod(\K Q))$ is the positive cone of the Grothendieck group, consisting of classes of objects of $\mod(\K Q)$.
As $K_0(\mod(\K Q))$ and $K_0(\mod(\Lambda^\imath))$ are free abelian group with basis given by $\{\widehat{S_i}\mid i\in Q_0\}$, it is natural to identify them. In this way, we obtain a partial order $\leq$ on $K_0(\mod(\Lambda^\imath))$.

\begin{lemma}
   \label{lem:another basis of MRH 1}
For any $L\in\mod(\Lambda^{\imath})$, there exist (unique up to isomorphism) $X\in\mod (\K Q)\subseteq \mod (\Lambda^{\imath})$ and $\alpha\in K_0^+(\mod(\K Q))$ such that $[L]=q^{\langle \widehat{X}, \phi(\alpha)\rangle} [X]\diamond \E_{\alpha}$ in $\cs\cd\ch(\Lambda^\imath)$.
\end{lemma}

\begin{proof}
Let $L\in\mod(\Lambda^{\imath})$. As $\Lambda^{\imath}$ is $1$-Gorenstein,   by Lemma~\ref{lemma approx of GP} there exists a short exact sequence
\begin{align}
\label{eqn:resolution L}
0\longrightarrow P_L\longrightarrow G_L\longrightarrow L\longrightarrow0
\end{align}
with $G_L\in\Gproj(\Lambda^{\imath})$, $P_L\in \proj(\Lambda^{\imath})$.
Then $[L]\diamond [P_L]=q^{-\langle L, P_L\rangle } [G_L]$.

Denote by $G_L=P\oplus G$, where $P$ is the maximal projective direct summand of $G_L$. Then $[G]\diamond[P]=q^{-\langle G,P \rangle} [G_L]$.
By Corollary \ref{corollary for stalk complexes} there exists $X\in\mod(\K Q)$ such that
$G\cong X$ in $D_{sg}(\mod(\Lambda^{\imath}))$. For $X$, again there exists a short exact sequence
\begin{align}
\label{eqn:resolution X}
0\longrightarrow P_X\longrightarrow G_X\longrightarrow X\longrightarrow0,
\end{align}
with $G_X\in\Gproj(\Lambda^{\imath})$, $P_X\in \proj(\Lambda^{\imath})$. Note that $G_X\cong G$ in $\underline{\Gproj}(\Lambda^{\imath})$. Then $G$ is a direct summand of $G_X$ since $G$ has no projective direct summand by our assumption. Thus there exists a projective $\Lambda^\imath$-module $P'$ such that
$G_X\cong G\oplus P'$. Then
\begin{align*}
[L]&=q^{-\langle \widehat{L}, \widehat{P_L}\rangle } [G_L]\diamond [P_L]^{-1}
\\
&= q^{\langle \widehat{G},\widehat{P} \rangle -\langle \widehat{L}, \widehat{P_L}\rangle} [G]\diamond[P]\diamond [P_L]^{-1}
\\
&= q^{\langle \widehat{G},\widehat{P} \rangle -\langle \widehat{L}, \widehat{P_L}\rangle} q^{\langle \widehat{P}-\widehat{P_L},\widehat{P_L} \rangle} [G]\diamond \E_{\beta}
\\
&=q^{\langle \widehat{G},\widehat{P} \rangle} q^{ -\langle \widehat{L} -\widehat{P}+\widehat{P_L}, \widehat{P_L}\rangle}  [G]\diamond \E_{\beta}
\\
&= q^{\langle \widehat{G},\widehat{P}-\widehat{P_L}\rangle}  [G]\diamond \E_{\beta}
\end{align*}
by noting that $\widehat{L} -\widehat{P}+\widehat{P_L}=\widehat{G}$ and setting $\beta=\phi^{-1}(\widehat{P}-\widehat{P_L})$. Therefore, we have
\begin{align}
[L]&= q^{\langle \widehat{G},\widehat{P}-\widehat{P_L}\rangle} q^{\langle \widehat{G_X}-\widehat{P'},\widehat{P'}\rangle} [G_X]\diamond [P']^{-1}\diamond \E_\beta
  \label{eq:XL}  \\
&= q^{\langle \widehat{G},\widehat{P}-\widehat{P_L}\rangle} q^{\langle \widehat{G_X}-\widehat{P'},\widehat{P'}\rangle} q^{\langle \widehat{P'},\widehat{P}-\widehat{P_L}-\widehat{P'}\rangle}[G_X]\diamond \E_{\beta-\phi^{-1}(\widehat{P'})}
  \notag \\
&=q^{\langle \widehat{G_X},\widehat{P}-\widehat{P_L}+\widehat{P'} \rangle} [G_X]\diamond  \E_{\beta-\phi^{-1}(\widehat{P'})}
  \notag \\
&=q^{\langle \widehat{G_X},\widehat{P}-\widehat{P_L}+\widehat{P'} \rangle} q^{\langle \widehat{X},\widehat{P_X}\rangle} [X]\diamond [P_X]\diamond  \E_{\beta-\phi^{-1}(\widehat{P'})}
   \notag \\
&=q^{\langle \widehat{X}, \widehat{P_X}+ \widehat{P}-\widehat{P_L}-\widehat{P'}\rangle} [X]\diamond \E_{\alpha}=q^{\langle \widehat{X}, \phi(\alpha)\rangle} [X]\diamond \E_{\alpha},
   \notag
\end{align}
where $\alpha=\phi^{-1}(\widehat{P_X}+ \widehat{P}-\widehat{P_L}-\widehat{P'})$.

The uniqueness follows from Corollary \ref{corollary for stalk complexes} by noting that $X\cong L$ in $D_{sg}(\mod(\Lambda^{\imath}))$.

Finally, it remains to prove that $\alpha\in K_0^+(\mod(k Q))$, or equivalently, $\widehat{X}\leq \widehat{L}$ in $K_0(\mod(\Lambda^\imath))$, thanks to \eqref{eq:XL}.

Recall from \eqref{eq:H} that $\BH=\bigoplus_{i\in\I_\btau} \BH_i$, and $\res_\BH:\mod(\Lambda^\imath)\longrightarrow \mod(\BH)$ is the restriction functor. For $i\in\I_\btau$, we note
\begin{align*}
\BH_i =\left\{ \begin{array}{ll}
R_2 & \text{ if }\btau i\neq i,\\
R_1 & \text{ if }\btau i=i,\end{array} \right.
\qquad
\mod(\BH_i)\cong\left\{ \begin{array}{ll}
\cc_{\Z/2}(\mod(k)) & \text{ if }\btau i\neq i,\\
\cc_{\Z/1}(\mod(k)) & \text{ if }\btau i=i.\end{array} \right.
\end{align*}
So any $V\in \mod(\BH_i)$ can be viewed as a graded complex, and let $H ^\bullet(V)\in \mod(\BH_i)$ be the cohomology of $V$. Note that $H ^{2a+\epsilon}(V)=H^\epsilon(V)$ for $a \in\Z$ and $\epsilon=0,1$ if $\btau i\neq i$; $H^a(V)=H^0(V)$ for any $a \in\Z$ if $\btau i=i$. For any $V\in\mod(\BH_i)$, define
\begin{align*}
H(V):=\left\{ \begin{array}{ll} H^0(V)\oplus H^1(V)&
\text{ if }\btau i\neq i,\\
H^0(V) &\text{ if }\btau i=i.\end{array}\right.
\end{align*}
The definition of $H(V)$ can be extended via a direct sum decomposition to any $V\in\mod(\BH)$ since $\BH=\bigoplus_{i\in\I_\btau} \BH_i$.

Clearly, $H^\bullet(\E_i)=0$ for any $i\in\I$, and then $H^\bullet(\res_\BH(M))=0$ for any $M\in \cp^{\leq 1}(\Lambda^\imath)$ by Corollary \ref{corollary locally projective modules}.

So $H^\bullet(G_L)=H^\bullet(G)=H^\bullet(G_X)$ by noting that $H^\bullet(P)=0=H^\bullet(P')$.
After applying the exact functor $\res_\BH$ to the short exact sequences \eqref{eqn:resolution L}--\eqref{eqn:resolution X}, the induced long exact sequences on cohomologies give us
\begin{align*}
&H^\bullet(\res_\BH(L))= H^\bullet(\res_\BH(G_L))=H^\bullet(\res_\BH(G))
= H^\bullet(\res_\BH(G_X)) =H^\bullet(\res_\BH(X)),
\end{align*}
thanks to $P_L,P_X\in\proj(\Lambda^\imath)$.
Then we have $H(\res_\BH(L))=H(\res_\BH(X))=\res_\BH(X)$ thanks to $X\in\mod(kQ)$.

On the other hand, $\res_\BH$ induces an isomorphism and an identification $K_0(\mod(\Lambda^\imath)) =K_0(\mod(\BH))$.
As $\widehat{H(\res_\BH(L))}\leq \widehat{L}$ by construction, we have established
$\widehat{X}\leq \widehat{L}$.
\end{proof}

\subsection{A Hall basis}

\begin{theorem}
\label{thm:utMHbasis}
The algebra $\utMH$ has a (Hall) basis given by
\begin{equation}
  \label{eq:Hall basis}
\big\{ [X]\diamond \E_\alpha ~\big |~ [X]\in\Iso(\mod(\K Q))\subseteq \Iso(\mod(\Lambda^{\imath})), \alpha\in K_0(\mod(\K Q)) \big\}.
\end{equation}
\end{theorem}

\begin{proof}
Denote by
\begin{align}
  \label{eq:GProjnp}
\Gproj^{\np}(\Lambda^\imath)
= & \text{ the smallest subcategory of $\Gproj (\Lambda^\imath)$ formed by all Gorenstein}
\\
&\text{ projective modules without any projective summands}.
\notag
\end{align}
We have by Lemma \ref{lemma:isomorphic of Grothendieck groups} that
$K_0(\mod(\K Q))\cong K_0(\cp^{\leq 1}(\Lambda^{\imath}))$. With Lemma~ \ref{cor:basis of MRH as torus-module 1-Gor}, this implies that $\utMH$ has a basis given by $[M]\diamond\E_\alpha$, where $[M]\in \Iso(\Gproj^{\np}( \Lambda^{\imath}))$ and $\alpha\in  K_0(\mod(\K Q))$.
For any $M\in \Gproj^{\np}( \Lambda^{\imath})$, by Corollary \ref{corollary for stalk complexes} there exists a unique (up to isomorphism) $\K Q$-module $X_M$ such that $M\cong X_M$ in $D_{sg}(\mod(\Lambda^{\imath}))$. Then by Lemma~ \ref{lem:another basis of MRH 1}, we have $[M]=q^{\langle \widehat{X_M}, \phi(\beta_M)\rangle} [X_M]\diamond \E_{\beta_M}$ for some $\beta_M\in K_0(\mod(\K Q))$. So $\utMH$ is spanned by the set \eqref{eq:Hall basis}.

On the other hand, for any $[X]\in \Iso(\mod(\K Q))$, there exists a unique isoclass $[M_X]\in  \Iso(\Gproj^{\np}( \Lambda^{\imath}))$ such that $X\cong M_X$ in $D_{sg}(\mod(\Lambda^{\imath}))$. It follows by Lemma~ \ref{lem:another basis of MRH 1} and its proof that there exists a unique $\alpha_X\in  K_0(\mod(\K Q))$ with $[X]=q^{\langle\widehat{M_X},\phi(\alpha_X)\rangle}[M_X]\diamond \E_{\alpha_X}$. The theorem follows.
\end{proof}

By Definition~\ref{def:torus} and its subsequent discussions, the quantum torus
\[
\T(\Lambda^\imath) :=\T(\mod(\Lambda^\imath))
\]
is the group algebra of $K_0(\cp^{\leq 1}(\Lambda^\imath))$ with its multiplication twisted by  $q^{-\langle \cdot,\cdot \rangle}$. By the isomorphism $\phi: K_0(\mod(\K Q))\stackrel{\simeq}{\rightarrow} K_0 \big(\cp^{\leq 1}(\Lambda^\imath) \big)$ from \eqref{dfn:phi}, $\T(\Lambda^\imath)$ is generated by $\E_\alpha$, for $\alpha\in K_0(\mod(\K Q))$, where
\begin{equation}
  \label{eq:EE2}
\E_\alpha\diamond \E_\beta= q^{-\langle \phi(\alpha),\phi(\beta)\rangle} \E_{\alpha+\beta}.
\end{equation}
Then $\T(\Lambda^\imath)$ is a subalgebra of $\utMH$, which is isomorphic to $\cs\cd\ch(\cp^{\leq 1}(\Lambda^\imath))$.
The following is immediate from Theorem \ref{theorem basis of MRH 1-Gor} and Theorem \ref{thm:utMHbasis}.

\begin{corollary}\label{cor:basis of MRH as torus-module}
$\utMH$ is a free right (respectively, left) $\T(\Lambda^\imath)$-module, with a basis given by $\Iso(\mod(\K Q))$.
\end{corollary}
\subsection{Hall algebras for $\imath$quivers}
\label{subsec:MH}

Recall
\[
\sqq =\sqrt{q}.
\]
Via the restriction functor $\res: \mod(\Lambda^{\imath})\rightarrow\mod (\K Q)$, we define the twisted semi-derived Ringel-Hall algebra $\tMH$ to be the $\Q(\sqq)$-algebra on the same vector space as $\utMH$ with twisted multiplication given by
\begin{align}
   \label{eqn:twsited multiplication}
[M]* [N] =\sqq^{\langle \res(M),\res(N)\rangle_Q} [M]\diamond[N].
\end{align}
We shall refer to this algebra as the {\em Hall algebra associated to $\La^\imath$} or {\em Hall algebra associated to the $\imath$quiver $(Q, \btau)$}; we call it an {\em $\imath$Hall algebra} for short.

\begin{lemma}
    \label{lemma multiplcation of locally projective modules}
For any $M,N\in \cp^{\leq 1}(\Lambda^\imath)$, we have $[M]*[N]=[M\oplus N]$ in $\tMH$. In particular, for any $\alpha,\beta\in K_0(\mod(kQ))$, we have
\begin{equation}  \label{eq:EE}
  \E_\alpha*\E_\beta=\E_{\alpha+\beta}.
\end{equation}
\end{lemma}

\begin{proof}
Follows from a comparison of \eqref{eq:EE2} and \eqref{eqn:twsited multiplication} using Lemma~ \ref{lemma compatible of Euler form}(ii).
\end{proof}

Let $\tTL$ be the subalgebra of $\tMH$ generated by $\E_\alpha$, $\alpha\in K_0(\mod(\K Q))$. Then $\{\E_\alpha\mid\alpha\in K_0(\mod(\K Q))\}$ is a basis of $\tTL$ whose multiplication is given by \eqref{eq:EE}. Hence we have obtained the following.

\begin{lemma}
  \label{lem:torus}
  The twisted quantum torus $\tTL$ is a Laurent polynomial algebra generated by $[\E_i]$, for $i\in \I$.
\end{lemma}

The following is a variant of Theorem \ref{thm:utMHbasis} and Corollary~\ref{cor:basis of MRH as torus-module}, and it follows from the definition of $\tMH$.

\begin{proposition}
The $\imath$Hall algebra $\tMH$ has a (Hall) basis given by
\begin{equation*}
\big\{ [X] * \E_\alpha ~\big |~ [X]\in\Iso(\mod(\K Q))\subseteq \Iso(\mod(\Lambda^{\imath})), \alpha\in K_0(\mod(\K Q)) \big\}.
\end{equation*}
Moreover, $\tMH$ is free as a right (respectively, left) $\tTL$-module, with a basis given by $[X]$ in $\Iso(\mod(\K Q))\subseteq \Iso(\mod(\Lambda^{\imath}))$.
\end{proposition}

The following proposition will be useful in defining a reduced version of $\tMH$.

\begin{proposition}
   \label{Prop:centralMH}
Let $i\in Q_0$. Then $[\E_i]$ is central in $\tMH$ if $\btau i=i$, and
$[\E_i]*[\E_{\btau i}]$ is central in $\tMH$ if $\btau i\neq i$.
\end{proposition}

\begin{proof}
Assume first $\btau i=i$. Then for any $M\in \mod(\Lambda^{\imath})$,  we have by Lemma \ref{lemma compatible of Euler form} that
\begin{align*}
[\E_i]*[M]&= {\sqq}^{\langle \res(\E_i),\res(M)\rangle_Q} [\E_i]\diamond[M]\\
&= {\sqq}^{\langle \res(\E_i),\res(M)\rangle_Q}q^{-\langle \E_i,M\rangle}[\E_i\oplus M]\\
&= {\sqq}^{\langle \res(\E_i),\res(M)\rangle_Q}q^{ -\langle S_i, \res(M) \rangle_Q}[\E_i\oplus M]\\
&= {\sqq}^{2\langle S_i,\res(M)\rangle_Q}q^{ -\langle S_i, \res(M) \rangle_Q}[\E_i\oplus M]\\
&= [\E_i\oplus M].
\end{align*}
Similarly, we have
\begin{align*}
[M]*[\E_i]&= {\sqq}^{\langle \res(M),\res(\E_i)\rangle_Q} [M]\diamond[\E_i]\\
&= {\sqq}^{\langle \res(M), \res(\E_i)\rangle_Q}q^{-\langle M, \E_i\rangle}[M\oplus \E_i]\\
&= {\sqq}^{\langle \res(M), \res(\E_i)\rangle_Q}q^{ -\langle \res(M),S_i\rangle_Q}[M\oplus \E_i]\\
&= [M\oplus \E_i].
\end{align*}
Then $[\E_i]*[M]=[M]*[\E_i]$ and so $[\E_i]$ is central in $\tMH$.

Assume now $\btau i\neq i$. First, Lemma \ref{lemma multiplcation of locally projective modules} implies that
$[\E_i]*[\E_{\btau i}]= [\E_i\oplus \E_{\btau i}].$ Then for any $M\in \mod(\Lambda^{\imath})$ it follows from Lemma \ref{lemma compatible of Euler form}  that
\begin{align*}
([\E_i]*[\E_{\btau i}])*[M]&=[\E_i\oplus \E_{\btau i}]*[M]\\
&= {\sqq}^{\langle \res(\E_i)\oplus \res(\E_{\btau i}),\res(M)\rangle_Q} [\E_i\oplus \E_{\btau i}]\diamond[M]\\
&= {\sqq}^{\langle \res(\E_i)\oplus \res(\E_{\btau i}),\res(M)\rangle_Q}q^{-\langle \E_i\oplus \E_{\btau i},M\rangle}[\E_i\oplus \E_{\btau i}\oplus M]\\
&= {\sqq}^{\langle S_i\oplus S_{\btau i}\oplus S_{\btau i}\oplus S_i,\res(M)\rangle_Q}q^{-\langle S_i\oplus S_{\btau i}, \res(M) \rangle_Q}[\E_i\oplus \E_{\btau i}\oplus M]\\
&= [\E_i\oplus \E_{\btau i}\oplus M].
\end{align*}
Similarly, we have
\begin{align*}
[M]* ([\E_i]*[\E_{\btau i}])
&= [M]*[\E_i\oplus \E_{\btau i}]\\
&= {\sqq}^{\langle \res(M),\res(\E_i) \oplus \res(\E_{\btau i})\rangle_Q} [M]\diamond[\E_i\oplus \E_{\btau i}]\\
&= {\sqq}^{\langle \res(M), \res(\E_i)\oplus \res(\E_{\btau i}) \rangle_Q}q^{-\langle M, \E_i\oplus \E_{\btau i}\rangle_Q}[M\oplus \E_i\oplus \E_{\btau i}]\\
&= {\sqq}^{\langle \res(M), S_i\oplus S_{\btau i}\oplus S_{\btau i}\oplus S_i\rangle_Q}q^{ -\langle \res(M), S_{\btau i}\oplus S_{i}\rangle_Q}[M\oplus \E_i\oplus \E_{\btau i}]\\
&= [M\oplus \E_i\oplus \E_{\btau i}].
\end{align*}
Then $[\E_i]*[\E_{\btau i}]$ is central in $\tMH$.
\end{proof}

Inspired by \cite{Br}, we now define a reduced version of $\tMH$. The precise definition is motivated by Proposition~\ref{prop:QSP12} and it will feature in Theorem~\ref{theorem isomorphism of Ui and MRH}.

\begin{definition}
  \label{def:reducedHall}
Let $\bvs=(\vs_i)\in  (\Q(\bv)^\times)^{\I}$ be such that $\vs_i=\vs_{\btau i}$ for each $i\in \I$. The \emph{reduced Hall algebra associated to $(Q,\btau)$} (or {\em reduced $\imath$Hall algebra}), denoted by $\rMH$, is defined to be the quotient $\Q(\sqq)$-algebra of $\tMH$ by the ideal generated by the central elements
\begin{align}
\label{eqn: reduce 1}
[\E_i] +q \vs_i \; (\forall i\in \I \text{ with } \btau i=i), \text{ and }\; [\E_i]*[\E_{\btau i}] -\vs_i^2\; (\forall i\in \I \text{ with }\btau i\neq i).
\end{align}
\end{definition}

\subsection{Hall algebras for $\imath$subquivers}
\label{subsec:subquiver}

Let $(Q, \btau)$ be an $\imath$quiver, and $\Lambda^\imath$ be its $\imath$quiver algebra. Let $'Q$ be a full subquiver of $Q$ such that it is invariant under $\btau$. Denoting by $'\btau$ the restriction of $\btau$ to $'Q$, we obtain an {\em $\imath$subquiver} $('Q, {}'\btau)$ of $(Q, \btau)$. Denote by $'\Lambda^\imath$ the $\imath$quiver algebra of $('Q, {}'\btau)$. Clearly, $'\Lambda^\imath$ is a quotient algebra (and also subalgebra) of $\Lambda^\imath$. Then we can and shall view $\mod('\Lambda^\imath)$ as a full subcategory of $\mod(\Lambda^\imath)$.

\begin{lemma}
\label{lem:subalgebra}
Retain the notation as above. Then $\tM('\Lambda^\imath)$ is naturally a subalgebra of $\tMH$, with the inclusion morphism induced by
$\mod('\Lambda^\imath)\subseteq \mod(\Lambda^\imath)$.
\end{lemma}

\begin{proof}
Associated to the idempotent $e=\sum_{i\notin Q_0'}e_i$, we have
$'\Lambda^\imath=\Lambda^\imath/\Lambda^\imath e\Lambda^\imath$. Clearly, $\res_\BH ('\Lambda^\imath)$ is projective as an $\BH $-module. It follows from Proposition \ref{propostion equivalence of locally projective modules} that
\[
\pd_{\Lambda^\imath}('\Lambda^\imath) \leq1,
\qquad
\pd_{(\Lambda^\imath)^{op}}('\Lambda^\imath)\leq1.
\]
By the exact sequence
$0\rightarrow \Lambda^\imath e\Lambda^\imath\rightarrow \Lambda^\imath\rightarrow {}'\Lambda^\imath\rightarrow0$ in $\mod(\Lambda^\imath)$, we have that $\Lambda^\imath e\Lambda^\imath$ is projective as a left and right $\Lambda^\imath$-module.
It follows from \cite[Theorem~ 2.7]{PS89} that $\mod('\Lambda^\imath)$ is a full subcategory of $\mod(\Lambda^\imath)$ closed under taking extensions.
Therefore, $\ch('\Lambda^\imath)$ is naturally a subalgebra of $\ch(\Lambda^\imath)$.

For $K\in\mod('\Lambda^\imath)$, by Corollary \ref{corollary locally projective modules} we have that $\pd_{'\Lambda^\imath}(K)<\infty$ if and only if
$\pd_{\Lambda^\imath}K<\infty$. So we obtain a morphism $\phi': \cs\cd\ch('\Lambda^\imath)\rightarrow \utMH$ by definition.
By viewing $\mod(\K\, 'Q)$ as a subcategory of $\mod(\K Q)$, we have by Theorem \ref{thm:utMHbasis} that this morphism $\phi'$ is injective, and we view $\cs\cd\ch('\Lambda^\imath)$ as a subalgebra of $\utMH$. Similarly, $\mod(\K 'Q)$ is closed under taking extensions, so the Euler form $\langle \cdot,\cdot\rangle_{'Q}$ is the restriction of the Euler form $\langle \cdot,\cdot\rangle_Q$ to $K_0(\mod(k \,  'Q))\subseteq K_0(\mod(\K Q))$. Therefore, $\tM('\Lambda^\imath)$ is a subalgebra of $\tMH$.
\end{proof}

\section{Monomial bases and PBW bases of Hall algebras for $\imath$quivers}
  \label{sec:HallPBW}

In this section, $\K= \F_q$, $(Q, \btau)$ is assumed to be a Dynkin $\imath$quiver, and various Hall algebras are considered over $\Q({\sqq})$.
We prove that the semi-derived Ringel-Hall algebra $\utMH$ and its twisted version $\tMH$ admit monomial bases and PBW bases similar to those in the Ringel-Hall algebra $\ch(\K Q)$ and its twisted version $\tH(\K Q)$.

\subsection{Monomial basis of $\ch(\K Q)$}
\label{subsec:monomial-Hall}

In this subsection, we briefly recall the monomial bases of $\ch(\K Q)$ (cf. \cite{L90, Rin2, Rin3, Rei01}); we shall use the book \cite{DDPW} as a main reference. We refer to Appendix~ \ref{subsec:HA} for the definition of Hall algebras.

\begin{definition}  [\cite{Rie86}; also cf. \cite{Bon96}]
Let $M, N$ be $\K Q$-modules with same dimension vector. We say {\em $M$ degenerates to $N$}, or {\em $N$ is a degeneration of $M$}, and denote by
$[N]\leq_{\dg}[M]$, if $\dim \Hom_{\K Q}(X,N)\geq \dim\Hom_{\K Q}(X,M)$ for all $X\in\mod(\K Q)$.
\end{definition}

The relation $\leq_{\dg}$ defines a partial ordering on the set of isoclasses of $\mod(\K Q)$.
For any two $\K Q$-modules $M,N$ with same dimension vector, we have $\dim \Hom_{\K Q}(X,N)= \dim\Hom_{\K Q}(X,M)$ for all $X\in\mod(\K Q)$ if and only if $M\cong N$.

Let $M, N$ be $\K Q$-modules. In the set
\[
\ce xt_{M,N}:=\{[L] \mid \text{there exists a short exact sequence }0\rightarrow N\rightarrow L\rightarrow M\rightarrow0\},
\]
there exists a unique (up to isomorphism) $G$ having endomorphism algebra of minimal dimension. Then $[G]$ is called the \emph{generic extension} of $M$ by $N$, and denoted by
$[M]\circ[N]$. The operator $\circ$ is associative \cite{Rei01}; also see \cite[Proposition 1.30]{DDPW}.

Denote by $\Phi^+$ the set of positive roots of the Dynkin quiver $Q$ with simple roots $\alpha_i$, $i\in Q_0$.
Gabriel's Theorem says that there exists a one-to-one correspondence between $\Ind(\mod(\K Q))$ and $\Phi^+$ by mapping $M\mapsto \dimv M$. Here we view $\alpha_i$ as the dimension vector of $S_i$, $i\in\I$.

For any $\alpha\in\Phi^+$, denote by $M(\alpha):=M_q(\alpha)$ its corresponding indecomposable $\K Q$-module, i.e., $\dimv \, M(\alpha)=\alpha$.
Let $\mathfrak{P}=\mathfrak{P}(Q)$ be the set of functions $\lambda: \Phi^+\rightarrow \N$.
Then the modules
\[
M(\lambda):= \bigoplus_{\alpha\in\Phi^+}\lambda(\alpha) M(\alpha), \qquad \text{ for } \lambda\in\mathfrak{P},
\]
provide a complete set of isoclasses of $\K Q$-modules.

Let $\cw:=\cw_\I$ be the set of words in the alphabet $\I=Q_0$. For any $w=i_1\cdots i_m\in\cw$, let $\wp(w) \in \mathfrak{P}$ be specified by
\[
[M(\wp(w))]:= [S_{i_1}]\circ\cdots\circ[S_{i_m}].
\]
This defines a map
$$\wp:\cw\longrightarrow \mathfrak{P}, \quad w\mapsto \wp(w).$$
By \cite[Chapter 11.2]{DDPW}, $\wp$ is onto and it leads to a partition $\cw=\bigsqcup_{\lambda\in\mathfrak{P}} \wp^{-1}(\lambda)$.

Each word $w$ can be uniquely written in a {\em tight} form $w=j_1^{c_1}\cdots j_t^{c_t}$, where $c_1,\dots,c_t$ are positive integers and $j_r\neq j_{r+1}$ for $1\leq r<t$. A filtration
\[
M=M_0\supset M_1\supset\cdots\supset M_{t-1}\supset M_t=0
\]
of a module $M$ is called a \emph{reduced filtration} of $M$ of type $w$ if $M_{r-1}/M_r\cong c_r S_{j_r}$ for all $1\leq r\leq t$.
Denote by $\gamma_{w}^\lambda(q)$ the number of the reduced filtrations of $M(\lambda)$ of type $w$.
A word $w$ is called \emph{distinguished} if $\gamma_w^{\wp(w)}(q)=1$.

To each word $w=i_1\cdots i_m\in\cw$, we associate monomials
\begin{align*}
\ov{S}^{\diamond}_w &:=[S_{i_1}]\diamond\cdots \diamond [S_{i_m}] \in \ch(\K Q),
\\
\ov{S}^{*}_w &:=[S_{i_1}]*\cdots *[S_{i_m}] \in \tH(\K Q).
\end{align*}

\begin{proposition} [cf. \text{\cite[Corollary 11.14]{DDPW}}]
   \label{monomial basis of HQ}
Let $Q$ be a Dynkin quiver. For every $\lambda\in \mathfrak{P}=\mathfrak{P}(Q)$, choose an arbitrary distinguished word $w_\lambda\in \wp^{-1}(\lambda)$. Then $\{\ov{S}^{\diamond}_{w_\lambda}|\lambda\in \mathfrak{P}\}$ forms a $\Q(\sqq)$-basis of $\ch(\K Q)$, and $\{\ov{S}^{*}_{w_\lambda}|\lambda\in \mathfrak{P}\}$ forms a $\Q(\sqq)$-basis of $\tH(\K Q)$.
\end{proposition}

\subsection{Monomial bases for Hall algebras of $\imath$quivers}
  \label{subsection:monomial basis2}
Let $(Q, \btau)$ be a Dynkin $\imath$quiver.
Recall from Theorem \ref{thm:utMHbasis} that $\utMH$ has a Hall basis given by $[M]\diamond \E_\alpha$, where $[M]\in\Iso(\mod(\K Q))\subseteq \Iso(\mod(\Lambda^{\imath}))$, and $\alpha\in K_0(\mod(\K Q))$.
For any $\gamma\in K_0(\mod(\K Q))$, denote by $\cs\cd\ch(\Lambda^{\imath})_{\leq\gamma}$ (respectively, $\cs\cd\ch(\Lambda^\imath)_{\gamma}$) the subspace of $\utMH$ spanned by elements from this basis for which $\widehat{M}\leq \gamma$ (respectively, $\widehat{M}= \gamma$) in $K_0(\mod(\Lambda^\imath))$.
Then $\utMH$ is a $K_0(\mod(\K Q))$-graded linear space, i.e.,
\begin{align}
\label{eqn: graded linear MH}
\utMH=\bigoplus_{\gamma\in K_0(\mod(\K Q))} \utMH_\gamma.
\end{align}

\begin{lemma}
   \label{lem:filtration algebra}
We have
$
\cs\cd\ch( \Lambda^{\imath})_{\leq\alpha}\diamond \cs\cd\ch(\Lambda^{\imath})_{\leq\beta} \subseteq \cs\cd\ch(\Lambda^{\imath})_{\leq\alpha+\beta},$
for $\alpha,\beta\in K_0(\mod(\K Q))$. This defines a filtered algebra structure on $\utMH$.
\end{lemma}

\begin{proof}
For any $M$, $N\in \mod(\K Q) \subseteq \mod(\Lambda^{\imath})$, with $\widehat{M}\leq \alpha$, and $\widehat{N}\leq \beta$, we obtain that
$$[M]\diamond [N]=\sum_{[L]\in \Iso(\mod(\Lambda^{\imath}))}\frac{|\Ext^1(M,N)_L|}{|\Hom(M,N)|}[L].$$
For any $[L]$ with $|\Ext^1(M,N)_L|\neq0$, we have $\widehat{L}=\widehat{M}+\widehat{N}$ in $K_0(\mod(\Lambda^{\imath}))$.
By identifying $K_0(\mod(\K Q))$ and $K_0(\mod(\Lambda^{\imath}))$, we obtain that $\widehat{L}\leq \alpha+\beta$. By Lemma \ref{lem:another basis of MRH 1}, we can write $[L]=q^{\langle L',\phi(\gamma)\rangle} [L']\diamond \E_\gamma$ for some (unique up to isomorphism) $L'\in\mod(\K Q)$, and $\gamma\in K_0^+(\mod(\K Q))$. Then $\widehat{M}+\widehat{N}=\widehat{L}=\widehat{L'}+\widehat{\E}_\gamma$. Hence we have $\widehat{L'}\leq \widehat{L}$. So we obtain that
$[M]\diamond [N]\in \cs\cd\ch(\Lambda^{\imath})_{\leq\alpha+\beta}$.
\end{proof}

Let $(\cs\cd\ch^{\gr}(\Lambda^{\imath}), \diamond_{\gr})$ be the graded algebra associated to the filtered algebra $\utMH$, that is,
\[
\cs\cd\ch^{\gr}(\Lambda^{\imath}) = \bigoplus_{\alpha \in K_0(\mod(\K Q))} \cs\cd\ch^{\gr}(\Lambda^{\imath})_{\alpha},
\]
where $\cs\cd\ch^{\gr}(\Lambda^{\imath})_{\alpha} =\cs\cd\ch(\Lambda^{\imath})_{\leq \alpha}\big / \cs\cd\ch (\Lambda^{\imath})_{<\alpha}$. It is natural to view the quantum torus $\T(\Lambda^\imath)$ as a subalgebra of $\cs\cd\ch^{\gr}(\Lambda^{\imath})$. Then $\cs\cd\ch^{\gr}(\Lambda^{\imath})$ is also a $\T(\Lambda^\imath)$-bimodule.

For any $[M]\in\mod(kQ)$ with $\widehat{M}=\alpha$, we have by definition that $[M]\in \cs\cd\ch(\Lambda^{\imath})_{\alpha}$. By abuse of notation we also use $[M]$ to denote $[M]+\cs\cd\ch (\Lambda^{\imath})_{<\alpha} \in \cs\cd\ch^{\gr}(\Lambda^{\imath})$, i.e., $[M]\in \cs\cd\ch^{\gr}(\Lambda^\imath)_{\alpha}$.
The map
\[
\varphi: \ch(\K Q)\longrightarrow \cs\cd\ch^{\gr}(\Lambda^{\imath})
\]
is defined to be $\varphi([M])=[M]$ for any $M\in \mod(\K Q)$.

\begin{lemma}
   \label{lemma embedding of algebras}
{\quad}
\begin{itemize}
\item[(i)] $\cs\cd\ch^{\gr}(\Lambda^\imath)$ has a basis $\{[M]\diamond_{\gr}\E_\alpha \mid [M]\in\Iso(\mod(kQ)), \alpha\in K_0(\mod(\Lambda^\imath))\}$. In particular, $\cs\cd\ch^{\gr}(\Lambda^\imath)$ is free as a right $\T(\Lambda^\imath)$-module.
\item[(ii)] The linear map $\varphi: \ch(\K Q)\rightarrow \cs\cd\ch^{\gr}(\Lambda^{\imath})$, $\varphi([M])=[M], \forall M\in \mod(\K Q)$, is an embedding of algebras.
\end{itemize}
\end{lemma}

\begin{proof}
The set in (i) clearly spans $\cs\cd\ch^{\gr}(\Lambda^\imath)$ by definitions and using Theorem~\ref{thm:utMHbasis}.
It remains to check the linear independence of this set. Assume that
\begin{equation}  \label{eq:sum=0}
\sum_{[M],\alpha} \xi_{[M],\alpha} [M]\diamond_{\gr}\E_\alpha=0, \text{ for }\xi_{[M],\alpha}\in\Q(\sqq).
\end{equation}
As $\cs\cd\ch^{\gr}(\Lambda^\imath)$ is graded, without loss of generality we assume the relation \eqref{eq:sum=0} is homogeneous, i.e., $\xi_{[M],\alpha}\neq 0$ only if $\widehat{M}=\gamma$ for some fixed $\gamma$. Then the relation \eqref{eq:sum=0} in $\cs\cd\ch^{\gr}(\Lambda^{\imath})_{\gamma}$ implies that
$\sum_{[M],\alpha} \xi_{[M],\alpha} [M]\diamond\E_\alpha \in \cs\cd\ch (\Lambda^{\imath})_{<\gamma}$. But by definition we have $\sum_{[M],\alpha} \xi_{[M],\alpha} [M]\diamond\E_\alpha \in \cs\cd\ch (\Lambda^{\imath})_{\gamma}$, and hence it must be $0$, i.e., $ \xi_{[M],\alpha}=0$.
This proves (i).

For (ii), we only need to prove $\varphi$ is a morphism of algebras.
For any $M$, $N\in \mod(\K Q) \subseteq \mod(\Lambda^{\imath})$, with $\widehat{M}=\alpha$, and $\widehat{N}=\beta$, we have
\[
[M]\diamond [N]=\sum_{[L]\in \Iso(\mod(\Lambda^{\imath}))}\frac{|\Ext^1_{\Lambda^\imath}(M,N)_L|}{|\Hom_{\Lambda^\imath}(M,N)|}[L].
\]
For any $[L]$ with $|\Ext^1(M,N)_L|\neq0$, we observe from the proof of Lemma \ref{lem:filtration algebra} that $[L]=q^{\langle L',\phi(\gamma)\rangle} [L']\diamond \E_\gamma$ for some (unique up to isomorphism) $L'\in\mod(\K Q)$, and $\gamma\in K_0^+(\mod(\K Q))$. Then $[L]\in \cs\cd\ch_{ \alpha+\beta}(\Lambda^{\imath})$ if and only if $\gamma=0$, if and only if $L\in\mod(kQ)$.
By definition, we obtain that
\[
[M]\diamond_{\gr} [N]=\sum_{[L]\in \Iso(\mod (\K Q))\subseteq\Iso(\mod(\Lambda^{\imath}))}\frac{|\Ext^1_{\Lambda^{\imath}}(M,N)_L|}{|\Hom_{\Lambda^{\imath}}(M,N)|}[L].
\]
For any $L\in\mod(\K Q)$, we have
\[
\Ext^1_{\Lambda^{\imath}}(M,N)_L\cong \Ext^1_{\K Q}(M,N)_L,
\; \text{ and } \;  \Hom_{\Lambda^{\imath}}(M,N)=\Hom_{\K Q}(M,N)
\]
since $\iota:\mod(\K Q)\longrightarrow\mod(\Lambda^{\imath})$ is fully faithful and exact. By the Hall multiplication in $\ch(\K Q)$, it follows that $\varphi$ is a homomorphism.
\end{proof}

From Lemma \ref{lemma embedding of algebras}(ii), we can view $\ch(\K Q)$ as a subalgebra of $\cs\cd\ch^{\gr}(\Lambda^{\imath})$, and moreover, $\cs\cd\ch^{\gr}(\Lambda^{\imath})=\ch(\K Q)\diamond_{\gr} \T(\Lambda^\imath)$.

Recall that $\cw=\cw_\I$ is the set of words in the alphabet $\I=Q_0$.
For any $w=i_1\cdots i_m$, we set
\begin{align*}
S^{\diamond}_w &:=[S_{i_1}]\diamond\cdots \diamond [S_{i_m}] \in \utMH,
\\
S^*_w &:=[S_{i_1}]*\cdots * [S_{i_m}] \in \tMH.
\end{align*}

\begin{proposition}
   \label{prop:monomial basis of MRH}
Let $(Q, \btau)$ be a Dynkin $\imath$quiver. If $\{\ov{S}^{\diamond}_w\mid w\in \cj \}$ is a basis of $\ch(\K Q)$ for some subset $\cj$ of $\cw$, then $\{S^{\diamond}_w\mid w\in \cj \}$ is a basis of $\utMH$ as a right $\T(\Lambda^\imath)$-module.
\end{proposition}

\begin{proof}
For any $w=i_1\cdots i_m\in\cw$, denote by  $S^{\diamond,\gr}_w:= [S_{i_1}]\diamond_{\gr}\cdots \diamond_{\gr} [S_{i_m}]$ in $\cs\cd\ch^{\gr}(\Lambda^{\imath})$ and $\alpha_w:=\sum_{j=1}^m\widehat{S}_{i_j}$. It follows by Lemma~ \ref{lemma embedding of algebras}  that $\varphi(\ov{S}^{\diamond}_w)=S^{\diamond,\gr}_w$. As by assumption $\{\ov{S}^{\diamond}_w \mid w\in \cj \}$ is a basis of $\ch(\K Q)$, it follows by Lemma~ \ref{lemma embedding of algebras} that
\begin{align}  \label{basis:Sgr}
\text{$\{S^{\diamond,\gr}_w \mid w\in\cj\}$ forms a basis in $\cs\cd\ch^{\gr}(\Lambda^{\imath})$  as a right $\T(\Lambda^\imath)$-module.}
\end{align}

Observe that $S^{\diamond}_w + \cs\cd\ch (\Lambda^{\imath})_{< \alpha_w} = S^{\diamond,\gr}_w \in \cs\cd\ch^{\gr}(\Lambda^{\imath})$, by a simple induction on the length of the word $w$.
The claim that $\{S^{\diamond}_w\mid w\in \cj \}$ spans the right $\T(\Lambda^\imath)$-module $\utMH$ follows by this observation and \eqref{basis:Sgr} (also compare Theorem~\ref{thm:utMHbasis} and Lemma~\ref{lemma embedding of algebras}).
By a standard filtered algebra argument, the set $\{S^{\diamond}_w\mid w\in \cj \}$ is linearly independent in $\utMH$ as a right $\T(\Lambda^\imath)$-module.
\end{proof}

We obtain the following monomial basis theorem for $\utMH$ and its twisted version $\tMH$.

\begin{theorem}[Monomial basis theorem]
  \label{thm:monomial}
Let $(Q, \btau)$ be a Dynkin $\imath$quiver. For every $\lambda\in \mathfrak{P}=\mathfrak{P}(Q)$, choose an arbitrary distinguished word $w_\lambda\in \wp^{-1}(\lambda)$. Then
\begin{itemize}
\item[(1)] the set
$\{S^{\diamond}_{w_\lambda}\mid\lambda\in \mathfrak{P}\}$ is a $\Q(\sqq)$-basis of $\utMH$ as a right $\T(\Lambda^\imath)$-module;
 \item[(2)] the set $\{S^*_{w_\lambda}\mid\lambda\in \mathfrak{P}\}$ is a $\Q(\sqq)$-basis of $\tMH$ as a right $\tTL$-module.
\end{itemize}
\end{theorem}

\begin{proof}
The assertion (1) follows from Propositions~\ref{monomial basis of HQ} and  \ref{prop:monomial basis of MRH}, and  (2) follows from (1).
\end{proof}

\begin{corollary}
    \label{proposition:surjectivity}
Let $(Q, \btau)$ be a Dynkin $\imath$quiver. Then $\utMH$ (and respectively, $\tMH$) is generated by
$[S_i]$ for $i\in Q_0$ and $\E_\alpha$ for $\alpha\in K_0(\mod(\K Q))$.
\end{corollary}

\subsection{PBW bases for Hall algebras of $\imath$quivers}

Write the set of positive roots as
\[
\Phi^+ =\{\beta_1, \ldots, \beta_N\},
\]
where $N=|\Phi^+|$. Let $M(\beta):=M_q(\beta)$ be the unique (up to isomorphism) indecomposable $\K Q$-module with dimension vector $\beta\in \Phi^+$. As before, we also view $M(\beta)$ as a $\Lambda^{\imath}$-module by Corollary~\ref{cor:subquotient}. Inspired by the PBW basis of $\tH(\K Q)$ constructed in \cite{Rin3} (see also \cite[Chapter~ 11.5]{DDPW}), we have the following result.

\begin{theorem}[PBW basis]
  \label{thm:HallPBW}
Let $(Q, \btau)$ be a Dynkin $\imath$quiver. Let $\beta_1,\dots,\beta_N$ be any ordering of the roots in $\Phi^+$. Then,
\begin{enumerate}
\item
 the set
$
\{[M(\beta_1)]^{{\diamond} \lambda_1}\diamond\cdots\diamond[M(\beta_N)]^{{\diamond} \lambda_N} \mid   \lambda_1, \ldots, \lambda_N \in \N \}$
is a $\Q(\sqq)$-basis of $\utMH$ as a right $\T(\Lambda^\imath)$-module;
\item
the set
$
\{[M(\beta_1)]^{{*} \lambda_1}*\cdots*[M(\beta_N)]^{*\lambda_N} \mid  \lambda_1, \ldots, \lambda_N \in \N \}$
is a $\Q(\sqq)$-basis of $\tMH$ as a right $\tTL$-module.
\end{enumerate}
\end{theorem}

\begin{proof}
By \cite[Theorem 11.24]{DDPW},
$\{[M(\beta_1)]^{\diamond \lambda_1}\diamond\cdots\diamond[M(\beta_N)]^{\diamond \lambda_N} \mid  \lambda_1, \ldots, \lambda_N \in \N \}$ forms a basis of $\ch(\K Q)$. Then
$\{[M(\beta_1)]^{\diamond\lambda_1}\diamond_{\gr}\cdots\diamond_{\gr}[M(\beta_N)]^{\diamond\lambda_N} \mid \lambda\in \mathfrak{P}\}$ forms a basis of $\cs\cd\ch^{\gr}(\Lambda^\imath)$ as a $\T(\Lambda^\imath)$-module by Lemma~ \ref{lemma embedding of algebras}. Now the assertion (1) follows by a standard filtered algebra argument.
Part (2) follows by (1) and the definition of the twisted multiplication in $\tMH$.
\end{proof}

\part{$\imath$Hall Algebras and $\imath$Quantum Groups}
  \label{part:2}

\section{Quantum symmetric pairs and $\imath$quantum groups}
  \label{sec:prelim2}

In this section, we recall the definitions of $\imath$quantum groups and quantum symmetric pairs (QSP). We introduce a universal version of $\imath${}quantum groups whose central reductions are identified with the $\imath$quantum groups.

\subsection{Quantum groups}
  \label{subsection Quantum groups}

Let $Q$ be an acyclic quiver (i.e., a quiver without oriented cycles) with vertex set $Q_0= \I$.
Let $n_{ij}$ be the number of edges connecting vertex $i$ and $j$. Let $C=(c_{ij})_{i,j \in \I}$ be the symmetric generalized Cartan matrix of the underlying graph of $Q$, defined by $c_{ij}=2\delta_{ij}-n_{ij}.$ Let $\fg$ be the corresponding Kac-Moody Lie algebra. Let $\alpha_i$ ($i\in\I $) be the simple roots of $\fg$.

Let $\bv$ be an indeterminant. Write $[A, B]=AB-BA$. Denote, for $r,m \in \N$,
\[
 [r]=\frac{\bv^r-\bv^{-r}}{\bv-\bv^{-1}},
 \quad
 [r]!=\prod_{i=1}^r [i], \quad \qbinom{m}{r} =\frac{[m][m-1]\ldots [m-r+1]}{[r]!}.
\]
Then $\tU := \tU_\bv(\fg)$ is defined to be the $\Q(\bv)$-algebra generated by $E_i,F_i, \tK_i,\tK_i'$, $i\in \I$, where $\tK_i, \tK_i'$ are invertible, subject to the following relations:
\begin{align}
[E_i,F_j]= \delta_{ij} \frac{\tK_i-\tK_i'}{\bv-\bv^{-1}},  &\qquad [\tK_i,\tK_j]=[\tK_i,\tK_j']  =[\tK_i',\tK_j']=0,
\label{eq:KK}
\\
\tK_i E_j=\bv^{c_{ij}} E_j \tK_i, & \qquad \tK_i F_j=\bv^{-c_{ij}} F_j \tK_i,
\label{eq:EK}
\\
\tK_i' E_j=\bv^{-c_{ij}} E_j \tK_i', & \qquad \tK_i' F_j=\bv^{c_{ij}} F_j \tK_i',
 \label{eq:K2}
\end{align}
 and the quantum Serre relations, for $i\neq j \in \I$,
\begin{align}
& \sum_{r=0}^{1-c_{ij}} (-1)^r \left[ \begin{array}{c} 1-c_{ij} \\r \end{array} \right]  E_i^r E_j  E_i^{1-c_{ij}-r}=0,
  \label{eq:serre1} \\
& \sum_{r=0}^{1-c_{ij}} (-1)^r \left[ \begin{array}{c} 1-c_{ij} \\r \end{array} \right]  F_i^r F_j  F_i^{1-c_{ij}-r}=0.
  \label{eq:serre2}
\end{align}
Note that $\tK_i \tK_i'$ are central in $\tU$ for all $i$.
The comultiplication $\Delta: \widetilde{\U} \rightarrow \widetilde{\U} \otimes \widetilde{\U}$ is defined as follows:
\begin{align}  \label{eq:Delta}
\begin{split}
\Delta(E_i)  = E_i \otimes 1 + \tK_i \otimes E_i, & \quad \Delta(F_i) = 1 \otimes F_i + F_i \otimes \tK_{i}', \\
 \Delta(\tK_{i}) = \tK_{i} \otimes \tK_{i}, & \quad \Delta(\tK_{i}') = \tK_{i}' \otimes \tK_{i}'.
 \end{split}
\end{align}
The Chevalley involution $\omega$ on $\tU$ is given by
\begin{align}  \label{eq:omega}
\omega(E_i)  = F_i,\quad  \omega(F_i) = E_i,\quad \omega(\tK_{i}) = \tK_{i}' , \quad \omega(\tK_{i}') =\tK_{i}, \quad \forall i\in \I.
\end{align}

Analogously as for $\tU$, the quantum group $\bU$ is defined to be the $\Q(v)$-algebra generated by $E_i,F_i, K_i, K_i^{-1}$, $i\in \I$, subject to the  relations modified from \eqref{eq:KK}--\eqref{eq:serre2} with $\tK_i$ and $\tK_i'$ replaced by $K_i$ and $K_i^{-1}$, respectively. The comultiplication $\Delta$ and Chevalley involution $\omega$ on $\U$ are obtained by modifying \eqref{eq:Delta}--\eqref{eq:omega} with $\tK_i$ and $\tK_i'$ replaced by $K_i$ and $K_i^{-1}$, respectively (cf. \cite{L93}; beware that our $K_i$ has a different meaning from $K_i \in \U$ therein.)

Let $\bvs=(\vs_i)\in  (\Q(\bv)^\times)^{\I}$.
Up to a base change to the field $\Q(v)(\sqvs_i \mid i\in \I)$, the algebra $\U$ is isomorphic to a quotient algebra of $\tU$ by the ideal $( \tK_i \tK_i'- \vs_i \mid \forall i\in \I )$, by sending $F_i \mapsto F_i, E_i \mapsto \sqvs_i^{-1} E_i, K_i \mapsto \sqvs_i^{-1} \tK_i, K_i^{-1} \mapsto \sqvs_i^{-1} K_i'$. This can be verified directly.

Let $\widetilde{\bU}^+$ be the subalgebra of $\widetilde{\bU}$ generated by $E_i$ $(i\in \I)$, $\widetilde{\bU}^0$ be the subalgebra of $\widetilde{\bU}$ generated by $\tK_i, \tK_i'$ $(i\in \I)$, and $\widetilde{\bU}^-$ be the subalgebra of $\widetilde{\bU}$ generated by $F_i$ $(i\in \I)$, respectively.
The subalgebras $\bU^+$, $\bU^0$ and $\bU^-$ of $\bU$ are defined similarly. Then both $\widetilde{\bU}$ and $\bU$ have triangular decompositions:
\begin{align*}
\widetilde{\bU} =\widetilde{\bU}^+\otimes \widetilde{\bU}^0\otimes\widetilde{\bU}^-,
\qquad
\bU &=\bU^+\otimes \bU^0\otimes\bU^-.
\end{align*}
Clearly, ${\bU}^+\cong\widetilde{\bU}^+$, ${\bU}^-\cong \widetilde{\bU}^-$, and ${\bU}^0 \cong \widetilde{\bU}^0/(\tK_i \tK_i' -\vs_i \mid   i\in \I)$.

\subsection{The $\imath$quantum groups $\Ui$ and $\tUi$}
  \label{subsec:iQG}

For a  (generalized) Cartan matrix $C=(c_{ij})$, let $\Aut(C)$ be the group of all permutations $\btau$ of the set $\I$ such that $c_{ij}=c_{\btau i,\btau j}$. An element $\btau\in\Aut(C)$ is called an \emph{involution} if $\btau^2=\Id$.

Let $\btau$ be an involution in $\Aut(C)$. We define $\widetilde{\bU}^\imath:=\widetilde{\bU}'_\bv(\fk)$ to be the $\Q(v)$-subalgebra of $\tU$ generated by
\[
B_i= F_i +  E_{\btau i} \tK_i',
\qquad \tk_i = \tK_i \tK_{\btau i}', \quad \forall i \in \I.
\]
Let $\tU^{\imath 0}$ be the $\Q(v)$-subalgebra of $\tUi$ generated by $\tk_i$, for $i\in \I$.

\begin{lemma}
The elements $\tk_i$ (for $i= \btau i$) and $\tk_i \tk_{\btau i}$  (for $i\neq \btau i$) are central in $\tUi$.
\end{lemma}

\begin{proof}
If $i= \btau i$, then $\tk_i = \tK_i \tK_{i}'$.
If $i\neq \btau i$, then $\tk_i \tk_{\btau i} = \tK_i \tK_{i}'\tK_{\btau i} \tK_{\btau i}'$. In both cases, these elements are clearly central in $\tU$ and hence central in $\tUi$.
\end{proof}


Let $\bvs=(\vs_i)\in  (\Q(\bv)^\times)^{\I}$ be such that $\vs_i=\vs_{\btau i}$ for each $i\in \I$ which satisfies $c_{i, \btau i}=0$.
Let $\Ui:=\Ui_{\bvs}$ be the $\Q(v)$-subalgebra of $\bU$ generated by
\[
B_i= F_i+\vs_i E_{\btau i}K_i^{-1},
\quad
k_j= K_jK_{\btau j}^{-1},
\qquad  \forall i \in \I, \;  j \in \I \setminus \ci.
\]
It is known \cite{Let99, Ko14} that $\bU^\imath$ is a right coideal subalgebra of $\bU$ in the sense that $\Delta: \Ui \rightarrow \Ui\otimes \U$; and $(\bU,\Ui)$ is called a \emph{quantum symmetric pair} ({\em QSP} for short), as they specialize at $v=1$ to $(U(\fg), U(\fg^\theta))$, where $\theta=\omega \circ \btau$, and $\btau$ is understood here as an automorphism of $\fg$.

The algebras $\Ui_{\bvs}$, for $\bvs \in  (\Q(\bv)^\times)^{\I}$, are obtained from $\tUi$ by central reductions.

\begin{proposition}
  \label{prop:QSP12}
(1) The algebra $\Ui$ is isomorphic to the quotient of $\tUi$ by the ideal generated by
\begin{align}   \label{eq:parameters}
\tk_i - \vs_i \; (\text{for } i =\btau i),
\qquad  \tk_i \tk_{\btau i} - \vs_i \vs_{\btau i}  \;(\text{for } i \neq \btau i).
\end{align}
The isomorphism is given by sending $B_i \mapsto B_i, k_j \mapsto \vs_j^{-1} \tk_j, k_j^{-1} \mapsto \vs_{\btau j}^{-1} \tk_{\btau j}, \forall i\in \I, j\in\I\setminus\ci$.

(2) The algebra $\widetilde{\bU}^\imath$ is a right coideal subalgebra of $\widetilde{\bU}$; that is, $(\widetilde{\bU}, \widetilde{\bU}^\imath)$ forms a QSP.
\end{proposition}

\begin{proof}
Part (1) follows by definitions and the identification of $\U$ as a quotient of $\tU$.
Part~(2) follows by a direct computation using the comultiplication formula \eqref{eq:Delta}.
\end{proof}
We shall refer to $\tUi$ and $\Ui$ as {\em (quasi-split) $\imath${}quantum groups}; they are called {\em split} if $\btau =\Id$.

\subsection{ $\imath$Quantum groups of type ADE}

Now we restrict ourselves to Dynkin $\imath$quivers, which automatically exclude the type $(A_{2r}, \btau \neq \Id)$ since a quiver $Q$ of type $A_{2r}$ does not admit an involution $\btau\neq \Id$. Under this assumption, we always have $c_{i, \btau i}=0$ when $i\neq \btau i$, and so $\vs_{\btau i} =\vs_i$ for all $i\in \I$; compare \eqref{eq:parameters}.

Let us fix the labelings for the ADE Dynkin diagrams (excluding type $A_{2r}$) with nontrivial diagram automorphisms, of type $A_{2r+1}, D_n, E_6$, as follows:

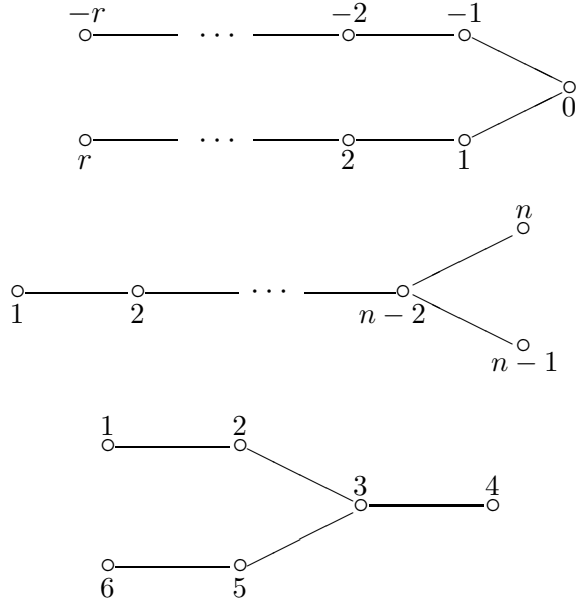
\begin{center}\setlength{\unitlength}{0.7mm}

\begin{picture}(100,40)(0,0)
\put(0,10){$\circ$}
\put(0,30){$\circ$}
\put(50,10){$\circ$}
\put(50,30){$\circ$}
\put(72,10){$\circ$}
\put(72,30){$\circ$}
\put(92,20){$\circ$}
\put(0,6){$r$}
\put(-2,34){${-r}$}
\put(50,6){\small $2$}
\put(48,34){\small ${-2}$}
\put(72,6){\small $1$}
\put(70,34){\small ${-1}$}
\put(92,16){\small $0$}

\put(3,11.5){\line(1,0){16}}
\put(3,31.5){\line(1,0){16}}
\put(23,10){$\cdots$}
\put(23,30){$\cdots$}
\put(33.5,11.5){\line(1,0){16}}
\put(33.5,31.5){\line(1,0){16}}
\put(53,11.5){\line(1,0){18.5}}
\put(53,31.5){\line(1,0){18.5}}

\put(75,12){\line(2,1){17}}
\put(75,31){\line(2,-1){17}}
\end{picture}
\vspace{-0.5cm}
\end{center}

\begin{center}\setlength{\unitlength}{0.8mm}
 \begin{picture}(50,20)(30,0)
\put(0,-1){$\circ$}
\put(0,-5){\small$1$}
\put(20,-1){$\circ$}
\put(20,-5){\small$2$}
\put(64,-1){$\circ$}
\put(58,-5){\small$n-2$}
\put(84,-10){$\circ$}
\put(80,-13){\small${n-1}$}
\put(84,9.5){$\circ$}
\put(84,12.5){\small${n}$}

\put(19.5,0){\line(-1,0){16.8}}
\put(38,0){\line(-1,0){15.5}}
\put(64,0){\line(-1,0){15}}

\put(40,-1){$\cdots$}

\put(83.5,9.5){\line(-2,-1){16.5}}
\put(83.5,-8.5){\line(-2,1){16.5}}
\end{picture}
\vspace{1.2cm}
\end{center}

\begin{center}\setlength{\unitlength}{0.8mm}
 \begin{picture}(100,40)(0,0)
\put(10,6){\small${6}$}
\put(10,10){$\circ$}
\put(32,6){\small${5}$}
\put(32,10){$\circ$}

\put(10,30){$\circ$}
\put(10,33){\small${1}$}
\put(32,30){$\circ$}
\put(32,33){\small${2}$}

\put(31.5,11){\line(-1,0){19}}
\put(31.5,31){\line(-1,0){19}}

\put(52,22){\line(-2,1){17.5}}
\put(52,20){\line(-2,-1){17.5}}

\put(54.7,21.2){\line(1,0){19}}

\put(52,20){$\circ$}
\put(52,23){\small$3$}
\put(74,20){$\circ$}
\put(74,23){\small$4$}
\end{picture}
\vspace{0.2cm}
\end{center}

We choose the subset $\ci$ \eqref{eq:ci} of representatives of $\btau$-orbits on $\I$ to be:
\begin{align}
\label{eqn:representative}
&\ci :=\left\{ \begin{array}{cl} \I, & \text{ if } \btau=\Id,\\
\left.\begin{array}{cl}
\{0,1,\dots,r\}, & \text{ if } \Delta \text{ is of type }A_{2r+1},\\
\{1,\dots,n-1\}, & \text{ if } \Delta \text{ is of type }D_n,\\
\{1,2,3,4\},  & \text{ if } \Delta \text{ is of type }E_6,\end{array} \right\}
& \text{ if } \btau\neq\Id.\end{array}\right.
\end{align}

The quasi-split $\imath$quantum group $\Ui$ admits the following presentation \cite{Let99, Let03} (also cf. \cite{Ko14}) when $\fg$ is of type ADE (excluding quasi-split $A_{2r}$ with $\btau\neq \Id$).

\begin{proposition}  [cf. \cite{Let99, Let03, Ko14}]
  \label{prop:Ui}
Let $(Q,\btau)$ be a Dynkin $\imath$quiver. 
The $\Q(v)$-algebra $\Ui =\Ui_\bvs$ has a presentation with generators $B_i$, for $i\in \I$, and $k_i$, for $i\in \I \setminus \ci$, subject to the relations \eqref{i-rel1}--\eqref{i-rel2}: for $m, \ell \in \I \setminus \ci$ and $i\neq j \in \I$,
\begin{align}
k_m k_\ell =k_\ell k_m,
\quad
k_\ell B_i & = v^{c_{\btau \ell, i} -c_{\ell i}} B_i k_\ell,
 \label{i-rel1}
 \\
B_iB_{j}-B_jB_i &=0, \quad \text{ if } c_{ij} =0 \text{ and }\btau i\neq j,
\label{i-rel1+}
\\
\sum_{s=0}^{1-c_{ij}} (-1)^s \qbinom{1-c_{ij}}{s} B_i^{s}B_jB_i^{1-c_{ij}-s} &=0, \quad \text{ if } j \neq \btau i\neq i,
\\
B_{\btau i}B_i -B_i B_{\btau i}& = \vs_i \frac{k_i -k_i^{-1}}{v-v^{-1}},  \quad \text{ if } \btau i \neq i \in \I \setminus \ci,
\label{i-rel5}
\\
B_i^2B_{j} - [2] B_iB_{j}B_i +B_{j}B_i^2 &= v\vs_i B_{j},  \quad \text{ if }  c_{ij} = -1 \text{ and }\btau i=i.
\label{i-rel2}
\end{align}
\end{proposition}
Note in the split case, the above presentation is much simplified: the split $\imath$quantum group $\Ui$ is generated by $B_i$ for $i\in \I$ with 2 relations \eqref{i-rel1+} and \eqref{i-rel2}.

Below we provide a complete list of quasi-split symmetric pairs $(\fg, \fk)$, where $\fg$ is of type ADE, $\theta=\omega \circ \btau$, and $\btau$ is understood here as an automorphism of $\fg$. (Note the non-split case with $\fg =\mathfrak{sl}_{2n+1}(\C)$ is excluded from the list, as $\btau$ cannot respect the arrows of any quiver.)

\begin{table}[h]
\caption{List of quasi-split symmetric pairs}
\label{table:values}
\begin{tabular}{| c | c | c | c | c | c | c | c |}
\hline
\begin{tikzpicture}[baseline=0]
\node at (0, 0.2) {$\fg$};
\end{tikzpicture}
&
\begin{tikzpicture}[baseline=0]
\node at (0, 0.2) {$\mathfrak{sl}_n(\C)$};
\end{tikzpicture}
&
\begin{tikzpicture}[baseline=0]
\node at (0, 0.2) {$\mathfrak{so}_{2n}(\C)$};
\end{tikzpicture}
	&
\begin{tikzpicture}[baseline=0]
\node at (0, 0.2) {$E_6$};
\end{tikzpicture}
	&
\begin{tikzpicture}[baseline=0]
\node at (0, 0.2) {$E_7$};
\end{tikzpicture}
	&
\begin{tikzpicture}[baseline=0]
\node at (0, 0.2) {$E_8$};
\end{tikzpicture}
\\
\hline
$\fk \text{ (split)}$
&
$\mathfrak{so}_{n}(\C)$
 &
$\mathfrak{so}_{n}(\C) \oplus \mathfrak{so}_{n}(\C)$
 &
$\mathfrak{sp}_4(\C)$
 &
$\mathfrak{sl}_{8}(\C)$
 &
$\mathfrak{so}_{16}(\C)$
\\
\hline
\hline
\begin{tikzpicture}[baseline=0]
\node at (0, 0.2) {$\fg$};
\end{tikzpicture}
&
\begin{tikzpicture}[baseline=0]
\node at (0, 0.2) {$\mathfrak{sl}_{2n}(\C)$};
\end{tikzpicture}
&
\begin{tikzpicture}[baseline=0]
\node at (0, 0.2) {$\mathfrak{so}_{2n}(\C)$};
\end{tikzpicture}
&
\begin{tikzpicture}[baseline=0]
\node at (0, 0.2) {$E_6$};
\end{tikzpicture}
&
&
\\
\hline
$\fk \text{ (non-split)}$
&
$\mathfrak{sl}_{n}(\C) \oplus \mathfrak{gl}_{n}(\C)$
&
$\mathfrak{so}_{n-1}(\C) \oplus \mathfrak{so}_{n+1}(\C)$	
&
$\mathfrak{sl}_{2}(\C) \oplus \mathfrak{sl}_{6}(\C)$
&
&
\\
\hline
\end{tabular}
\newline
\end{table}

\subsection{Presentation of $\tUi$}
   \label{subsec:Quasi-split}

Let $(Q, \btau)$ be a Dynkin $\imath$quiver, with underlying graph $\Delta$ and associated semisimple Lie algebra $\fg$. Let $n=|\I|$. Recall the $\Q(v)$-algebras $\tUi$ and $\Ui$ were defined in \S\ref{subsec:iQG}, depending on the parameters $\bvs =(\vs_1, \ldots, \vs_n) \in (\Q(v)^{\times})^n$ which satisfies $\vs_{\btau i} =\vs_i$ if $\btau i \neq i$.  The parameters $\vs_i$ are related to the parameters $s(i), c_i$ used in \cite{BK19} by $\vs_i =-c_i s(\btau(i))$. (As remarked in \cite{BW18b}, the parameters $s(i), c_i$ are never needed in any crucial formula separately.) 

The algebra $\tUi$ differs from $\Ui$ by having the additional central elements $\tk_j$ (for $j=\btau j$) and $\tk_i \tk_{\btau i}$ (for $i \in \ci$), and so the following presentation from $\tUi$ can be obtained by modifying slightly the presentation for $\Ui$ given in Proposition~\ref{prop:Ui}.

\begin{proposition}
  \label{prop:Serre}
Let $(Q, \btau)$ be a Dynkin $\imath$quiver. The $\Q(v)$-algebra $\tUi$ has a presentation with generators $B_i, \tk_i$ $(i\in \I)$, where $\tk_i$ are invertible, subject to the relations \eqref{relation1}--\eqref{relation2}: for $\ell \in \I$, and $i\neq j \in \I$,
\begin{align}
\tk_i \tk_\ell =\tk_\ell \tk_i,
\quad
\tk_\ell B_i & = v^{c_{\btau \ell,i} -c_{\ell i}} B_i \tk_\ell,
   \label{relation1}
\\
B_iB_{j}-B_jB_i &=0, \quad \text{ if } c_{ij} =0 \text{ and }\btau i\neq j,
 \label{relationBB}
\\
\sum_{s=0}^{1-c_{ij}} (-1)^s \qbinom{1-c_{ij}}{s} B_i^{s}B_jB_i^{1-c_{ij}-s} &=0, \quad \text{ if } j \neq \btau i\neq i,
\\
B_{\btau i}B_i -B_i B_{\btau i}& =   \frac{\tk_i -\tk_{\btau i}}{v-v^{-1}},
\quad \text{ if } \btau i \neq i,
\label{relation5}
\\
B_i^2B_{j} - [2] B_iB_{j}B_i +B_{j}B_i^2 &= v \tk_i B_{j},
%
\quad \text{ if }  c_{ij} = -1 \text{ and }\btau i=i.
\label{relation2}
\end{align}
\end{proposition}
Note by definition that $\tUi$ does not depend on the parameters $\bvs$.

\begin{corollary}
Let $(Q, \btau)$ be a Dynkin $\imath$quiver. There exists a bar involution $\psi_{\imath}$ of the $\Q$-algebra $\tUi$ such that $\psi_{\imath}(v) =v^{-1}$, and for $i\in \I$,
$$
\psi_{\imath}(B_i) =B_i, \;\;
\psi_{\imath}(\tk_i) =\tk_{\btau i} \; ({\rm if }\; \btau i \neq i),  \;\;
\psi_{\imath}(\tk_i) =v^2 \tk_i  \; ({\rm if }\; \btau i = i).
$$
\end{corollary}

In the split case (i.e., $\btau=\Id$), we only need relations \eqref{relation1}, \eqref{relationBB} and \eqref{relation2} for presentation of $\tUi$.  The non-split cases are made more explicit below.

\medskip
\subsubsection{Type $A_{n}$ with $n$ odd}

Set $n=2r+1$ for $r\geq0$. An example of  $\imath$quiver $Q$ of Dynkin type $A_{2r+1}$, with the involution $\btau$ given by $\btau(i)=-i$, is as follows:
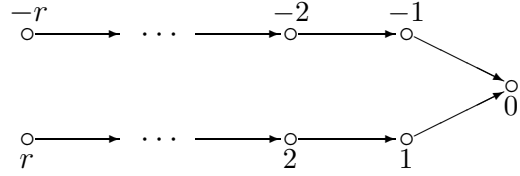
\begin{center}\setlength{\unitlength}{0.7mm}
 \begin{picture}(100,40)(0,0)
\put(0,10){$\circ$}
\put(0,30){$\circ$}
\put(50,10){$\circ$}
\put(50,30){$\circ$}
\put(72,10){$\circ$}
\put(72,30){$\circ$}
\put(92,20){$\circ$}
\put(0,6){$r$}
\put(-2,34){${-r}$}
\put(50,6){\small $2$}
\put(48,34){\small ${-2}$}
\put(72,6){\small $1$}
\put(70,34){\small ${-1}$}
\put(92,16){\small $0$}

\put(3,11.5){\vector(1,0){16}}
\put(3,31.5){\vector(1,0){16}}
\put(23,10){$\cdots$}
\put(23,30){$\cdots$}
\put(33.5,11.5){\vector(1,0){16}}
\put(33.5,31.5){\vector(1,0){16}}
\put(53,11.5){\vector(1,0){18.5}}
\put(53,31.5){\vector(1,0){18.5}}

\put(75,12){\vector(2,1){17}}
\put(75,31){\vector(2,-1){17}}
\end{picture}
\vspace{-0.2cm}
\end{center}
In this case $\tUi$ is the subalgebra of $\tU$ generated by
\begin{align*}
\tk_i =\tK_i \tK_{-i}',
\qquad
B_i=F_i+  E_{-i} \tK_i' \;\; (-r \le i \le r).
\end{align*}

For the convenience in the proof of Proposition~\ref{prop:qHall=Ui} later on, we rewrite the relations in Proposition~\ref{prop:Serre} for $\tUi$ in type A as the relations \eqref{relation A odd 1}-\eqref{relation A odd 10}:
\begin{align}
\label{relation A odd 1} \tk_i\tk_j=\tk_j\tk_i,&\quad \text{ if }-r\leq i\leq r, \\
\label{relation A odd 2} \tk_iB_j=v^{(-\alpha_i+\alpha_{-i},\alpha_j)} B_j\tk_i,  &\quad \text{ if }-r\leq i,  j \leq r, 
\\
\label{relation A odd 4}    B_iB_{-i}-B_{-i}B_i=  \frac{\tk_{-i}-\tk_{i}}{v-v^{-1}},&\quad\text{ if } 1 \leq i\leq r, \\
\label{relation A odd 5}   B_iB_j=B_jB_i,
&  \quad  \text{ if }-r \leq i,j\leq r,|i-j|>1, i\neq -j,
\\
\label{relation A odd 7}B_i^2B_j+B_jB_i^2=(v+v^{-1})B_iB_jB_i,
&\quad\text{ if }-r\le i,j\le r,|i-j|=1, i\neq 0,
\\
\label{relation A odd 9}B_0^2B_1+B_1B_0^2 & =(v+v^{-1})B_0B_1B_0+ v \tk_0 B_1,\\
\label{relation A odd 10}  B_0^2B_{-1}+B_{-1}B_0^2 & =(v+v^{-1})B_0B_{-1}B_0+v \tk_0 B_{-1}.
\end{align}
%
Associated to the quiver $Q$ of type $A_{2r+1}$ above, the quiver $\ov{Q}$ for the $\imath$quiver algebra $\Lambda^{\imath}$ in \eqref{eq:iLa} is as follows:
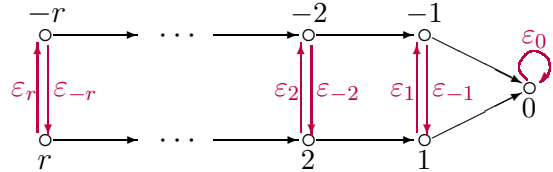
\begin{center}\setlength{\unitlength}{0.7mm}
 \begin{picture}(100,40)(0,0)
\put(0,10){$\circ$}
\put(0,30){$\circ$}
\put(50,10){$\circ$}
\put(50,30){$\circ$}
\put(72,10){$\circ$}
\put(72,30){$\circ$}
\put(92,20){$\circ$}
\put(0,6){$r$}
\put(-2,34){${-r}$}
\put(50,6){\small $2$}
\put(48,34){\small ${-2}$}
\put(72,6){\small $1$}
\put(70,34){\small ${-1}$}
\put(92,16){\small $0$}

\put(3,11.5){\vector(1,0){16}}
\put(3,31.5){\vector(1,0){16}}
\put(23,10){$\cdots$}
\put(23,30){$\cdots$}
\put(33.5,11.5){\vector(1,0){16}}
\put(33.5,31.5){\vector(1,0){16}}
\put(53,11.5){\vector(1,0){18.5}}
\put(53,31.5){\vector(1,0){18.5}}

\put(75,12){\vector(2,1){17}}
\put(75,31){\vector(2,-1){17}}
\color{purple}
\put(0,13){\vector(0,1){17}}
\put(2,29.5){\vector(0,-1){17}}
\put(50,13){\vector(0,1){17}}
\put(52,29.5){\vector(0,-1){17}}
\put(72,13){\vector(0,1){17}}
\put(74,29.5){\vector(0,-1){17}}

\put(-5,20){$\varepsilon_r$}
\put(3,20){$\varepsilon_{-r}$}
\put(45,20){\small $\varepsilon_2$}
\put(53,20){\small $\varepsilon_{-2}$}
\put(67,20){\small $\varepsilon_1$}
\put(75,20){\small $\varepsilon_{-1}$}
\put(92,30){\small $\varepsilon_0$}

\qbezier(93,23)(90.5,25)(92,27)
\qbezier(92,27)(94,30)(97,27)
\qbezier(97,27)(98,24)(95.5,22.6)
\put(95.6,23){\vector(-1,-1){0.3}}
\end{picture}
\vspace{-0.2cm}
\end{center}

\subsubsection{Type $D_{n}$}

Let $Q$ be a quiver of type $D_{n}$, such that there is an involution $\btau$ of $Q$ which interchanges $n-1$ and $n$ while fixing all other vertices. For example, $Q$ can be taken to be
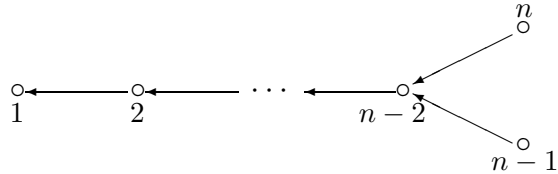
\begin{center}\setlength{\unitlength}{0.8mm}
 \begin{picture}(50,20)(30,0)
\put(0,-1){$\circ$}
\put(0,-5){\small$1$}
\put(20,-1){$\circ$}
\put(20,-5){\small$2$}
\put(64,-1){$\circ$}
\put(58,-5){\small$n-2$}
\put(84,-10){$\circ$}
\put(80,-13){\small${n-1}$}
\put(84,9.5){$\circ$}
\put(84,12.5){\small${n}$}

\put(19.5,0){\vector(-1,0){16.8}}
\put(38,0){\vector(-1,0){15.5}}
\put(64,0){\vector(-1,0){15}}

\put(40,-1){$\cdots$}

\put(83.5,9.5){\vector(-2,-1){16.5}}
\put(83.5,-8.5){\vector(-2,1){16.5}}
\end{picture}
\vspace{1.2cm}
\end{center}
%
%
In this case $\tUi$ is the subalgebra of $\tU$ generated by
\begin{align*}
\tk_i= \tK_i \tK_i'\text{ }(1\le i\le n-2), &\quad
\tk_{n-1}= \tK_{n-1} \tK_{n}', \quad \tk_{n}=K_nK_{n-1}',\\
B_{n-1}=F_{n-1}+  E_{n} \tK_{n-1}', & \quad B_n=F_{n}+ E_{n-1} \tK_n',
\\
B_i= F_i+ E_i \tK_i', & \quad \forall 1\leq i\leq n-2.
\end{align*}
%
Associated to the above quiver $Q$ of type $D_n$, the quiver $\ov{Q}$ for  the $\imath$quiver algebra $\Lambda^{\imath}$ in \eqref{eq:iLa} is given as follows:
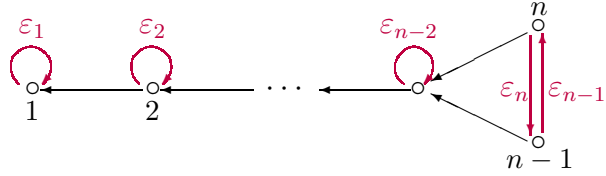
\begin{center}\setlength{\unitlength}{0.8mm}
 \begin{picture}(50,20)(30,0)
\put(0,-1){$\circ$}
\put(0,-5){\small$1$}
\put(20,-1){$\circ$}
\put(20,-5){\small$2$}
\put(64,-1){$\circ$}
\put(84,-10){$\circ$}
\put(80,-13){\small${n-1}$}
\put(84,9.5){$\circ$}
\put(84,12.5){\small${n}$}

\put(19.5,0){\vector(-1,0){16.8}}
\put(38,0){\vector(-1,0){15.5}}
\put(64,0){\vector(-1,0){15}}

\put(40,-1){$\cdots$}
\put(83.5,9.5){\vector(-2,-1){16}}
\put(83.5,-8.5){\vector(-2,1){16}}
\color{purple}
\put(86,-7){\vector(0,1){16.5}}
\put(84,9){\vector(0,-1){16.5}}

\qbezier(63,1)(60.5,3)(62,5.5)
\qbezier(62,5.5)(64.5,9)(67.5,5.5)
\qbezier(67.5,5.5)(68.5,3)(66.4,1)
\put(66.5,1.4){\vector(-1,-1){0.3}}
\qbezier(-1,1)(-3,3)(-2,5.5)
\qbezier(-2,5.5)(1,9)(4,5.5)
\qbezier(4,5.5)(5,3)(3,1)
\put(3.1,1.4){\vector(-1,-1){0.3}}
\qbezier(19,1)(17,3)(18,5.5)
\qbezier(18,5.5)(21,9)(24,5.5)
\qbezier(24,5.5)(25,3)(23,1)
\put(23.1,1.4){\vector(-1,-1){0.3}}

\put(-1,9.5){$\varepsilon_1$}
\put(19,9.5){$\varepsilon_2$}
\put(59,9.5){$\varepsilon_{n-2}$}
\put(79,-1){$\varepsilon_{n}$}
\put(87,-1){$\varepsilon_{n-1}$}
\end{picture}
\vspace{.8cm}
\end{center}

\subsubsection{Type $E_6$}

Let $Q$ be a quiver of Dynkin type $E_{6}$ with a nontrivial involution $\btau$ given by $\btau(1)=6$, $\btau(2)=5$, and $\btau(3)=3$, $\btau(4)=4$. For example, $Q$ can be taken to be
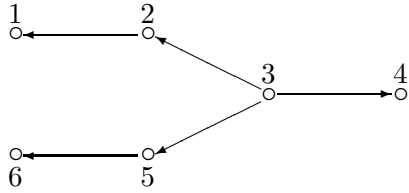
\begin{center}\setlength{\unitlength}{0.8mm}
 \begin{picture}(100,40)(0,0)
\put(10,6){\small${6}$}
\put(10,10){$\circ$}
\put(32,6){\small${5}$}
\put(32,10){$\circ$}

\put(10,30){$\circ$}
\put(10,33){\small${1}$}
\put(32,30){$\circ$}
\put(32,33){\small${2}$}

\put(31.5,11){\vector(-1,0){19}}
\put(31.5,31){\vector(-1,0){19}}

\put(52,22){\vector(-2,1){17.5}}
\put(52,20){\vector(-2,-1){17.5}}

\put(54.7,21.2){\vector(1,0){19}}

\put(52,20){$\circ$}
\put(52,23){\small$3$}
\put(74,20){$\circ$}
\put(74,23){\small$4$}
\end{picture}
\vspace{-0.2cm}
\end{center}
%
%
Associated to the quiver $Q$ above, the quiver $\ov{Q}$ for the $\imath$quiver algebra $\Lambda^{\imath}$ in \eqref{eq:iLa} is given as follows:
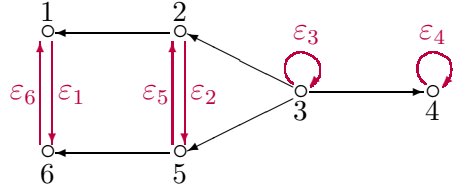
\begin{center}\setlength{\unitlength}{0.8mm}
 \begin{picture}(100,40)(0,0)
\put(10,6){\small${6}$}
\put(10,10){$\circ$}
\put(32,6){\small${5}$}
\put(32,10){$\circ$}

\put(10,30){$\circ$}
\put(10,33){\small${1}$}
\put(32,30){$\circ$}
\put(32,33){\small${2}$}

\put(31.5,11){\vector(-1,0){19}}
\put(31.5,31){\vector(-1,0){19}}

\put(52,22){\vector(-2,1){17.5}}
\put(52,20){\vector(-2,-1){17.5}}

\put(54.7,21.2){\vector(1,0){19}}

\put(52,20){$\circ$}
\put(52,16.5){\small$3$}
\put(74,20){$\circ$}
\put(74,16.5){\small$4$}
\color{purple}
\put(10,12.5){\vector(0,1){17}}
\put(12,29.5){\vector(0,-1){17}}
\put(32,12.5){\vector(0,1){17}}
\put(34,29.5){\vector(0,-1){17}}

\qbezier(52,22.5)(50,24)(51,26.5)
\qbezier(51,26.5)(53,29)(56,26.5)
\qbezier(56,26.5)(57.5,24)(55,22)
\put(55.1,22.4){\vector(-1,-1){0.3}}
\qbezier(74,22.5)(72,24)(73,26.5)
\qbezier(73,26.5)(75,29)(78,26.5)
\qbezier(78,26.5)(79,24)(77,22)
\put(77.1,22.4){\vector(-1,-1){0.3}}

\put(35,20){$\varepsilon_2$}
\put(27,20){$\varepsilon_5$}
\put(13,20){$\varepsilon_1$}
\put(5,20){$\varepsilon_6$}
\put(52,30){$\varepsilon_3$}
\put(73,30){$\varepsilon_4$}
\end{picture}
\vspace{-0.2cm}
\end{center}
In this case $\tUi$ is defined to be the subalgebra of $\tU$ generated by
\begin{align*}
\tk_i= \tK_i \tK_{7-i}' \; (i=1,2,5,6),
&\quad
 \tk_3= \tK_3 \tK_3',\quad \tk_4= \tK_4 \tK_4',\\
B_i=F_i+ E_{7-i} \tK_i' \; (i=1,2,5,6),
& \quad
B_i=F_i+ E_i \tK_i' \; (i=3,4).
\end{align*}

\section{Hall algebras for $\imath$quivers and $\imath$quantum groups}
  \label{sec:HallQSP}

In this section, we shall assume the quiver is Dynkin. We will use the $\imath$Hall algebras $\tMH$ and $\rMH$ (i.e., the twisted reduced semi-derived Ringel-Hall algebras of $\La^\imath$) to provide a realization of the $\imath$quantum groups $\tUi$ and $\Ui$ of finite type, respectively; see Theorem~\ref{theorem isomorphism of Ui and MRH}.

\subsection{Computations for rank 2 $\imath$quivers, I}
\label{subsec:rank2I}

The {\em rank} of an $\imath$quiver $(Q, \btau)$ is by definition the number of $\btau$-orbits in $Q_0$. Recall $\sqq =\sqrt{q}$. Recall that $\langle\cdot,\cdot\rangle_Q$ is the Euler form of $Q$. Define
\begin{align*}
(x,y)&= \langle x,y\rangle_Q+\langle y,x\rangle_Q.
\end{align*}
In particular, $(S_i,S_j)=c_{ij}$ for any $i,j\in \I$, the entries for the Cartan matrix $C$.

\begin{proposition}
  \label{prop:A2}
Let $Q$ be the quiver $1\xrightarrow{\alpha} 2$, with $\btau=\Id$. Then in $\tMH$ we have
\begin{align}
[S_2]*[S_1]*[S_1]-({\sqq}+{\sqq}^{-1})[S_1]*[S_2]*[S_1]+[S_1]*[S_1]*[S_2]=-\frac{(q-1)^2}{{\sqq}}[S_2]*[ \E_1],
  \label{eq:S211}
\\
[S_1]*[S_2]*[S_2]-({\sqq}+{\sqq}^{-1})[S_2]*[S_1]*[S_2]+[S_2]*[S_2]*[S_1]=-\frac{(q-1)^2}{{\sqq}}[S_1]*[ \E_2].
  \label{eq:S122}
\end{align}
\end{proposition}

\begin{proof}
Recall from Example \ref{example 2}(a) the quiver and relations of $\Lambda^{\imath}$ for this $\imath$quiver.  We shall only prove the first identity \eqref{eq:S211} while skipping a similar proof of the identity \eqref{eq:S122}.

Denote by $U_i$ the indecomposable projective $\Lambda^{\imath}$-module corresponding to $i$. Denote by $X$ the unique indecomposable $\Lambda^{\imath}$-module with $\widehat{S_1}+\widehat{S_2}$ as its class in $K_0(\mod(\Lambda^{\imath}))$. Then we have
\begin{align*}
[S_2]*[S_1]*[S_1]&= {\sqq}^{\langle S_2,S_1\rangle_Q} [S_1\oplus S_2]*[S_1]\\
&= {\sqq}^{\langle S_2,S_1\rangle_Q} {\sqq}^{\langle S_1\oplus S_2,S_1\rangle_Q}
\Big(\frac{1}{q} [S_2\oplus S_1\oplus S_1] +\frac{q-1}{q}[S_2\oplus \E_1] \Big)\\
&= \frac{1}{{\sqq}} [S_2\oplus S_1\oplus S_1] +\frac{q-1}{{\sqq}}[S_2\oplus \E_1];
\end{align*}
\begin{align*}
[S_1]*[S_2]*[S_1]&= {\sqq}^{\langle S_1,S_2\rangle_Q} ([S_1\oplus S_2]*[S_1]+(q-1)[X] *[S_1]) \\
&= {\sqq}^{\langle S_1,S_2\rangle_Q} {\sqq}^{\langle S_1\oplus S_2,S_1\rangle_Q}
\Big( \frac{1}{q}[S_1\oplus S_1\oplus S_2] +\frac{q-1}{q}[S_2\oplus \E_1] \Big)\\
& \quad +{\sqq}^{\langle S_1,S_2\rangle_Q} {\sqq}^{\langle S_1\oplus S_2,S_1\rangle_Q}
\Big(\frac{q-1}{q}[X\oplus S_1]+ \frac{(q-1)^2}{q}[U_1/S_2] \Big)\\
&= \frac{1}{q}[S_1\oplus S_1\oplus S_2] +\frac{q-1}{q}[S_2\oplus \E_1]
 \\
 &\qquad\qquad\qquad\qquad\, +\frac{q-1}{q}[X\oplus S_1]+ \frac{(q-1)^2}{q}[U_1/S_2];
\end{align*}
\begin{align*}
[S_1]*[S_1]*[S_2]&= [S_1]* \big({\sqq}^{\langle S_1,S_2\rangle_Q} ([S_1\oplus S_2]+(q-1)[X]) \big) \\
&= {\sqq}^{\langle S_1,S_2\rangle_Q} {\sqq}^{\langle S_1,S_1\oplus S_2\rangle_Q} \\
&\qquad
 \cdot \Big(\frac{1}{q} [S_1\oplus S_1\oplus S_2]+\frac{q-1}{q}[\E_1\oplus S_2]+\frac{q-1}{q}[S_1\oplus X] \Big)\\
& \quad +{\sqq}^{\langle S_1,S_2\rangle_Q} {\sqq}^{\langle S_1,S_1\oplus S_2\rangle_Q}
\Big(\frac{(q-1)^2}{q}[U_1/S_2]+(q-1)[S_1\oplus X] \Big)\\
&= \frac{1}{q{\sqq}} [S_1\oplus S_1\oplus S_2]+\frac{q-1}{q{\sqq}}[\E_1\oplus S_2]
 +\frac{q^2-1}{q{\sqq}}[S_1\oplus X]+\frac{(q-1)^2}{q{\sqq}}[U_1/S_2].
\end{align*}

By the short exact sequence $0\rightarrow S_2\rightarrow U_1/S_2\rightarrow
\E_1\rightarrow 0$ with $\E_1 \in \cp^{\leq 1}(\La^\imath)$ by Lemma~\ref{lemma locally projective modules}, we have
$[U_1/S_2]=[\E_1\oplus S_2]$ in $\tMH$, and
then by Lemma \ref{lemma compatible of Euler form} we have
\[
[S_2]*[ \E_1]={\sqq}^{\langle S_2,S_1\oplus S_1\rangle_Q} q^{-\langle S_2,\E_1\rangle}[U_1/S_2]=[U_1/S_2].
\]
Hence, in $\tMH$ we obtain that
\begin{align*}
[S_2]*[S_1]*[S_1]  - & ({\sqq}+{\sqq}^{-1})[S_1]* [S_2]*[S_1]+[S_1]*[S_1]*[S_2]
\\
&= -\frac{(q-1)^2}{{\sqq}}[U_1/S_2]
 = -\frac{(q-1)^2}{{\sqq}}[S_2]*[ \E_1].
\end{align*}
The proposition is proved.
\end{proof}
\subsection{Computations for rank 2 $\imath$quivers, II}
\label{subsec:rank2II}

\begin{proposition}
   \label{prop:iA3}
Let $Q$ be the quiver $1\xrightarrow{\alpha} 2\xleftarrow{\beta}3$ with $\btau$ being the nontrivial involution. Then in $\tMH$ we have, for $i=1, 3$,
\begin{align}
[S_i]*[S_i]*[S_2]-({\sqq}+{\sqq}^{-1})[S_i]*[S_2]*[S_i]+[S_2]*[S_i]*[S_i] &=0,
 \label{eq:Serre112} \\
[S_2]*[S_2]*[S_i]-({\sqq}+{\sqq}^{-1})[S_2]*[S_i]*[S_2]+[S_i]*[S_2]*[S_2]
& =-\frac{(q-1)^2}{{\sqq}}[ S_i]*[\E_2].      \label{eq:Serre221}
\end{align}
\end{proposition}

\begin{proof}
Recall from Example \ref{example 2}(b) the quiver and relations of $\Lambda^{\imath}$.
Denote by $U_i$ the indecomposable projective $\Lambda^{\imath}$-module corresponding to $i \in \{1,2,3\}$. Denote by $X$ the unique indecomposable $\Lambda^{\imath}$-module with $\widehat{S_1}+\widehat{S_2}$ as its class in $K_0(\mod(\Lambda^\imath))$.
We shall only prove the formulas for $i=1$, as the remaining case with $i=3$ follows by symmetry.

In $\tMH$, we have
\begin{align*}
[S_1]*[S_1]*[S_2]&= {\sqq}^{\langle S_1,S_2\rangle_Q}[S_1]* \big([S_1\oplus S_2] + (q-1)[X] \big) \\
&= {\sqq}^{\langle S_1,S_2\rangle_Q} {\sqq}^{\langle S_1,S_1\oplus S_2\rangle_Q}  \\
&\qquad \cdot \Big(\frac{1}{q}[S_1\oplus S_1\oplus S_2]+ \frac{q-1}{q}[S_1\oplus X]+(q-1)[S_1\oplus X] \Big)\\
&= \frac{1}{q{\sqq}}[S_1\oplus S_1\oplus S_2]+ \frac{q^2-1}{q{\sqq}}[S_1\oplus X];
\end{align*}
\begin{align*}
[S_1]*[S_2]*[S_1]&= {\sqq}^{\langle S_1,S_2\rangle_Q} \big([S_1\oplus S_2] + (q-1)[X] \big)*[S_1]\\
&= {\sqq}^{\langle S_1,S_2\rangle_Q} {\sqq}^{\langle S_1\oplus S_2,S_1\rangle_Q}
\Big(\frac{1}{q}[S_1\oplus S_1\oplus S_2]+ \frac{q-1}{q}[S_1\oplus X] \Big)\\
&= \frac{1}{q}[S_1\oplus S_1\oplus S_2]+ \frac{q-1}{q}[S_1\oplus X];
\\
[S_2]*[S_1]*[S_1]&= {\sqq}^{\langle S_2,S_1\rangle_Q}[S_1\oplus S_2]*[S_1]\\
&= {\sqq}^{\langle S_2,S_1\rangle_Q} {\sqq}^{\langle S_1\oplus S_2,S_1\rangle_Q} \frac{1}{q}[S_1\oplus S_1\oplus S_2]\\
&= \frac{1}{{\sqq}}[S_1\oplus S_1\oplus S_2].
\end{align*}
The first identity \eqref{eq:Serre112} follows from combining the above computations.

On the other hand, we have
\begin{align*}
[S_2]*[S_2]*[S_1]&= {\sqq}^{\langle S_2,S_1\rangle_Q} [S_2]*[S_2\oplus S_1]\\
&= {\sqq}^{\langle S_2,S_1\rangle_Q}{\sqq}^{\langle S_2,S_1\oplus S_2\rangle_Q}
\Big(\frac{1}{q}[S_2\oplus S_2\oplus S_1]+\frac{q-1}{q}[\E_2\oplus S_1] \Big)\\
&= \frac{1}{{\sqq}}[S_2\oplus S_2\oplus S_1]+\frac{q-1}{{\sqq}}[\E_2\oplus S_1];
\end{align*}
\begin{align}
\label{eqn:S212}
[S_2]*[S_1]*[S_2]&= {\sqq}^{\langle S_2,S_1\rangle_Q} [S_2\oplus S_1]*[S_2]\\
&= {\sqq}^{\langle S_2,S_1\rangle_Q}{\sqq}^{\langle S_2\oplus S_1,S_2\rangle_Q} \notag
\Big(\frac{1}{q}[S_2\oplus S_1\oplus S_2]+\frac{q-1}{q}[\E_2\oplus S_1] \\\notag
& \qquad +\frac{q-1}{q}[S_2\oplus X]+\frac{(q-1)^2}{q}[\rad (U_3)] \Big)\\\notag
&= \frac{1}{q}[S_2\oplus S_1\oplus S_2]+\frac{q-1}{q}[\E_2\oplus S_1]
  +\frac{q-1}{q}[S_2\oplus X]+\frac{(q-1)^2}{q}[\rad (U_3)] \notag
 \\
&= \frac{1}{q}[S_2\oplus S_1\oplus S_2]+(q-1)[\E_2\oplus S_1] +\frac{q-1}{q}[S_2\oplus X],\notag
\end{align}
where we have used $[\rad(U_3)]=[\E_2\oplus S_1]$ in $\tMH$ thanks to a short exact sequence $0\rightarrow \E_2\rightarrow \rad(U_3)\rightarrow S_1\rightarrow0$ with $\E_2 \in \cp^{\leq 1}(\La^\imath)$ by Lemma~\ref{lemma locally projective modules}. In addition, \eqref{eqn:S212} also implies that
\[
[S_1\oplus S_2]*[S_2]=\frac{1}{q}[S_2\oplus S_1\oplus S_2]+(q-1)[\E_2\oplus S_1] +\frac{q-1}{q}[S_2\oplus X].
\]
Then we have
\begin{align*}
[S_1]*[S_2]*[S_2]&= {\sqq}^{\langle S_1,S_2\rangle_Q} ([S_1\oplus S_2]*[S_2]+ (q-1)[X]*[S_2])\\
&= {\sqq}^{\langle S_1,S_2\rangle_Q}{\sqq}^{\langle S_2\oplus S_1,S_2\rangle_Q}
\Big(\frac{1}{q}[S_2\oplus S_1\oplus S_2]+(q-1)[\E_2\oplus S_1]  \\
& \qquad\quad +\frac{q-1}{q}[S_2\oplus X] +(q-1)[X\oplus S_2] \Big)\\
&= \frac{1}{q{\sqq}}[S_2\oplus S_1\oplus S_2]+\frac{q-1}{{\sqq}}[\E_2\oplus S_1] +\frac{q^2-1}{q{\sqq}}[S_2\oplus X].
\end{align*}
Summarizing, we have obtained
\begin{align*}
[S_2]*[S_2]*[S_1]& -({\sqq}+{\sqq}^{-1})[S_2]*[S_1]*[S_2]+[S_1]*[S_2]*[S_2] \\
&=  -\frac{(q-1)^2}{{\sqq}}[\E_2\oplus S_1]
 = -\frac{(q-1)^2}{{\sqq}}[ S_1]*[\E_2],
\end{align*}
whence the identity~\eqref{eq:Serre221}.
The proof is completed.
\end{proof}

\begin{remark}
Using the opposite quiver $1\xleftarrow{\alpha} 2\xrightarrow{\beta}3$ in Proposition \ref{prop:iA3} yields the same formulas.
\end{remark}

\begin{proposition}
\label{prop:cartan}
Let $Q=(Q_0,Q_1)$ be the quiver such that $Q_0=\{1,2\}$, and $Q_1=\emptyset$. Let $\btau$ be the involution such that $\btau(1)=2$ and $\btau(2)=1$. Then in $\tMH$, we have
\begin{align*}
[S_1]*[S_2]-[S_2]*[S_1]=(q-1)([\E_1]-[\E_2]).
\end{align*}
\end{proposition}

\begin{proof}
In $\tMH$, we have $[S_1]*[S_2]= [S_1\oplus S_2] + (q-1)[\E_1]$, and so
\begin{align*}
[S_1]*[S_2]-[S_2]*[S_1]&= [S_1\oplus S_2] + (q-1)[\E_1]-([S_1\oplus S_2]+(q-1)[\E_2])\\
&=(q-1)([\E_1]-[\E_2]).
\end{align*}
The proposition is proved.
\end{proof}

\subsection{The homomorphism $\widetilde{\psi}$}
  \label{subsec:mor}

Recall that $\ci$ is the subset of $\I$ defined in \eqref{eqn:representative}, and $\sqq =\sqrt{q}$. We denote by
\[
\tUi_{|v={\sqq}} =\Q(\sqq)\otimes_{\Q(v)} \tUi
\]
the specialization of $\tUi$ at $v=\sqq$, an algebra over $\Q(\sqq)$; similar notations will be used for specializations at $v=\sqq$ for other algebras. The following is a main step toward identifying the $\imath$Hall algebras and $\imath$quantum groups.

\begin{proposition}
   \label{prop:qHall=Ui}
Let $(Q, \btau)$ be a Dynkin $\imath$quiver. Then there exists a $\Q({\sqq})$-algebra homomorphism
\begin{align}
   \label{eqn:psi morphism}
\widetilde{\psi}: \tUi_{|v={\sqq}} &\longrightarrow \tMH,
\end{align}
which sends
\begin{align}
B_i \mapsto \frac{-1}{q-1}[S_{i}],\text{ if } i\in\ci,
&\qquad\qquad
\tk_j \mapsto - q^{-1}[\E_j], \text{ if }\btau j=j;
  \label{eq:split}
\\
B_{i} \mapsto \frac{{\sqq}}{q-1}[S_{i}],\text{ if }i\notin \ci,
&\qquad\qquad
\tk_j \mapsto [\E_j],\quad \text{ if }\btau j\neq j.
  \label{eq:extra}
\end{align}
\end{proposition}

\begin{proof}
The proof is reduced to rank $2$ $\imath$-subquivers, thanks to Lemma \ref{lem:subalgebra}. The relevant rank $2$ $\imath$-subquivers are listed in Example~\ref{example 2-cyclic complex of An}, which are exactly those studied in \S\ref{subsec:rank2I}--\ref{subsec:rank2II}.

We proceed the proof case-by-case.
First assume $\btau=\Id$, and so $\ci=\I$. For the split $\imath$quantum group $\tUi$, the generators appear in \eqref{eq:split} and the defining relations \eqref{relation1}-\eqref{relation2} can be simplified to be:
\begin{eqnarray}
&&\label{eqn:relation split 3} [\tk_i,\tk_j]=0, \,\,\tk_iB_j=B_j\tk_i, \quad \forall   i,j \in \I; \\
&&\label{eqn:relation split 1}B_iB_j=B_jB_i, \quad \text{ if }c_{ij}=0; \\
&&\label{eqn:serre split 1} B_i^2B_j-(v+v^{-1})B_iB_jB_i+B_jB_i^2 =v \tk_iB_j, \quad\text{ if }c_{ij}=-1.
\end{eqnarray}
It suffices to verify that $\widetilde{\psi}$ preserves the relations \eqref{eqn:relation split 3}--\eqref{eqn:serre split 1}.
As we consider the specialization $\tUi_{|v={\sqq}}$, we regard $v={\sqq}$ in these relations in this proof.

Since $\E_i$ lies in the center of $\tMH$ by Proposition \ref{Prop:centralMH}, $\widetilde{\psi}$ preserves the relation \eqref{eqn:relation split 3}.

If $c_{ij}=0$, then there is no arrow between $i$ and $j$ in $Q$ and also in $\ov{Q}$ by Proposition~\ref{prop:invariant subalgebra}. So we have
\[
[S_i]*[S_j]=[S_i\oplus S_i]=[S_j]*[S_i].
\]
Then $\widetilde{\psi}$ preserves the relation \eqref{eqn:relation split 1}.

On the other hand, if $c_{ij}=-1$, then there exists one arrow $\alpha$ between $i$ and $j$. 
It follows from Proposition \ref{prop:A2} that
\begin{equation*}
[S_i]*[S_i]*[S_j]-(\sqq+\sqq^{-1})[S_i]*[S_j]*[S_i]+[S_j]*[S_i]*[S_i]=-\frac{(q-1)^2}{{\sqq}}[S_j]*[\E_i].
\end{equation*}
So $\widetilde{\psi}$ preserves the relation \eqref{eqn:serre split 1}.
This completes the proof in the case when $\btau=\Id$.

\medskip
Now assume that $\btau\neq \Id$. There are 3 cases of type $ADE$, and the uniform computation depends only on the local configuration of the rank 2 $\imath$subquivers. For the sake of being concrete, we choose to present the detailed proof for type $A_{2r+1}$ below. The type $D$ and $E$ cases follow in the same manner (as there is no new rank 2 cases beyond split and quasi-split type A). It suffices to check that $\widetilde{\psi}$ preserves the relations \eqref{relation A odd 1}-\eqref{relation A odd 10}.

We obtain from Lemma \ref{lemma multiplcation of locally projective modules} that $[\E_i]*[\E_j]=[\E_j]*[\E_i]$, whence
\eqref{relation A odd 1}.

For any $-r\leq i,  j \leq r$, we have
\begin{align*}
[\E_i]*[S_{j}]&=\sqq^{\langle \res(\E_i),\res (S_{j})\rangle_Q} q^{-\langle \E_i, S_{j}\rangle}[\E_i\oplus S_{j}]\\
&=\sqq^{\langle S_{-i},S_{j}\rangle_Q-\langle S_i,S_{j}\rangle_Q}[\E_i\oplus S_j],
\\
[S_{j}]*[\E_i] &=\sqq^{\langle S_{j},S_i\rangle_Q-\langle S_{j},S_{-i}\rangle_Q}[\E_i\oplus S_{j}].
\end{align*}
Hence we have
\begin{align*}
[\E_i]*[S_{j}]&= \sqq^{\langle S_{-i},S_{j}\rangle_Q-\langle S_i,S_{j} \rangle_Q-\langle S_{j},S_i\rangle_Q+\langle S_{j},S_{-i}\rangle_Q}[S_{j}]*[\E_i]\\
&= \sqq^{(S_{-i},S_{j})-(S_i,S_{j})} [S_{j}]*[\E_i],
\end{align*}
whence \eqref{relation A odd 2}.

It follows from Proposition~\ref{prop:cartan} that $\widetilde{\psi}$ preserves the relation \eqref{relation A odd 4}. It follows from Proposition~ \ref{prop:iA3} that $\widetilde{\psi}$ preserves \eqref{relation A odd 9}--\eqref{relation A odd 10}.

Let $i, j$ be such that $-r \leq i,j\leq r, |i-j|>1, i\neq -j.$ We have
\[
[S_i]*[S_j]=\sqq^{\langle S_i,S_j\rangle_Q} [S_i\oplus S_j]=[S_i\oplus S_j]
\]
since $\Ext^1_{\Lambda^{\imath}}(S_i,S_j)=0$.
Similarly, we have $[S_j]*[S_i] =[S_i\oplus S_j]$. Hence $[S_i]*[S_j] =[S_j]*[S_i]$, whence \eqref{relation A odd 5}.

Let $i, j$ be such that $-r \le i,j\le r,|i-j|=1, i\neq 0$. Then $c_{ij}=-1$. So there exists an arrow between $i$ and $j$, we only give a proof when the arrow is of the form $\alpha:i\rightarrow j$ while skipping the other similar case. Denote by $X$ the indecomposable module with $\widehat{S_i}+\widehat{S_j}$ as its class in $K_0(\mod(\Lambda^{\imath}))$. We have
\begin{align*}
[S_i]*[S_i]*[S_j]&= [S_i]*\sqq^{\langle S_i, S_j\rangle_Q}([S_i\oplus S_j]+(q-1)[X])\\
&= \frac{1}{q\sqq}[S_i\oplus S_i\oplus S_j] +\frac{q-1}{q\sqq}[S_i\oplus X]+\frac{q-1}{\sqq}[S_i\oplus X];
\\ %
[S_i]*[S_j]*[S_i]&= \sqq^{\langle S_i, S_j\rangle_Q}([S_i\oplus S_j]+(q-1)[X])*[S_i]\\
&= \frac{1}{q}[S_i\oplus S_i\oplus S_j]+\frac{(q-1)}{q}[X\oplus S_i];
\\ %
[S_j]*[S_i]*[S_i]&= \sqq^{\langle S_j, S_i\rangle_Q}([S_i\oplus S_j]*[S_i]\\
&= \frac{1}{\sqq}[S_i\oplus S_i\oplus S_j].
\end{align*}
So we obtain that $[S_i]*[S_i]*[S_j]-(\sqq+\sqq^{-1}) [S_i]*[S_j]*[S_i]+[S_j]*[S_i]*[S_i]=0$, whence \eqref{relation A odd 7}.

The proposition is proved.
\end{proof}

We shall show that $\widetilde{\psi}$ is actually an isomorphism, cf. Theorem~\ref{theorem isomorphism of Ui and MRH}.

\subsection{$\imath$Quantum groups via $\imath$Hall algebras}
   \label{subsec:iQG=iH}

Let $(Q, \btau)$ be a Dynkin $\imath$quiver, where $Q_0 =\I$. Recall $\tUi$ is the $\imath${}quantum group with parameters $\bvs=(\vs_i)_{i\in\I}$, and $\Ui$ is a quotient algebra of $\tUi$ by the ideal $(\tk_i-\vs_i, \tk_j\tk_{\btau j}-\vs_i\vs_{\tau i}\mid \btau i=i, \btau j\neq j)$; cf. Proposition~\ref{prop:QSP12}.

For any $w=i_1\cdots i_m\in\cw_\I$, define
\[
F_w=F_{i_1}\cdots F_{i_m} \in \tU^-, \qquad
B_w=B_{i_1}\cdots B_{i_m} \in \tUi.
\]
Let $\cj$ be a fixed subset of $\cw_\I$ such that
$\{F_w\mid w\in\cj\}$ is a (monomial) basis of $ \tU^-$.

\begin{lemma}  [\cite{Let99}]   \label{lem:kob}
Retain the notation as above. Then $\{B_w\mid w\in\cj\}$ is a basis of $\tUi$ as the left (or right) $\tU^{\imath 0}$-module. (This basis is called {\rm a monomial basis}.)
\end{lemma}

Recall from Definition~\ref{def:torus} or from \S\ref{subsec:MH}  the twisted quantum torus $\tTL$, which is a $\Q(\sqq)$-algebra. Recall that $\widetilde{\psi}: \tUi_{|v={\sqq}} \rightarrow \tMH$ is defined in Proposition~\ref{prop:qHall=Ui}.
Then the homomorphism $\widetilde{\psi}: \tUi_{|v={\sqq}} \rightarrow \tMH$ induces  an algebra homomorphism
\begin{align*}
\widetilde{\psi}:\widetilde{\bU}^{\imath 0}_{|v={\sqq}} &\longrightarrow  \tTL,
  \\
\tk_i\mapsto - q^{-1} [\E_i], \text{ if }\btau i=i,
&\qquad
\tk_i\mapsto [\E_i], \text{ if }\btau i\neq i.
\end{align*}
Since both $\widetilde{\bU}^{\imath 0}_{|v={\sqq}}$ and $\tTL$ are Laurent polynomial algebras in the same number of generators, $\widetilde{\psi}:\widetilde{\bU}^{\imath 0}_{|v={\sqq}} \rightarrow \tTL$ is an isomorphism.

Recall the reduced $\imath$Hall algebra $\rMH$ from Definition~\ref{def:reducedHall}. We now state the main result of this section.

\begin{theorem}  \label{theorem isomorphism of Ui and MRH}
Let $(Q, \btau)$ be a Dynkin $\imath$quiver. Then we have the following isomorphism of $\Q({\sqq})$-algebras, see \eqref{eqn:psi morphism}:
\begin{align*}
\widetilde{\psi}:\tUi_{|v={\sqq}}\stackrel{\simeq}{\longrightarrow} \tMH.
\end{align*}
Moreover, it induces an isomorphism $\psi: \Ui_{|v={\sqq}}\stackrel{\simeq}{\rightarrow} \rMH$, which sends $B_i$ as in \eqref{eq:split}--\eqref{eq:extra} and  $k_j \mapsto  \vs_j^{-1}[\E_j], \text{ if } j \in \I\setminus\ci$.
\end{theorem}

\begin{proof}
By  Proposition~ \ref{monomial basis of HQ}, $\tH(\K Q)$ has a monomial basis, i.e., there exists a subset $\cj$ of $\cw_\I$ such that $\{\ov{S}^{*}_w\mid w\in\cj\}$ is a basis of $\tH(\K Q)$.
By \cite[Theorem 7]{Rin2}, there exists an isomorphism of algebras: $R^-: \bU^-_{|v={\sqq}} \xrightarrow{\simeq} \tH(\K Q)$, with $R^-(F_i)=\frac{{-1}}{q-1}[S_i]$ for any $i\in \I$.
So $\{F_w\mid w\in\cj\}$ is a monomial basis of $\U^-$.

By Lemma \ref{lem:kob}, $\{B_w\mid w\in\cj\}$ is a basis of $\tUi$ as a right $\tU^{\imath 0}$-module.
It follows by Theorem \ref{thm:monomial} that $\{S^*_w\mid w\in\cj\}$ is a basis of $\tMH$ as a right $\tTL$-module. Recall $\widetilde{\psi}:\tU^{\imath 0}|_{v=\sqq}\xrightarrow{\simeq} \tTL$. Therefore, for any $w\in\cw_\I$, $\psi(B_w)=a_wS^*_{w}$ for some scalar $a_w\in\Q(\sqq)^\times$, and thus $\widetilde{\psi}: \tUi|_{v=\sqq}\rightarrow \tMH$ is an isomorphism of algebras.

As the isomorphism $\widetilde{\psi}: \tUi\rightarrow \tMH$ sends the ideal $(\tk_i- \vs_i, \tk_j\tk_{\btau j}-\vs_i\vs_{\btau i} \mid \btau i=i, \btau j\neq j)$ generated by \eqref{eq:parameters} onto the ideal of $\tMH$ generated by \eqref{eqn: reduce 1}, it induces an isomorphism $\psi: \Ui|_{v=\sqq}\xrightarrow{\simeq} \rMH$ as stated.
\end{proof}

\begin{remark}
A variant of Bridgeland's Hall algebra via the module category $\mod (\K Q\otimes R_1)$ is studied in an interesting paper by H.~Zhang \cite{Zh18} independent of our work, who established a connection to $\U^+$.
\end{remark}

We expect the following generalization of Theorem~\ref{theorem isomorphism of Ui and MRH} for general quivers. 

\begin{conjecture}  \label{conj:mono}
Let $(Q, \btau)$ be an arbitrary {\em acyclic} $\imath$quiver. Then we have an injective homomorphism of $\Q({\sqq})$-algebras $\widetilde{\psi}:\tUi_{|v={\sqq}}\longrightarrow \tMH$  defined as in \eqref{eqn:psi morphism}--\eqref{eq:extra}.
Moreover, it induces an injective homomorphism $\psi: \Ui_{|v={\sqq}}\longrightarrow \rMH$, which sends $B_i$ as in \eqref{eq:split}--\eqref{eq:extra} and  $k_j \mapsto  \vs_j^{-1}[\E_j], \text{ if } j \in \I\setminus\ci$.
\end{conjecture}

\section{Bridgeland's theorem revisited}
  \label{sec:Bridgeland}

In this section, we show that the Hall algebra associated to the $\imath$quiver of diagonal type $(Q^{\dbl}, \swa)$ is isomorphic to the specialization at $v={\sqq}$ of a quantum group. This is a reformulation of Bridgeland's Hall algebra construction of quantum groups.

\subsection{A category equivalence}

Let $Q$ be an acyclic quiver (not necessarily of finite type). Let $Q^{\dbl} =Q\sqcup Q^\diamond$,  where $Q^\diamond$ is an identical copy of a quiver $Q$. Retain the notation as in Example \ref{example diagonal} and Example~ \ref{example diagonal 2}. Recall that $\swa$ is the natural involution of $Q^{\dbl}$. As explained in Example~ \ref{example diagonal 2} we can and shall identify $\La$ as the $\imath$quiver algebra $\Lambda=(\Lambda^{\dbl})^\imath$ with $(Q^{\sharp},I^{\sharp})$ as its bound quiver, throughout this section.

Clearly, $D^b(\K Q^{\dbl})\simeq D^b(\K Q\times \K Q^\diamond)$. Let $\Psi_{\swa}$ be the triangulated autoequivalence functor of $D^b(\K\cq)$ induced by ${\swa}$. We also use $\Sigma$ to denote the shift functor in $D^b(kQ^{\dbl})$.

\begin{lemma}
  \label{lem:equivCat}
We have an equivalence of categories  $D^b(\K\cq)/\Sigma\circ \Psi_{\swa}\simeq D^b(\K Q)/\Sigma^2$.
\end{lemma}

\begin{proof}
Any $\K\cq$-module is of the form $(M,M')$, where $M\in \mod (\K Q)$, $M'\in\mod (\K Q^\diamond)$. Similarly,  any object in $D^b(\K\cq)$ is of the form $(M,M')$, where $M\in D^b(\K Q)$, $M'\in D^b(\K Q^\diamond)$. Obviously,
\[
\Hom_{D^b(\K\cq)}((M,M'),(N,N'))= \Hom_{D^b(\K Q)}(M,N)\times \Hom_{D^b(\K Q^\diamond)}(M',N').
\]
So any morphism between them is of the form $(f,f')$, where $f:M\rightarrow N$, $f':M'\rightarrow N'$.
By identifying $D^b(\K Q^\diamond)$ with $D^b(\K Q)$, we define a triangulated functor
\[
G:D^b(\K\cq)\longrightarrow D^b(\K Q),
\]
which sends the object $(M,M')\mapsto M\oplus \Sigma M'$ and sends the morphism $(f,f'):(M,M')\rightarrow (N,N')$ to $\diag(f,\Sigma f'): M\oplus \Sigma M'\rightarrow N\oplus \Sigma N'$.

Combining with the natural projection $\pi_{Q}: D^b(\K Q)\rightarrow D^b(\K Q)/\Sigma^2$, $G$ induces a triangulated functor
$\tilde{G}:D^b(\K\cq)\rightarrow D^b(\K Q)/\Sigma^2$.

Note that $\Psi_{\swa}((M,M'))=(M',M)$ for $(M,M')\in D^b(\K\cq)$, and $\Psi_{\swa}((f,f'))=(f',f)$ for any morphism
$(f,f')$ in $D^b(\K\cq)$. It follows that $\tilde{G}\circ(\Sigma\circ\Psi_{\swa})\cong \tilde{G}$.
Then \cite[\S9.4]{Ke2} shows that there exists a triangulated functor
$$\bar{G}: D^b(\K\cq)/\Sigma\circ \Psi_{\swa} \longrightarrow D^b(\K Q)/\Sigma^2$$
such that the following diagram commutes:
\[
\xymatrix{ D^b(\K\cq) \ar[r]^{\tilde{G}} \ar[d]^{\pi_{\cq}}& D^b(\K Q)/\Sigma^2
  \\
 D^b(\K\cq)/\Sigma\circ \Psi_{\swa} \ar[ur]_{\bar{G}} & }
 \]
Note that $\bar{G}$ is dense.

Concerning the morphism spaces, we have
\begin{align*}
&\Hom_{ D^b(\K\cq)/\Sigma\circ \Psi_{\swa} } ((M,M'),(N,N'))\\
=&\bigoplus_{\ell\in\Z}\Hom_{D^b(\K\cq)}((M,M'),(\Sigma\circ \Psi_{\swa})^\ell(N,N'))\\
=&\bigoplus_{\ell\in\Z} \Hom_{D^b(\K\cq)}((M,M'),(\Sigma^{2\ell}N,\Sigma^{2\ell}N'))\\
&\bigoplus \bigoplus_{\ell\in\Z}  \Hom_{D^b(\K\cq)}((M,M'),(\Sigma^{2\ell+1}N',\Sigma^{2\ell+1}N))\\
=&\bigoplus_{\ell\in\Z} \Hom_{D^b(\K Q)}(M,\Sigma^{2\ell}N\oplus \Sigma^{2\ell+1}N')\oplus  \Hom_{D^b(\K Q)}(M',\Sigma^{2\ell}N'\oplus \Sigma^{2\ell+1}N)\\
\cong&\bigoplus_{\ell\in\Z}\Hom_{D^b(\K Q)}(M\oplus \Sigma M',\Sigma^{2\ell}N\oplus \Sigma^{2\ell+1} N')\\
=&\Hom_{D^b(\K Q)/\Sigma^2}(M\oplus \Sigma M',N\oplus \Sigma N').
\end{align*}
Therefore, $\bar{G}$ is fully faithful.
This proves the lemma.
\end{proof}

\begin{remark}
Theorem~ \ref{thm:sigma} for $\imath$quivers of diagonal type reads that
\[
\underline{\Gproj}(\Lambda)\simeq D_{sg}(\mod(\Lambda))\simeq D^b(\K\cq)/\Sigma\circ \Psi_{\swa}.
\]
This together with Lemma~\ref{lem:equivCat} recovers the results  on root categories in \cite{PX} (see also  \cite{Lu}) that $\underline{\Gproj}(\Lambda)\simeq D_{sg}(\mod(\Lambda))\simeq D^b(\K Q)/\Sigma^2$.
\end{remark}

\subsection{Quantum groups as $\imath$quantum groups}

Let $Q$ be an acyclic quiver with its vertex set $\I$. Recall from \S\ref{subsection Quantum groups} that $\bU=\bU_v(\fg)$ is the quantum group associated to $Q$ and $\tU =\langle E_i, F_i, \tK_i, \tK_i'\mid i\in \I  \rangle$ is the version of $\U$ with enlarged Cartan subalgebra. Consider the $\Q(v)$-subalgebra $\tUUi$ of $\tUU$
generated by
\[
\ck_i:=\tK_{i} \tK_{i^{\diamond}}', \quad
\ck_i':=\tK_{i^{\diamond}} \tK_{i}',  \quad
\cb_{i}:= F_{i}+ E_{i^{\diamond}} \tK_{i}', \quad
\cb_{i^{\diamond}}:=F_{i^{\diamond}}+ E_{i} \tK_{i^{\diamond}}',
\qquad \forall i\in \I.
\]
Here we drop the tensor product notation and use instead $i^\diamond$ to index the generators of the second copy of $\tU$ in $\tUU$ (consistent with the notation $Q^{\dbl}=Q\sqcup Q^\diamond$).
Note that $\ck_i\ck_{i}'$ are central in $\tUU$ for all $i\in\I$.

\begin{lemma}
   \label{lem:diagonalQSP2}
 (1)  $\tUUi$ is a right coideal subalgebra of $\tUU$. 

(2) There exists a $\Q(v)$-algebra isomorphism $\widetilde{\phi}: \tU \rightarrow \tUUi$ such that
\[
\widetilde{\phi}(F_i)= \cb_{i},\quad \widetilde{\phi}(E_i)= \cb_{i^{\diamond}}, \quad \widetilde{\phi}(\tK_i)= \ck_i, \quad \widetilde{\phi}(\tK_i')= \ck_i', \qquad \forall  i\in \I.
\]
\end{lemma}

\begin{proof}
(1) Follows by a direct computation using the comultiplication $\Delta$ in \eqref{eq:Delta}.

(2) Recall $\omega$ is the Chevalley involution of $\tU$ given in \eqref{eq:omega} and $\Delta$ is the comultiplication given in \eqref{eq:Delta}.
A direct computation shows that the subalgebra $\tUUi \subset \tUU$ is identified with the homomorphic image of the injective homomomorphism
$(\omega \otimes 1) \circ \Delta \circ\omega: \tU \rightarrow \tUU$,
which sends
$\tK_i \mapsto \ck_i, \tK_i' \mapsto \ck_i', E_i \mapsto \cb_{i^\diamond}, F_i\mapsto \cb_{i}$;
 this is a variant of the observation in \cite[Remark 4.10]{BW18b}. Setting $\widetilde{\phi} =(\omega \otimes 1) \circ \Delta  \circ\omega$ finishes the proof.
\end{proof}

Let $\bvs=(\vs_i)_{i\in \I}\in(\Q(v)^\times)^{\I}$. We define the subalgebra $\UUi$ of $\UU$ to be the one generated by
\[
k_i:= K_{i}K_{i^{\diamond}} ^{-1} , \quad
k_i^{-1}= K_{i^{\diamond}} K_{i} ^{-1},  \quad
B_{i}:= F_{i}+\vs_iE_{i^{\diamond}} K_{i}^{-1}, \quad
B_{i^{\diamond}} :=F_{i^{\diamond}}+ \vs_iE_{i} K_{i^{\diamond}}^{-1}, \qquad \forall i\in \I.
\]
Here we drop the tensor product notation.

\begin{lemma}
   \label{lem:diagonalQSP}
(1)  $\UUi$ is a right coideal subalgebra of $\UU$. 

(2) We have a $\Q(v)$-algebra isomorphism
\begin{align}
\UUi  \longrightarrow & \tUUi / ( \ck_i\ck_{i}' -\vs_i^2 ),
     \label{eq:U=UUi} \\
B_i \mapsto \cb_i, \quad &
 B_{i^\diamond} \mapsto  \cb_{i^\diamond}, \quad
 k_i \mapsto  \vs_i^{-1} \ck_i, \quad
 k_i^{-1} \mapsto  \vs_i^{-1} \ck_i^{-1}.
 \notag
\end{align}

(3) There exists a $\Q(v)$-algebra isomorphism $\phi: \U \rightarrow \UUi$ such that
\[
 F_i \mapsto B_{i},\quad E_i \mapsto \vs_i^{-1} B_{i ^{\diamond}},\quad K_i \mapsto k_i, \qquad \forall  i\in \I.
\]
\end{lemma}

\begin{proof}
Parts (1) and (2) follow by direct computations.

(3) First consider the special case with all $\vs_i=1$. Recall $\omega$ is the Chevalley involution of $\U$ and $\Delta$ is the comultiplication given in \eqref{eq:Delta}. It follows by \cite[Remark 4.10]{BW18b} that the subalgebra $\UUi \subset \U\otimes \U$ is identified with the homomorphic image of
$(\omega \otimes 1) \circ \Delta \circ \omega: \U \rightarrow \U\otimes \U$, which sends $K_i \mapsto k_i, F_i \mapsto B_{i}, E_i\mapsto B_{i^\diamond}$.
The case for general parameters $\bvs$ follows from this special case by a rescaling automorphism.
\end{proof}

\subsection{Bridgeland's theorem reformulated}

\begin{theorem} [Bridgeland's Theorem reformulated; cf. \cite{Br}]
   \label{thm:bridgeland2}
There exists an injective morphism of algebras $\tUUi_{|v={\sqq}} \longrightarrow  \tM(\Lambda)$ such that
\begin{eqnarray*}
\widetilde{\psi} (\ck_i)= [\E_{i}], \quad
\widetilde{\psi} (\ck_i')= [\E_{i^{\diamond}}],   \quad
\widetilde{\psi} (\cb_i)=\frac{-1}{q-1}[S_{i}], \quad
\widetilde{\psi} (\cb_{i^{\diamond}})=\frac{{\sqq}}{q-1}[S_{i^{\diamond}}],
\qquad \forall i\in \I.
\end{eqnarray*}
Equivalently, there exists an injective homomorphism $\widetilde\Psi: \tU_{|v={\sqq}} \rightarrow \tM(\Lambda)$ such that \begin{eqnarray*}
\widetilde{\Psi} (\tK_i)= [\E_{i}], \quad
\widetilde{\Psi} (\tK_i')= [\E_{i^{\diamond}}],   \quad
\widetilde{\Psi} (F_i)=\frac{-1}{q-1}[S_{i}], \quad
\widetilde{\Psi} (E_i)=\frac{{\sqq}}{q-1}[S_{i^{\diamond}}],
\qquad \forall i\in \I.
\end{eqnarray*}
\end{theorem}

\begin{proof}
Thanks to the isomorphism $\widetilde\phi$ in Lemma~\ref{lem:diagonalQSP2}, the two assertions regarding $\widetilde{\psi}$ and $\widetilde{\Psi}$ are equivalent by letting $\widetilde{\Psi}=\widetilde{\psi}\circ \widetilde{\phi}$.

We shall prove that $\widetilde{\Psi}$ is an injective algebra homomorphism.
The proof is similar to the proof of Proposition~\ref{prop:qHall=Ui}, and we shall only check that $\widetilde{\Psi}$ preserves the 2 most complicated relations, i.e., the quantum Serre relations \eqref{eq:serre1}--\eqref{eq:serre2}. In fact, $kQ$ and $kQ^{\diamond}$ are quotient algebras of $\Lambda$, so we can view $\mod(kQ)$ and $\mod(kQ^{\diamond})$ as subcategories of $\mod(kQ^{\dbl})$. Note that these two subcategories are full and closed under taking extensions. So there exist two morphisms
\begin{align*}
I^+:\tH(kQ^{\diamond})\longrightarrow\tM(\Lambda), &\qquad I^+([S_{i^\diamond}])=[S_{i^\diamond}],  \\
I^-:\tH(kQ)\longrightarrow\tM(\Lambda),  &\qquad I^-([S_{i}])=[S_{i}].
\end{align*}

By Ringel-Green's Theorem, we obtain two injective homomorphisms of algebras:
\begin{align*}
&R^+: \tU^+_{|v={\sqq}} \longrightarrow \tH(\K Q), \,\,R^+(E_i)=\frac{\sqq}{q-1}[S_{i^{\diamond}}],
  \\
&R^-: \tU^-_{|{v={\sqq}}} \longrightarrow \tH(\K Q),\,\, R^-(F_i)=\frac{-{1}}{(q-1)}[S_{i}].
\end{align*}
Using $R^+\circ I^+$ and $R^-\circ I^-$, together with the quantum Serre relations of $\tU$, we have
\begin{align*}
&\sum_{r=0}^{1-c_{ij}} (-1)^r \left[ \begin{array}{c} 1-c_{ij} \\r \end{array} \right] \cdot [S_i]^r*[S_j]* [S_i]^{1-c_{ij}-r}=0,
 \\
&\sum_{r=0}^{1-c_{ij}} (-1)^r \left[ \begin{array}{c} 1-c_{ij} \\r \end{array} \right] \cdot [S_{i^{\diamond}}]^r*[S_{j^{\diamond}}] *[S_{i^{\diamond}}]^{1-c_{ij}-r}=0,\quad i\neq j.
\end{align*}
So $\widetilde{\Psi}$ preserves \eqref{eq:serre1}--\eqref{eq:serre2}. Hence $\widetilde{\Psi}$ is an algebra homomorphism.

On the other hand, let $\tT$ be the subalgebra of $\tM(\Lambda)$ generated by $\E_{i},\E_{i^{\diamond}}$, for $i\in \I$. Note that $\tT$ is the $\Q(\sqq)$-group algebra of the Grothendieck group $K_0(\mod (\K Q))\times K_0(\mod(\K Q^\diamond))$. Then $\tM(\Lambda)$ is a $\tT$-bimodule. Corollary \ref{cor:basis of MRH as torus-module} shows that
$\tM(\Lambda)$ is a free right  $\tT$-module with a basis given by $[M\oplus M']$, where $M\in\mod (\K Q)$ and $M'\in\mod (\K Q^{\diamond})$. Similar to Lemma~ \ref{lemma embedding of algebras}, one can prove that $[M]*[M']*\E_\alpha* \E_\beta$, $M\in\mod (\K Q)$, $M'\in\mod (\K Q^{\diamond})$, $\alpha\in K_0(\mod(\K Q))$, $\beta\in K_0(\mod(\K Q^\diamond))$  is a basis of  $\tM(\Lambda)$; see also \cite[Theorem 3.20]{LP}.
Therefore, the multiplication gives rise to a linear isomorphism $\tH(\K Q)\otimes_{\Q(\sqq)} \tT\otimes_{\Q(\sqq)} \tH(\K Q)\cong\tM(\Lambda)$.

 Clearly, we also have an isomorphism of $\Q(\sqq)$-algebras:
\[
R^0:\tU^0_{|v={\sqq}} \stackrel{\simeq}{\longrightarrow} \tT,\quad R^0(\tK_i)=[\E_{i}],\,\, R^0(\tK_i')=[\E_{i^\diamond}].
\]
Recall $\tU=\tU^+\otimes \tU^0\otimes \tU^-$.
Composing $R^+\otimes R^0\otimes R^-$ with the isomorphism $\tH(\K Q)\otimes_{\Q(\sqq)} \tT\otimes_{\Q(\sqq)} \tH(\K Q) \stackrel{\simeq}{\rightarrow} \tM(\La)$ gives the injective homomorphism $\widetilde{\Psi}$.
\end{proof}

Let $\bvs =(\vs_i\mid i\in \I) \in (\Q(\sqq)^{\times})^\I$. Let $\rM(\Lambda)$ be the reduced $\imath$Hall algebra for $\La$ (or the reduced twisted semi-derived Ringel-Hall algebra of $\Lambda$), i.e., the quotient algebra of $\tM(\La)$ by the ideal generated by the central elements $[\E_{i}]*[\E_{i^{\diamond}}] - \vs_i^2$, for all $i \in \I$.

\begin{proposition} [Bridgeland's theorem reformulated; cf. \cite{Br}]
   \label{prop:bridgeland}
There exists an injective homomorphism $\psi: \UUi_{|v={\sqq}} \longrightarrow  \rM(\Lambda)$ such that
\begin{eqnarray*}
\psi(k_i)=\frac{1}{\vs_i}[\E_{i}],   \quad
\psi(B_i)=\frac{-1}{q-1}[S_{i}], \quad
\psi(B_{i^\diamond})=\frac{{\sqq}}{q-1}[S_{i^{\diamond}}], \qquad \forall i\in \I.
\end{eqnarray*}
Equivalently, there exists an injective homomorphism $\Psi: \bU_{|v={\sqq}} \rightarrow \rM(\Lambda)$ such that
\begin{eqnarray*}
\Psi(K_i)=\frac{1}{\vs_i}[\E_{i}],   \quad
\Psi(F_i)=\frac{-1}{q-1}[S_{i}], \quad
\Psi(E_{i^\diamond})=\frac{{\sqq}}{\vs_i(q-1)}[S_{i^{\diamond}}], \qquad \forall i\in \I.
\end{eqnarray*}
\end{proposition}

\begin{proof}
The isomorphism $\widetilde{\psi}$ in Theorem \ref{thm:bridgeland2} induces an isomorphism
\[
\tUUi_{|v={\sqq}} \big/ ( \ck_i\ck_{i}' -\vs_i^2 ) \stackrel{\simeq}{\longrightarrow} \tMH \big/ ( \E_i\E_{i^\diamond} -\vs_i^2 ).
\]
This gives us the isomorphism $\psi$ by Lemma~ \ref{lem:diagonalQSP}. Then $\Psi :=\psi\circ \phi$ (where $\phi$ is the isomorphism in Lemma~ \ref{lem:diagonalQSP}) provides the desired map in the second assertion.
\end{proof}
Ringel-Green's Theorem \cite{Rin, Gr} implies that the homomorphisms $\Psi$ and $\psi$ in Theorem~\ref{thm:bridgeland2} and Proposition~ \ref{prop:bridgeland} are isomorphisms if and only if $Q$ is Dynkin.

Recall $\rM(\La)$ depends on a parameter $\bvs \in(\Q(\sqq)^\times)^{\I}$.
Let ${\bf 1}$ denote the distinguished parameter ${\bf 1} =({\bf 1}_{i} \mid i\in \I)$ with ${\bf 1}_{i} =1$ for all $i\in \I$. We use the index ${\bf 1}$ to indicate the algebras with parameter $\bf 1$ are under consideration. Note $\Psi_{{\bf 1}}$ in Proposition~ \ref{prop:bridgeland} is the morphism obtained in \cite[Proposition 9.26]{Gor1}; compare \cite[Theorem 4.9]{Br}.

If $Q$ is of finite type, by Proposition~ \ref{prop:bridgeland} we have  that $\rM(\Lambda) \cong\rM(\Lambda)_{{\bf 1}}$ (where we recall the parameters $\bvs$ in $\rM(\Lambda)$ are arbitrary). For arbitrary $Q$, let ${\F}=\Q({\sqq})(a_i\mid i\in \I)$ be a field extension of $\Q({\sqq})$, where $a_i=\sqrt{\vs_i}$ for $i\in \I$.
Denote by $_{\F}\rM(\Lambda) =\F \otimes_{\Q({\sqq})} \rM(\Lambda)$ the $\F$-algebra obtained by a base change.

\begin{proposition}
  \label{prop:morphism}
There exists an isomorphism of ${\mathbb F}$-algebras
\begin{align*}
\varphi:& {}_{\F}\rM(\Lambda)_{{\bf 1}} \longrightarrow {}_{\F}\rM(\Lambda),
\\
&[M] \mapsto \prod_{i\in\I}a_i^{-\dim_\K( M_i)-\dim_\K(M_{i^{\diamond}})}[M],\quad\forall M=(M_{i},M_{i^{\diamond}},M(\alpha))\in\mod(\Lambda).
\end{align*}
\end{proposition}

\begin{proof}
For any $M$, $N\in \mod(\Lambda)$, we have
\[
[M]*[N]={\sqq}^{\langle \res(M),\res(N)\rangle_{Q\sqcup Q^{\diamond}}}\sum_{[L]\in \Iso(\mod(\Lambda))}\frac{|\Ext^1_{\Lambda}(M,N)_L|}{|\Hom_{\Lambda}(M,N)|}[L].
\]
If $|\Ext^1_{\Lambda}(M,N)_L|\neq0$, then $\dim_k L =\dim_k M +\dim_k N$. So the rescaling map
\[
\widetilde \varphi: {}_{\F}\tM(\Lambda) \longrightarrow {}_{\F}\tM(\Lambda),
\quad [M] \mapsto \prod_{i\in\I}a_i^{-\dim_\K( M_i)-\dim_\K(M_{i^{\diamond}})}[M],
\]
is an algebra isomorphism. It follows that $\widetilde \varphi([\E_i])=\vs_i^{-1} [\E_{i}]$, $\widetilde \varphi(\E_{i^{\diamond}})=\vs_i^{-1} [\E_{i^{\diamond}}]$, and thus
$
\widetilde \varphi([\E_{i}]*[\E_{i^{\diamond}}]-1)=\vs_i^{-2} [\E_{i}]*[\E_{i^{\diamond}}]-1,$ for all $i \in \I.$
Therefore, $\widetilde \varphi$ induces the isomorphism $\varphi: {}_{\F}\rM(\Lambda)_{{\bf 1}} \rightarrow {}_{\F}\rM(\Lambda)$ as desired.
\end{proof}

\section{Generic Hall algebras for Dynkin $\imath$quivers}
  \label{sec:generic}

In this section, we show that the structure constants for the $\imath$Hall algebras $\tMH$ and $\rMH$ for Dynkin $\imath$quivers are Laurent polynomials in ${\sqq}$, which allow us to formulate the generic Hall algebras. We then show that generic Hall algebras are isomorphic to $\imath$quantum groups.

\subsection{Hall polynomials}

In this subsection, we prove the Hall polynomial property for $\Gproj(\Lambda^{\imath})$.

Let $\ca$ be an additive category. A path in $\ca$ is a sequence
$$M_0\stackrel{f_1}{\longrightarrow} M_1\stackrel{f_2}{\longrightarrow} M_2\longrightarrow\cdots \longrightarrow M_{t-1}\stackrel{f_t}{\longrightarrow}M_t$$
of nonzero non-isomorphisms $f_1,\dots,f_t$ between indecomposable objects $M_0,M_1,\dots,M_t$ with $t\geq1$. We call $M_0$ a {\em predecessor} of $M_t$ and $M_t$ a {\em successor} of $M_0$. A path in $\ca$ is called a {\em cycle} if its source $M_0$ is isomorphic to its target $M_t$. An indecomposable object that lies on no cycle in $\ca$ is called a {\em directing object}. The category $\ca$ is called {\em directed} if every indecomposable object is directing.

According to \cite[Chapter I.5]{Ha2}, $D^b(\K Q)$ is a directed category. Furthermore, any indecomposable object in $D^b(\K Q)$ is isomorphic to some $\Sigma^i M$ where $M\in\mod (\K Q)$.
We denote $\Ext^1_{D^b(\K Q)}(M,N)=\Hom_{D^b(\K Q)}(M,\Sigma N)$, for $M,N\in D^b(\K Q)$.

\begin{lemma}
  \label{lemma assumption holds}
Let $Q$ be a Dynkin quiver. Let $F$ be an autoequivalence of $D^b(\K Q)$ such that $F^2\simeq \Sigma^2$. Then $\Ext^1_{D^b(\K Q)}(M,F^\ell N)$ vanishes for all but at most one $\ell\in\Z$, for any indecomposable objects $M,N \in D^b(\K Q)$. In particular, this holds for $F=\Sigma$.
\end{lemma}

\begin{proof}
Without loss of generality, we assume $M\in\mod (\K Q)$ and $\Hom_{D^b(kQ)}(M,N) \neq 0$, but $\Hom_{D^b(kQ)}(M,F^\ell N)=0$ for any $\ell<0$.
As $\K Q$ is hereditary, we obtain that either $N\in\mod(\K Q)$ or $\Sigma^{-1}(N)\in\mod(\K Q)$. Then $\Hom_{D^b(\K Q)}(M,F^\ell N)=0$ for any $\ell>1$.

It remains to show that $\Hom_{D^b(\K Q)}(M,FN) =0$. Suppose $\Hom_{D^b(\K Q)}(M,FN) \neq 0$. 
Then
\[
0\neq \Hom_{D^b(\K Q)}(FM,F^2N)\cong \Hom_{D^b(\K Q)}(FM,\Sigma^2 N),
\]
and so there exists a triangle
\[
\Sigma N\longrightarrow L\longrightarrow FM\longrightarrow \Sigma^2N,
\]
which implies that $\Sigma N$ is a predecessor of $FM$.  As $\Hom_{D^b(kQ)}(M,N)\neq 0$, $M$ is a predecessor of $N$ or $M\cong N$, and then $FM$ is a predecessor of $FN$ or $FM\cong FN$. Choose a slice $\cl$ such that $N\in\cl$. Then $\Hom_{D^b(\K Q)}(\cl,\Sigma\cl)=0$. So any morphism $f: M\rightarrow FN$ factors through some morphism in $\Hom_{D^b(\K Q)}(\cl,\Sigma\cl)$, which implies that $f=0$; this contradicts with the assumption that $\Hom_{D^b(\K Q)}(M,FN)\neq0$.
\end{proof}

For any indecomposable module $X\in\mod(\K Q)\subseteq \mod(\Lambda^\imath)$, there exists a unique (up to isomorphism)  indecomposable $G_X\in\Gproj^{\np}(\Lambda^\imath)$ (cf. \eqref{eq:GProjnp}) such that $G_X\cong X$ in $D_{sg}(\mod(\Lambda^\imath))$.
By Corollary \ref{corollary for stalk complexes} this gives a bijection
\begin{align}
\label{eqn:bijection}
\Ind(\mod(\K Q))&\stackrel{1:1}{\longleftrightarrow} \Ind(\Gproj^{\np}(\Lambda^i)),\qquad
X \mapsto G_X.
\end{align}

Recall that $\Phi^+$ is the set of positive roots of $Q$ with simple roots $\alpha_i$, $i\in Q_0$.
For any $\alpha\in\Phi^+$, denote by $M_q(\alpha)$ its corresponding indecomposable $\K Q$-module, i.e., $\dimv M_q(\alpha)=\alpha$.
Let $\mathfrak{P}:=\mathfrak{P}(Q)$ be the set of functions $\lambda: \Phi^+\rightarrow \N$.
Then the modules
\begin{align}
\label{def:Mlambda}
M_q(\lambda):= \bigoplus_{\alpha\in\Phi^+}\lambda(\alpha) M_q(\alpha),\quad \text{ for } \lambda\in\mathfrak{P},
\end{align}
provide a complete set of isoclasses of $\K Q$-modules.

Let $\Phi^0$ denote a set of symbols $\gamma_i$, i.e., $\Phi^0=\{\gamma_{i}\mid  i\in \I\}$, and denote
\[
\Phi^\imath=\Phi^+\cup \Phi^{0},
\qquad
\mathfrak{P}^{\imath} =\{ \lambda:\Phi^\imath\rightarrow\N\}.
\]
Let $G_q(\alpha)$ be the unique indecomposable Gorenstein projective $\Lambda^\imath$-module such that $G_q(\alpha)\cong M_q(\alpha)$ in $D_{sg}(\mod(\Lambda^\imath))$, for $\alpha\in \Phi^+$; see \eqref{eqn:bijection}.
Set $G_q(\gamma_i)= \Lambda^\imath \, e_i$, for $i\in \I$. Then $\{G_q(\alpha), G_q(\gamma_i) \mid \alpha\in \Phi^+,\gamma_i\in \Phi^0\}$ forms a complete set of isoclasses of indecomposable Gorenstein projective $\Lambda^\imath$-modules.

For any $\lambda\in \mathfrak{P}^{\imath}$, we define a Gorenstein projective $\Lambda^\imath$-module $G_q(\lambda)$ as
\begin{align}
 \label{def: Glambda}
G_q(\lambda):= \bigoplus_{\alpha\in\Phi^+}\lambda(\alpha) G_q(\alpha)\oplus\bigoplus_{\gamma_i\in\Phi^0}\lambda(\gamma_i)G_q(\gamma_i),\quad \text{ for }  \lambda\in\mathfrak{P}^\imath.
\end{align}
Then $G_q(\lambda)$, for $\lambda\in \fpi$, give a complete set of isoclasses of Gorenstein projective $\Lambda^\imath$-modules.

Let $\fp^0$ be the subset of $\fpi$ which consists of functions supported on $\Phi^0$, i.e., $\fp^0 =\{\lambda \in \fpi\mid \lambda(\alpha) =0, \forall \alpha \in \Phi^+\}$, and we identify $\fp$ with the subset of $\fpi$ which consists of functions supported on $\Phi^+$. We view each $x\in\Phi^\imath$ as the characteristic function $f\in\fpi$ defined by $f(y)=\delta_{xy}$ for $y\in \Phi^\imath$.

Recall from  \S\ref{subsec:HA} the Hall number $F_{M,N}^L=\frac{|\Ext^1(M,N)_L|}{|\Hom(M,N)|}$,
for $M,N,L \in \Gproj(\Lambda^{\imath})$, and $\sqq =\sqrt{q}$.

\begin{proposition}
  \label{prop:Hall polynomial}
The Frobenius category $\Gproj(\Lambda^{\imath})$ over the field $\K=\F_q$ satisfies the Hall polynomial property, that is, there exists a polynomial $\bF^\lambda_{\mu,\nu}(v) \in\Z[v,v^{-1}]$  such that
$\bF^\lambda_{\mu,\nu}({\sqq})  = F^{G_q(\lambda)}_{G_q(\mu),G_q(\nu)}$, for all $\lambda, \mu, \nu \in \mathfrak{P}^{\imath}$, and for each prime power $q$.
\end{proposition}
The polynomials $\bF^\lambda_{\mu,\nu}(v) \in \Z[v,v^{-1}]$ are called the {\em Hall polynomials}.

\begin{proof}
Recall that the pushdown functor $\Pd: \mod(\Lambda)\longrightarrow\mod(\Lambda^\imath)$ induces
a Galois covering $\Pd:\Gproj(\Lambda)\longrightarrow\Gproj(\Lambda^\imath)$, see \eqref{FGproj}.
Note that $\Gproj(\Lambda)\cong\cc_{\Z/2}(\proj(\K Q))$ by \eqref{eq:LaZ2}. So we have a Galois covering $\Pd: \cc_{\Z/2}(\proj(\K Q))\longrightarrow \Gproj(\Lambda^\imath)$. By the Auslander-Reiten quiver (AR-quiver) of $\cc_{\Z/2}(\proj(\K Q))$ described in \cite[Section 2]{CD}, one obtains the AR-quiver of $\Gproj(\Lambda^\imath)$ 
which is independent of the field $\F_q$.

For any $\lambda,\mu\in \fpi$, the same argument as in \cite[Section 2]{Rin0} and \cite[Lemma~ 3.5]{CD} shows that
$\dim_{\F_q}\Hom_{\Lambda^\imath}(G_q(\lambda),G_q(\mu))\text{ and }\dim_{\F_q}\Hom_{\underline{\Gproj}(\Lambda^\imath)}(G_q(\lambda),G_q(\mu))$
only depend on $\lambda$ and $\mu$, but not on $q$.

Since $\underline{\Gproj}^\Z(\Lambda^{\imath})$ is equivalent to $D^b(\K Q)$ and $F_{{\btau}^{\sharp}}^2\simeq \Sigma^2$ in $D^b(\K Q)$, Lemma~\ref{lemma assumption holds} is applicable and implies that for any indecomposable objects $M,N\in \Gproj^\Z(\Lambda^{\imath})$,
$\Ext^1_{\underline{\Gproj}^\Z(\Lambda^{\imath})}(M,N(\ell))$
does not vanish for at most one $\ell$ by noting that the degree shift $(1)$ of $\underline{\Gproj}^\Z(\Lambda^{\imath})$ corresponds to $F_{{\btau}^{\sharp}}$.

To complete the proof, it remains to prove that the cardinality $|\Ext^1_{\Gproj(\Lambda^\imath)}(M,N)_L|$ is a polynomial in $q$ for any $L,M,N\in\Gproj(\Lambda^\imath)$. {A reduction as in \cite[Theorem 3.11]{CD} (see also \cite[Theorem 3.5]{SS16}) allows us to assume that $M$ or $N$ is indecomposable.}

Suppose that $M$ is indecomposable. Write $N=\oplus_{j=1}^t N_j$ with $N_j$ indecomposable. Then there exists at most one $\ell_j$ such that $\Ext^1_{\Gproj^\Z(\Lambda^\imath)}(M,N_j(\ell_j))\neq0$ for each $1\le j\le t$.
Since $M,N$ are gradable, we have
\begin{align*}
\Ext^1_{\Gproj(\Lambda^\imath)}(M,N)
&=\bigoplus_{\ell\in\Z}\Ext^1_{\underline{\Gproj}^\Z(\Lambda^\imath)}(M, N(\ell))
=\Ext^1_{\underline{\Gproj}^\Z(\Lambda^\imath)}(M, \oplus_{j=1}^t N_j(\ell_j)).
\end{align*}
It follows that $\Ext^1_{\Gproj(\Lambda^\imath)}(M,N)_L= \sqcup_{m \in\Z} \Ext^1_{\underline{\Gproj}^\Z(\Lambda^\imath)}(M, \oplus_{j=1}^t N_j(\ell_j))_{L(m)}$. By noting that $\Gproj^\Z(\Lambda^\imath)\cong \cc^b(\proj(\K Q))$ (cf. Remark \ref{rem:bounded}), we conclude by \cite[Corollary 3.7]{CD} that $|\Ext^1_{\underline{\Gproj}^\Z(\Lambda^\imath)}(M, \oplus_{j=1}^t N_j(\ell_j))_{L(m)}|$ is a polynomial in $q$, and then so is $|\Ext^1_{\Gproj(\Lambda^\imath)}(M,N)_L|$.

The case when $N$ is indecomposable is proved analogously. The proof is completed.
\end{proof}

\begin{remark}
In case $\btau=\Id$, the above proposition was proved in \cite[Theorem~ 3.6]{RSZ}.
\end{remark}

\subsection{Generic Hall algebras, I} 

The class of $G_q(\lambda)$ defined in \eqref{def: Glambda} in $K_0(\mod(\Lambda^\imath))$, for $\lambda\in \Phi^\imath$, does not depend on the base field $\K$. It follows that the class of $\res(G_q(\lambda))$ does not depend on the base field $\K$ either. We denote by $\lambda^g$ the class of $\res(G_q(\lambda))$ in $K_0(\mod(\K Q))$.

We consider the twisted generic Ringel-Hall algebra of $\Gproj (\Lambda^\imath)$ over $\Q(\bv)$, denoted by $\tH^{\rm Gp}(\ov{Q}, \ov{I})$,  as follows; cf. Proposition \ref{prop:invariant subalgebra}.  More precisely, $\tH^{\rm Gp}(\ov{Q}, \ov{I})$ is the free $\Q(\bv)$-module with basis $\{\fv_\lambda\mid \lambda\in\fpi\}$ and its multiplication is given by
\begin{align*}
\fv_\mu*\fv_\nu=\bv^{\langle \mu^g,\nu^g\rangle_Q}\sum_{\lambda\in\fp^\imath}\bF_{\mu,\nu}^\lambda(\bv)\fv_\lambda.
 \end{align*}

Let $\tGpg:=\tH^{\rm Gp}(\ov{Q}, \ov{I})[\fv_\lambda^{-1}:\lambda\in\fp^0 ]$ be the localization of $\tH^{\rm Gp}(\ov{Q}, \ov{I})$ with respect to $\fv_\lambda$, for $\lambda\in\fp^0$.
In fact, $\tGpg$ is the {\em generic (twisted) semi-derived Hall algebra } $\cs\cd\ch(\Gproj(\Lambda^\imath))$, see \cite{Gor2} or \S\ref{subsec:SDH}.

Let $\bvs=(\vs_i)_{i\in \I}\in (\Q(\bv)^\times)^\I$ be such that $\vs_i=\vs_{\btau i}$ for any $i$. We define the {\em generic reduced $\imath$Hall algebra} (or {\em generic reduced twisted semi-derived Hall algebra}) $\rGpg$ as follows. Note that $K_0(\mod(\K Q))$ is freely generated by $\alpha_i= \widehat{S_i}$  (for $i\in \I$). By viewing each $\gamma_i$ as the class of the indecomposable projective $\K Q$-module $P_i=(\K Q)e_i$ for $i\in \I$, $K_0(\mod(\K Q))$ is also  freely generated by $\gamma_i$ ($i\in \I$). So there exists an invertible matrix $A=(a_{ij})_{n\times n}\in M_{n}(\Z)$ such that $\alpha_i=\sum_{i=1}^na_{ij}\gamma_j$ for any $i$.
So $\widehat{\E}_{i} = \sum_{j=1}^na_{ij}\widehat{\E}_{\gamma_j}$ in $K_0(\cp^{\leq 1}(\Lambda^\imath))$.
Then $\rGpg$ is defined to be the quotient of $\tGpg$ by the ideal generated by
\begin{align}
   \label{eqn:specializing1}
\prod_{j\in \I}\fv_{{\gamma_j}}^{a_{ij}} +v^2\vs_i  \; (\text{for }\btau i=i),\qquad
\prod_{j\in \I}\fv_{{\gamma_j}}^{a_{ij}}  * \prod_{j\in \I}\fv_{{\gamma_{ j}}}^{a_{\btau i, j}} -\vs_i\vs_{\btau i} \; (\text{for }\btau i\neq i).
\end{align}

Recall $\rMH$ from Definition~\ref{def:reducedHall}.

\begin{proposition}
  \label{prop:specialization}
Let $(Q, \btau)$ be a Dynkin $\imath$quiver. We have a $\Q({\sqq})$-algebra isomorphism
\begin{align*}
\rGpg_{|v={\sqq}}  \cong\rMH.
\end{align*}
\end{proposition}

\begin{proof}
Let $\cs\cd\tH(\Gproj(\Lambda^\imath))$ be the twisted semi-derived Hall algebra with the twisting  as in \eqref{eqn:twsited multiplication}. By construction, the map $\fv_\lambda\mapsto [G_q(\lambda)]$ gives an isomorphism
\begin{align*}
\tGpg_{|v={\sqq}}\cong \cs\cd\tH(\Gproj(\Lambda^\imath)).
\end{align*}
By Theorem \ref{theorem isomorphism of algebras} we have an algebra isomorphism $\cs\cd\tH(\Gproj(\Lambda^\imath))\cong \tMH$, and thus
$\tGpg_{|v={\sqq}} \cong\tMH$. This isomorphism sends the ideal of $\tGpg_{|v={\sqq}}$ generated by \eqref{eqn:specializing1} onto the ideal of $\tMH$ generated by $[\E_i]+q\vs_i$ ( for $\btau i=i$) and $[\E_i]*[\E_{\btau i}]=-\vs_i\vs_{\btau i}$ (for $\btau i\neq i$). The proposition follows.
\end{proof}

The following corollary is a generic version of Theorem \ref{theorem isomorphism of Ui and MRH} by using $\tGpg$.
\begin{corollary}\label{cor: generic version}
Let $(Q, \btau)$ be a Dynkin $\imath$quiver. Then we have $\Q(\bv)$-algebra isomorphisms
\begin{align*}
\tUi\stackrel{\simeq}{\longrightarrow} \tGpg,
&\qquad
\Ui\stackrel{\simeq}{\longrightarrow} \rGpg.
\end{align*}
\end{corollary}

\begin{proof}
Follows by Proposition \ref{prop:specialization}, Theorem \ref{theorem isomorphism of Ui and MRH} and Proposition~\ref{prop:Hall polynomial}.
\end{proof}

It is possible but somewhat messy to write down the formulas on generators for the isomorphisms in Corollary~\ref{cor: generic version}; compare Bridgeland's original version \cite{Br} of Proposition~ \ref{prop:bridgeland}.

\subsection{Generic Hall algebras, II}  

We shall formulate the generic semi-derived Ringel-Hall algebra as the generic $\imath$Hall algebra.

Let $\Phi^\imath:=\Phi^+\cup\Phi^0$ be as above. For $\alpha\in\Phi^+$, $M_q(\alpha)$ is the indecomposable $\K Q$-module, viewed as $\Lambda^\imath$-module.
For any $\gamma_i\in\Phi^0$, define $M_q(\gamma_i)=\E_i$. For $\lambda\in\fpi$, we define
\[
M_q(\lambda):=[ \oplus_{\alpha\in\Phi^+} {\lambda(\alpha)}M_q(\alpha)]*[\oplus_{i\in\I} \lambda(\gamma_i)\E_i]=[ \oplus_{\alpha\in\Phi^+} {\lambda(\alpha)}M_q(\alpha)]*\E_{\sum_{i\in \I}\lambda(\gamma_i)\alpha_i}
\]
in the $\imath$Hall algebra $\tMH$,  which is compatible with $M_q(\lambda)$ for $\lambda\in\fp$ in \eqref{def:Mlambda}.
Let
\[
\tilde{\fp}^\imath=\{\lambda: \Phi^\imath \rightarrow \Z \mid \lambda(\alpha)\in \N\,\, \forall \alpha\in \Phi^+,  \lambda(\gamma_i)\in\Z\,\, \forall \gamma_i\in\Phi^0\}.
\]
Then for $\lambda\in \tilde{\fp}^\imath$, define
\[
M_q(\lambda):=[ \oplus_{\alpha\in\Phi^+} {\lambda(\alpha)}M_q(\alpha)]*\E_{\sum_{i\in \I}\lambda(\gamma_i)\alpha_i}.
\]

Denote by $\tilde{\fp}^0$ the subset of $\tilde{\fp}^\imath$ which consists of functions supported at $\Phi^0$, and identify  $\fp$ with  the subset of $\tilde{\fp}^\imath$ which consists of functions supported at $\Phi^+$. 
By Corollary \ref{cor:basis of MRH as torus-module}, $\{M_q(\lambda) \mid \lambda\in\tilde{\fp}^\imath \}$ forms a basis of $\tMH$. For  $\mu,\nu\in \tilde{\fp}^\imath$, we write
\[
[M_q(\mu)]*[M_q(\nu)]=\sum_{\lambda\in\tilde{\fp}^\imath}\varphi_{ M_q(\mu),M_q(\nu)}^{M_q( \lambda)}[M_q(\lambda)].
\]

\begin{lemma}\label{lem: Hall polynomial for MRH}
For every $\lambda,\mu,\nu \in \tilde{\fp}^\imath$, there exists a polynomial $\boldsymbol{\varphi}^\lambda_{\mu,\nu}(\bv)\in\Z[\bv,\bv^{-1}]$ such that
$\boldsymbol{\varphi}^\lambda_{\mu,\nu}({\sqq})  =\varphi_{ M_q(\mu),M_q(\nu)}^{M_q( \lambda)},$
 for each prime power $q$ (recall $\sqq =\sqrt{q}$).
\end{lemma}

\begin{proof}
For any $\nu: \Phi^0\rightarrow\Z$ (viewed as a function $\nu:\Phi^\imath\rightarrow\Z$ supported at $\Phi^0$), we define $[G_q(\nu)]:=\prod_{\gamma_i\in\Phi^0} [G_q(\gamma_i)]^{\nu(\gamma_i)}$ in $\cs\cd\tH(\Gproj(\Lambda^\imath))$  and also in $\tMH$. 
In this way, for any $\nu\in\tilde{\fp}^0$, there exists a unique $\omega'_\nu\in\tilde{\fp}^0$ such that $[M_q(\nu)] = [G_q(\omega'_\nu)]$ in $\tMH$.
Note that $\omega_\nu'$ only depends on $\nu$, not on $q$.

For any $\lambda\in\fp$, there exists a unique $\omega_\lambda:\Phi^0\rightarrow\Z$ such that $[M_q(\lambda)]={\sqq}^{b_\lambda}[G_q(\lambda)]*[G_q(\omega_\lambda)]$ in $\tMH$ for some $b_\lambda\in\Z$.
Note that $\omega_\lambda$ only depends on $\lambda$, not on the field $\K$. Furthermore, $b_\lambda$ comes from the Euler form of $\K Q$, and so it does not depend on $q$.

For any $\mu,\nu\in\tilde{\fp}^\imath$, there exist unique $\mu_0:\Phi^0\rightarrow\Z$, $\mu_1:\Phi^+\rightarrow\N$, $\nu_0:\Phi^0\rightarrow\Z$, $\nu_1:\Phi^+\rightarrow\N$ such that
$\mu=\mu_1+\mu_0$, $\nu=\nu_1+\nu_0$. Then $[M_q(\mu)] =[M_q(\mu_1)]*[M_q(\mu_0)]$ and $[M_q(\nu)] =[M_q(\nu_1)]*[M_q(\nu_0)]$ by definition.
It follows by Lemma \ref{lem:bimodule of MRH} and Lemma~ \ref{lemma multiplcation of locally projective modules} that
\begin{align*}
[M_q(\mu)] &*[M_q(\nu)]\\
=& [M_q(\mu_1)]*[M_q(\mu_0)]*[M_q(\nu_1)]*[M_q(\nu_0)]\\
=&{\sqq}^{d_{\mu,\nu}} [G_q(\mu_1)]*[G_q(\nu_1)]*[G_q(\omega_{\mu_1}+\omega_{\nu_1}+\omega_{\mu_0}'+\omega_{\nu_0}')]\\
=&{\sqq}^{d_{\mu,\nu}}{\sqq}^{d'_{\mu,\nu}}\sum_{\lambda\in\tilde{\fp}^\imath} F_{G_q(\mu_1),G_q(\nu_1)}^{G_q(\lambda)}[G_q(\lambda)]*[G_q(\omega_{\mu_1}+\omega_{\nu_1}+\omega_{\mu_0}'+\omega_{\nu_0}')]\\
=&\sum_{\lambda\in\tilde{\fp}^\imath} {\sqq}^{d_{\mu,\nu}+d'_{\mu,\nu}-b_\lambda}F_{G_q(\mu_1),G_q(\nu_1)}^{G_q(\lambda)}[M_q(\lambda)]*[G_q(\omega_{\mu_1}+\omega_{\nu_1}+\omega_{\mu_0}'+\omega_{\nu_0}'-\omega_\lambda)],
\end{align*}
where $d_{\mu,\nu},d'_{\mu,\nu}\in\Z$ do not depend on $q$ (as they come from the Euler form). It follows from Proposition~ \ref{prop:Hall polynomial} that there exists a polynomial $\bF^\lambda_{\mu,\nu}({v})\in\Z[v, v^{-1}]$ such that $\bF^\lambda_{\mu,\nu}({{\sqq}}) = F^{G_q(\lambda)}_{G_q(\mu),G_q(\nu)}.$
Clearly, $[G_q(\omega_{\mu_1}+\omega_{\nu_1}+\omega_{\mu_0}'+\omega_{\nu_0}'-\omega_\lambda)]=\E_\alpha$ for some $\alpha\in K_0(\mod(\K Q))$, which does not depend on $q$. The lemma follows.
\end{proof}

Let $\Phi^+=\{\beta_1,\dots,\beta_N\}$, and $\beta_j=\sum_{i\in\I} b_{ji}\alpha_i$ for any $1\le j\le N$.
Motivated by the dimension vectors of modules, we define $\dimv\, \lambda=(\lambda(\gamma_i)+\lambda(\gamma_{\btau i})+ \sum_{\beta_j\in\Phi^+}\lambda(\beta_j)b_{ji})_{i\in \I}$ for any $\lambda \in\fp^\imath$. In particular, by definition, $\dimv\, \lambda=\dimv\, M_q(\lambda)$ for any $\lambda\in \fp$ or $\lambda\in\fp^0$.

\begin{corollary}
 For any triple $\lambda,\mu,\nu$, $\varphi^\lambda_{\mu,\nu}(v) = 0$ unless $\dimv\, \lambda=\dimv\, \mu+\dimv\, \nu$.
\end{corollary}

\begin{proof}
From the proof of Lemma \ref{lem: Hall polynomial for MRH}, the assertion follows by noting that $\tMH$ is a $K_0(\mod(\Lambda^\imath))=K_0(\mod(kQ))$-graded algebra.
\end{proof}

We now define the {\em generic $\imath$Hall algebra} as the generic (twisted) semi-derived Ringel-Hall algebra of $\Lambda^\imath$  over $\Q(\bv)$, denoted by $\tMHg$ as follows. The algebra $\tMHg$ is the free $\Q(\bv)$-module with a basis $\{\fu_\lambda\mid \lambda\in\tilde{\fp}^\imath\}$ and its multiplication is given by
\begin{align}
\fu_\mu*\fu_\nu=\sum_{\lambda\in\tilde{\fp}^\imath}\boldsymbol{\varphi}_{\mu,\nu}^\lambda(\bv)\fu_\lambda.
\end{align}
Its reduced generic version, denoted by $\rMHg$, is defined to be the quotient of $\tMHg$ by the ideal generated by
\begin{align}
   \label{eqn:specializing3}
\fu_{\gamma_i} +v^2\vs_i \; (\text{for }\btau i=i),\qquad
\fu_{\gamma_i}*\fu_{\btau \gamma_i} -\vs_i^2 \; (\text{for }\btau i\neq i).
\end{align}

\begin{theorem}
   \label{generic U MRH}
Let $(Q, \btau)$ be a Dynkin $\imath$quiver. Then we have a $\Q(\bv)$-algebra isomorphism
\begin{align}
\widetilde{\psi}: \tUi\stackrel{\simeq}{\longrightarrow}& \tMHg,
 \label{eq:psi2} \\
B_i\mapsto \frac{-1}{v^2-1}\fu_{\alpha_i},\text{ if } i\in\ci, &\qquad\qquad
B_{i}\mapsto\frac{\bv}{v^2-1}\fu_{\alpha_i},\text{ if }i\notin \ci,
 \label{eq:psiB} \\
\tk_j\mapsto \fu_{\gamma_j}, \text{ if }\btau j\neq j,
&\qquad\qquad
\tk_j\mapsto -v^{-2}\fu_{\gamma_j}, \text{ if }\btau j=j.
 \notag
\end{align}
Moreover, this induces a $\Q(\bv)$-algebra isomorphism
\begin{align}
\psi: \Ui\stackrel{\simeq}{\longrightarrow}& \rMHg,
 \label{eq:psi}
 \end{align}
 which sends $B_i$ as in \eqref{eq:psiB} and sends $k_j\mapsto \vs_j^{-1} \fu_{\gamma_j}, \text{ for }j \in \I \setminus \ci.$
\end{theorem}

\begin{proof}
Follows by Proposition~\ref{prop:qHall=Ui}, Theorem \ref{theorem isomorphism of Ui and MRH} and Lemma~\ref{lem: Hall polynomial for MRH}.
\end{proof}

\begin{remark}
The PBW bases for $\tMH$ in Theorem~\ref{thm:HallPBW} lifts to PBW bases on the generic $\imath$Hall algebra $\tMHg$ (and respectively, its reduced version $\rMHg$). This in turn provides PBW bases for $\tUi$ (and respectively, $\Ui$) via the isomorphism $\widetilde \psi$ (and respectively, $\psi$) in Theorem~\ref{generic U MRH}. The PBW bases will be made more explicit using the reflection functors/braid group actions in a sequel \cite{LW19b}.
\end{remark}

\appendix

\section{Semi-derived Ringel-Hall algebras of $1$-Gorenstein algebras\\  by Ming Lu}
 \label{app:A}

In this Appendix, following the basic ideas of \cite{LP}, we shall formulate semi-derived Ringel-Hall algebras for {\em weakly $1$-Gorenstein} exact categories, including the module categories of $1$-Gorenstein algebras.

\subsection{Hall algebras}
   \label{subsec:HA}
Let $\ce$ be an essentially small exact category in the sense of Quillen, linear over a finite field $\K=\F_q$.
Assume that $\ce$ has finite morphism and extension spaces, i.e.,
\[
|\Hom(M,N)|<\infty,\quad |\Ext^1(M,N)|<\infty,\,\,\forall M,N\in\ce.
\]

Given objects $M,N,L\in\ce$, define $\Ext^1(M,N)_L\subseteq \Ext^1(M,N)$ as the subset parameterizing extensions whose middle term is isomorphic to $L$. We define the {\em Ringel-Hall algebra} $\ch(\ce)$ (or {\em Hall algebra} for short) to be the $\Q$-vector space whose basis is formed by the isoclasses $[M]$ of objects $M$ of $\ce$, with the multiplication defined by (see \cite{Br})
\[
[M]\diamond [N]=\sum_{[L]\in \Iso(\ce)}F_{M,N}^L[L],
\]
where
\begin{align*}
F_{M,N}^L:=\frac{|\Ext^1(M,N)_L|}{|\Hom(M,N)|}
\end{align*}
is the Hall number.

\begin{remark}
Ringel's version of Hall algebra \cite{Rin} uses a different Hall product, but these two versions of Hall algebra are isomorphic by rescaling the generators by the orders of automorphisms.
\end{remark}

It is well known that the algebra $\ch(\ce)$ is associative and unital. The unit is given by $[0]$, where $0$ is the zero object of $\ce$, see \cite{Rin0, Rin} and also \cite{Br}.  We shall use Bridgeland's version of Hall product as above throughout this paper.

\subsection{Definition of semi-derived Ringel-Hall algebras}
   \label{subsection:Def of MRH}
Let $\ca$ be an essentially small exact category
linear over $k=\F_q$. We introduce the following subcategories of $\ca$:
\begin{align*}
\cp^{\leq i}(\ca) &=\{X\in\ca\mid\Ext\text{-}\pd X\leq i\},
 \\
{\mathcal I}^{\leq i}(\ca) &=\{X\in\ca\mid\Ext\text{-}\ind X\leq i\}, \quad \forall i\in\N,
\\
\cp^{<\infty}(\ca) &= \{X\in\ca\mid\Ext\text{-}\pd X<\infty\},
  \\
{\mathcal I}^{<\infty}(\ca) &= \{X\in\ca\mid\Ext\text{-}\ind X<\infty\}.
\end{align*}
The category $\ca$ is called \emph{weakly Gorenstein} if $\cp^{<\infty}(\ca)={\mathcal I}^{<\infty}(\ca)$, and $\ca$ is a \emph{weakly $d$-Gorenstein} exact category if $\ca$ is weakly Gorenstein and $\cp^{<\infty}(\ca)=\cp^{\leq d}(\ca)$, ${\mathcal I}^{<\infty}(\ca)={\mathcal I}^{\leq d}(\ca)$.

\begin{lemma}[Iwanaga's Theorem]
Let $\ca$ be a weakly Gorenstein exact category with enough projectives and injectives. Then $\cp^{<\infty}(\ca)=\cp^{\leq d}(\ca)$ if and only if ${\mathcal I}^{<\infty}(\ca)={\mathcal I}^{\leq d}(\ca)$.\end{lemma}

Throughout this section, we always assume that $\ca$ is an exact category satisfying the following conditions:
\begin{itemize}
\item[(E-a)] $\ca$ is essentially small, with finite morphism spaces, and finite extension spaces,
\item[(E-b)] $\ca$ is linear over $\K=\F_q$,
\item[(E-c)] $\ca$ is weakly $1$-Gorenstein.
\item[(E-d)] For any object $X\in\ca$, there exists an object $P_X\in\cp^{<\infty}(\ca)$ and a deflation $P_X\twoheadrightarrow X$.
\end{itemize}
In this case, $\cp^{<\infty}(\ca)=\cp^{\leq1}(\ca)={\mathcal I}^{<\infty}(\ca)={\mathcal I}^{\leq 1}(\ca)$.

Note that for any finite-dimensional $1$-Gorenstein algebra $\Lambda$ over $\K$ (cf. \S\ref{subsec:Gproj}), $\ca=\mod(\Lambda)$ satisfies (E-a)--(E-d).

\begin{example}\label{example 1}
Let $\ce$ be a hereditary abelian $\K$-category (not necessarily with enough projective objects). Let $\cc_{\Z/n}(\ce)$ be the category of $\Z/n$-graded complexes, for $n\geq2$. Denote by $\cc_{ac,\Z/n}(\ce)$ the subcategory of acyclic complexes in
$\cc_{\Z/n}(\ce)$. It follows from \cite[Proposition 2.3]{LP} that $\Ext^p_{\cc_{\Z/n}(\ce)}(K,M)=0=\Ext^p_{\cc_{\Z/n}(\ce)}(M,K)$ for any $K\in\cc_{ac,\Z/n}(\ce)$, $M\in\cc_{\Z/n}(\ce)$ and $p\geq2$. On the other hand, any $\Z/n$-graded complex with finite Ext-projective dimension or Ext-injective dimension must be acyclic.
So $\cc_{\Z/n}(\ce)$ is weakly $1$-Gorenstein which satisfies (E-a)--(E-d) with $\cp^{<\infty}(\cc_{\Z/n}(\ce))=\cc_{ac,\Z/n}(\ce)=\mathcal{I}^{<\infty}(\cc_{\Z/n}(\ce))$.
\end{example}

Let $\ch(\ca)$ be the Ringel-Hall algebra of $\ca$, i.e., $\ch(\ca)=\bigoplus_{[M]}\Q[M]$ with the multiplication given by
$$[M]\diamond [N]=\sum_{M\in \Iso(\ca)}\frac{|\Ext^1(M,N)_L|}{|\Hom(M,N)|}[L].$$
It is well known that $\ch(\ca)$ is a $K_0(\ca)$-graded algebra, where $K_0(\ca)$ is the Grothendieck group of $\ca$. For any $M\in\ca$, denote by $\widehat{M}$ for the corresponding element in $K_0(\ca)$.

For objects $K,M\in \ca$, if $K\in\fpr(\ca)$, we define
the Euler forms
\begin{equation}\label{left Euler form}
\langle K,M\rangle=\sum_{i=0}^{+\infty}(-1)^i \dim_k\Ext^i(K,M)=\dim_k\Hom(K,M)-\dim_k\Ext^1(K,M),
\end{equation}
and
\begin{equation}\label{right Euler form}
\langle M,K\rangle=\sum_{i=0}^{+\infty}(-1)^i \dim_k\Ext^i(M,K)=\dim_k\Hom(M,K)-\dim_k\Ext^1(M,K).
\end{equation}
These forms descend to bilinear Euler forms on the Grothendieck groups $K_0(\fpr(\ca))$ and $K_0(\ca)$, denoted by the same symbol:
\begin{equation*}
\langle\cdot,\cdot\rangle: K_0(\fpr(\ca))\times K_0(\ca)\longrightarrow \Z,
\end{equation*}
and
\begin{equation*}
\langle\cdot,\cdot\rangle: K_0(\ca)\times K_0(\fpr(\ca))\longrightarrow \Z,
\end{equation*}
such that
\begin{equation}
\langle \widehat{K},\widehat{M}\rangle:=\langle K,M\rangle,\qquad\langle \widehat{M},\widehat{K}\rangle:=\langle M,K\rangle,\qquad\forall K\in\fpr(\ca),M\in\ca.
\end{equation}
We have used the same symbol by noting that these two forms coincide when restricting to $K_0(\fpr(\ca))\times K_0(\fpr(\ca))$.

Inspired by the construction in \cite{LP}, we consider the following quotient algebra.
Let $I$ be the two-sided ideal of $\ch(\ca)$ generated by all differences $[L]-[K\oplus M]$ if there is a short exact sequence
\begin{equation}
  \label{eq:ideal}
 0 \longrightarrow K \longrightarrow L \longrightarrow M \longrightarrow 0
\end{equation}
in $\ca$ with $K\in \fpr(\ca)$. Since $I$ is generated by $K_0(\ca)$-homogeneous elements, the quotient algebra $\ch(\ca)/I$ is a $K_0(\ca)$-graded algebra.

The lemma below follows by definition.
\begin{lemma}\label{lem: formular in quotient of Hall}
For any $K\in\fpr(\ca)$ and $M\in\ca$, we have
\[
[M]\diamond [K]=q^{-\langle M,K\rangle} [M\oplus K]
\]
in $\ch(\ca)/I$. In particular, for any $K_1,K_2\in \fpr(\ca)$, we have
\begin{equation}
   \label{multiplication formula for quantum torus}
[K_1]\diamond[K_2]=q^{-\langle K_1,K_2\rangle} [K_1\oplus K_2]
\end{equation}
in $\ch(\ca)/I$.
\end{lemma}

Let $A$ be a ring with identity $1$, and $S$ a subset of $A$ closed under multiplication, and $1\in S$. Recall that a \emph{right localization of $A$ with respect to $S$} is a ring $R$ and a ring map $i:A\rightarrow R$ such that
\begin{itemize}
\item[(i)] $i(s)$ is a unit in $R$ for each $s\in S$,
\item[(ii)] every element of $R$ has the form $i(a)i(s)^{-1}$ for some $a\in A$, $s\in S$,
\item[(iii)] $i(a)i(s)^{-1}=i(b)i(s)^{-1}$ if and only if $at=bt$ for some $t\in S$.
\end{itemize}
Such a $R$ is a \emph{universal $S$-inverting} ring, and so is unique. We shall denote $R$ by $AS^{-1}$ when it exists. We will suppress the map $i$ and write the elements of $AS^{-1}$ as $as^{-1}$.
We say $S$ satisfies the \emph{right Ore condition}, if for all $a\in A$ and $s\in S$, there exists $a_1\in A$ and $s_1\in S$ such that $sa_1=as_1$. We say $S$ is \emph{right reversible} if for any $a\in A$, $s\in S$ and $sa=0$ in $A$, then there exists $t\in S$ such that $at=0$ in $A$.
Ore's localization theorem states that the right localization $AS^{-1}$ exists if and only if $S$ is a right Ore, right reversible subset of $A$.

Returning to the category $\ca$, we consider the following subset of $\ch(\ca)/I$:
\begin{equation}
  \label{eq:Sca}
\cs_{\ca} := \{ a[K] \in \ch(\ca)/I \mid a\in \Q^\times, K\in \cp^{\leq1}(\ca)\}.
\end{equation}
We see that $\cs_{\ca}$ is a multiplicatively closed subset with the identity $[0]\in \cs_{\ca}$. 

\begin{proposition}\label{prop:Def of MRH}
Let $\ca$ be an exact category satisfying (E-a)--(E-d). Then the multiplicatively closed subset $\cs_{\ca}$ is a right Ore, right reversible subset of $\ch(\ca)/I$. Equivalently, the right localization of
$\ch(\ca)/I$ with respect to $\cs_{\ca}$ exists, and will be denoted by $(\ch(\ca)/I)[\cs_{\ca}^{-1}]$.
\end{proposition}
\begin{proof}

For any $M\in\ca$, $K\in\fpr(\ca)$,
$$[K]\diamond[M]=\sum_{[L]\in\Iso(\ca)}\frac{|\Ext^1(K,M)_L|}{|\Hom(K,M)|}[L].$$
We only consider $[L]\in\Iso(\ca)$ with $|\Ext^1(K,M)_L|\neq0$, which form a subset $\cv$. Then for any $[L]\in\cv$, choose (and fix) a short exact sequence
\begin{equation}\label{short exact sequence 1}
0\longrightarrow M\stackrel{f}{\longrightarrow} L\stackrel{g}{\longrightarrow} K\longrightarrow0.
\end{equation}
Let $P_L\xrightarrow{s} L$ be a deflation with $P_L\in\fpr(\ca)$. By doing pullback, we have a short exact sequence
\begin{align}\label{eqn:short exact sequence 3}
0\longrightarrow A_L\xrightarrow{\tiny\left( \begin{array}{c} t_1\\ t_2 \end{array}\right)} P_L\oplus M\xrightarrow{(s,f)} L\longrightarrow0.
\end{align}
Then the following diagram is commutative:
\[\xymatrix{& M\ar@{=}[r] \ar[d]&M\ar[d]^f \\
A_L\ar[r]^{\tiny\left( \begin{array}{c} t_1\\ t_2 \end{array}\right)\quad}\ar@{=}[d]& P_L\oplus M\ar[r]^{\quad(s,f)} \ar[d]& L \ar[d]^{g} \\
A_L \ar[r]^{t_1} &  P_L   \ar[r]^{gs}& K }
\]
where the middle column is the split short exact sequence.
Then $0\rightarrow A_L\xrightarrow{t_1}P_L\xrightarrow{gs} K\rightarrow0$
is a short exact sequence. Since $K\in\fpr(\ca)$, we have $A_L\in\fpr(\ca)$. So $[A_L\oplus K]=[P_L]$ in $\ch(\ca)/I$.
It follows from \eqref{eqn:short exact sequence 3} that $[A_L\oplus L]=[P_L\oplus M]$ in $\ch(\ca)/I$.

Let $[C]=[\bigoplus_{[L]\in\cv}A_L]$. Let $[D_L]=[\bigoplus_{[L']\in\cv,L'\ncong L}A_{L'}]$. Then $[C]= q^{\langle A_L,D_L  \rangle} [A_L]\diamond [D_L]$ for any $[L]\in\cv$. In $\ch(\ca)/I$, we have
\begin{align}\label{eqn: right Ore}
[K]\diamond[M]\diamond[C]=&\sum_{[L]\in\Iso(\ca)} \frac{|\Ext^1(K,M)_L| } {|\Hom(K,M)|} [L]\diamond[C] \\\notag
=&\sum_{[L]\in\cv} \frac{|\Ext^1(K,M)_L| } {|\Hom(K,M)|}  q^{\langle A_L,D_L\rangle}[L]\diamond [A_L]\diamond [D_L]\\\notag
=&\sum_{[L]\in\cv} \frac{|\Ext^1(K,M)_L| }{|\Hom(K,M)|}  q^{\langle A_L,D_L\rangle-\langle L,A_L\rangle} [M\oplus P_L ]\diamond [D_L]\\
\notag
=&\sum_{[L]\in\cv} \frac{|\Ext^1(K,M)_L| }{|\Hom(K,M)|}  q^{\langle A_L,D_L\rangle-\langle L,A_L\rangle+\langle M, P_L\rangle}[M]\diamond [P_L ]\diamond [D_L]\\\notag
=&\sum_{[L]\in\cv} \frac{|\Ext^1(K,M)_L|}{|\Hom(K,M)|}q^{\langle A_L,D_L\rangle+\langle M, K\rangle} [M]\diamond [K]\diamond[A_L]\diamond [D_L]
\\
\notag=&\sum_{[L]\in\cv} \frac{|\Ext^1(K,M)_L|}{|\Hom(K,M)| } q^{\langle M, K\rangle} [M]\diamond [K]\diamond[C]
\\\notag
=& q^{\langle M, K\rangle-\langle K,M\rangle}[M]\diamond [K]\diamond [C]\\\notag
=& [M]\diamond ( q^{ \langle M, K\rangle-\langle K,M\rangle-\langle K,C\rangle} [K\oplus C]).
\end{align}
So $\cs_{\ca}$ satisfies the right Ore condition.

If $[K]\diamond( \sum_{i=1}^n a_i [M_i])=0$, then since $\ch(\ca)/I$ is a $K_0(\ca)$-graded algebra, we can assume that all $M_i,1\leq i\leq n$ are in the same class of $K_0(\ca)$.
From \eqref{eqn: right Ore}, we have
$$0=[K]\diamond( \sum_{i=1}^n a_i [M_i])\diamond [C]= \sum_{i=1}^n a_iq^{\langle M_i,K\rangle -\langle K,M_i\rangle }[M_i] \diamond [K]\diamond [C], $$
for some $C\in\fpr(\ca)$. By our assumption,
$q^{\langle M_i,K\rangle -\langle K,M_i\rangle }$ are equal for all $1\leq i\leq n$, which is denoted by $b$. Note that $b\neq0$.
So
\[
 (\sum_{i=1}^n a_i [M_i]) \diamond( b [K\oplus C])=\sum_{i=1}^n a_ibq^{\langle K,C\rangle} [M_i] \diamond [K]\diamond [C] =0.
 \]
Thus $\cs_{\ca}$ is a right reversible subset.

The assertion of the proposition follows now from Ore's localization theorem.
\end{proof}

The semi-derived Ringel-Hall algebras are defined in \cite{LP} for the category of $\Z/n$-graded complexes over any hereditary abelian category, for $n\geq2$. The following definition generalizes \cite{LP}; see the remarks in Example~\ref{example 1}.

\begin{definition}
\label{def:semi-derived}
For any exact category $\ca$ satisfying (E-a)--(E-d), $(\ch(\ca)/I)[\cs_{\ca}^{-1}]$ is called the semi-derived Ringel-Hall algebra of $\ca$, and denoted by $\cs\cd\ch(\ca)$.
\end{definition}

Let $\ch(\fpr(\ca))$ be the Ringel-Hall algebra of the exact category $\fpr(\ca)$. Then $\ch(\fpr(\ca))$ is a subalgebra of
$\ch(\ca)$.
\begin{definition}
   \label{def:torus}
The quantum torus $\T(\ca)$ is defined to be the subalgebra of $\cs\cd\ch(\ca)$ generated by $[M]$ in $\Iso(\fpr(\ca))$.
\end{definition}
Then $\cs\cd\ch(\ca)$ is naturally a $\T(\ca)$-bimodule.

Since $\fpr(\ca)$ is an exact category satisfying (E-a)--(E-d), the semi-derived Ringel-Hall algebra $\cs\cd\ch(\fpr(\ca))$ is defined. We can and shall always identify $\T(\ca)\cong \cs\cd\ch(\fpr(\ca))$. Then the natural embedding $\ch(\fpr(\ca))\rightarrow \ch(\ca)$ implies that $\cs\cd\ch(\fpr(\ca))$ is a subalgebra of $\cs\cd\ch(\ca)$, which coincides with $\T(\ca)\subseteq \cs\cd\ch(\ca)$.
Let $\T(K_0(\fpr(\ca)),q^{-\langle \cdot,\cdot \rangle})$ be the group algebra of $K_0(\fpr(\ca))$ over $\Q$, with the multiplication twisted by $q^{-\langle \cdot,\cdot \rangle}$ as in (\ref{multiplication formula for quantum torus}).
Then $\T(\ca)$ is isomorphic to $\T(K_0(\fpr(\ca)),q^{-\langle \cdot, \cdot \rangle})$ as algebras, see e.g., \cite[Lemma 4.5]{Gor1}.


\begin{lemma}\label{lem:bimodule of MRH}
For any $K\in\fpr(\ca)$ and $M\in\ca$, we have
\begin{eqnarray}
[M]\diamond [K]&=&q^{-\langle M,K\rangle} [M\oplus K] \label{right module structure}\\ \,
[K]\diamond[M]&=&q^{-\langle K,M\rangle} [K\oplus M]\label{left module structure}
\end{eqnarray}
in $\cs\cd\ch(\ca)$. In particular, \eqref{right module structure}--\eqref{left module structure} give the $\T(\ca)$-bimodule structure of $\cs\cd\ch(\ca)$.
\end{lemma}

\begin{proof}
\eqref{right module structure} follows from Lemma \ref{lem: formular in quotient of Hall}.

For (\ref{left module structure}), using \eqref{eqn: right Ore} in the proof of Proposition \ref{prop:Def of MRH}, we have that
$$[K]\diamond[M]\diamond[C]= q^{ \langle M,K\rangle-\langle K,M\rangle  } [M]\diamond [K]\diamond [C]$$ for some $C\in\fpr(\ca)$,
which implies that
$$[K]\diamond[M]=q^{ \langle M,K\rangle-\langle K,M\rangle  }  [M]\diamond [K]=q^{-\langle K,M\rangle} [K\oplus M]$$
in $\cs\cd\ch(\ca)$.
\end{proof}

It follows from the proof of Lemma \ref{lem:bimodule of MRH} that $[L]=[K\oplus M]$ in $\cs\cd\ch(\ca)$ if there exists a short exact sequence
$0\rightarrow M\rightarrow L\rightarrow K\rightarrow0$ for $K\in\fpr(\ca)$ and $M\in\ca$.

\subsection{$1$-Gorenstein algebras} \label{subsec:MRH for 1-Gor}

Let $ A$ be a finite-dimensional $1$-Gorenstein algebra over $\K$ and $\mod( A)$ the abelian category of finitely generated $ A$-modules.

Let $\ch( A)$ be the Ringel-Hall algebra of $\mod( A)$.
It is well known that $\ch( A)$ is a $K_0(\mod( A))$-graded algebra, where $K_0(\mod( A))$ is the Grothendieck group of $\mod( A)$.

The following lemma is well known.
\begin{lemma}
For any $M\in\mod( A)$, the following are equivalent.
\begin{itemize}
\item[(i)] $\pd(M)\leq1$;
\item[(ii)] $\ind(M)\leq1$;
\item[(iii)] $\pd(M)<\infty$;
\item[(iv)] $\ind(M)<\infty$.
\end{itemize}
\end{lemma}

Recall the subcategory $\Gproj( A)$ of $\mod(A)$ from \S\ref{subsec:Gproj}. 
We also have
\begin{align*}
\Gproj( A)&=\{M\in\mod(A) \mid \Ext^1_ A(M, A)=0\},\\
\proj( A)&=\fpr( A)\cap \Gproj( A).
\end{align*}
It follows by the above discussion that the category $\mod( A)$ satisfies (E-a)--(E-d) in \S\ref{subsection:Def of MRH}, and hence the semi-derived Ringel-Hall algebra of $ A$ is well defined.

The category $\fpr( A)$ is a hereditary exact subcategory with enough projective and injective objects, i.e. $\Ext^p_{\fpr( A)}(-,-)$ vanishes for $p\geq2$, and $\Gproj( A)$ is a Frobenius category with projective modules as its projective-injective objects. By Buchweitz-Happel's theorem (see Theorem \ref{thm:Buchweitz-Happel}), $\underline{\Gproj}( A)$ is triangle equivalent to the singularity category $D_{sg}(\mod( A))$.

Let $\ch(\Gproj( A))$ be the Ringel-Hall algebra of $\Gproj( A)$, which is a subalgebra of $\ch( A)$. Recall that $\T( A)=\cs\cd\ch(\fpr( A))$ which is a subalgebra of $\cs\cd\ch( A)$.

The following construction is inspired by \cite{Gor1,LP}.
Define $I'$ to be the following linear subspace of $\ch( A)$:
\begin{align}  \label{eq:I'}
I'=\text{Span} & \{ [L]-[K\oplus M]  \mid \exists \text{ a short exact sequence }
\\
& \qquad\qquad
0 \rightarrow K \rightarrow L \rightarrow M \rightarrow 0
\text{ for }
K\in\fpr( A),  L, M\in\mod( A) \}.
 \notag
\end{align}
The quotient space $\ch( A)/I'$ is a bimodule over $\ch(\fpr( A))$ by letting
\begin{equation}\label{eq:bimodule of SDH}
[K]\diamond[M]:=q^{-\langle K,M\rangle} [K\oplus M], \quad [M]\diamond[K]:=q^{-\langle M,K\rangle} [M\oplus K],
\end{equation}
for any $K\in \fpr (A)$ and $M\in\mod( A)$.
Since $\T( A)$ is a localization of $\ch(\fpr( A))$, one can define
$(\ch( A)/I')[\cs_A^{-1}]:= \T( A)\otimes_{ \ch(\fpr( A)) } \ch( A)/I'\otimes_{\ch(\fpr( A))} \T( A)$,
which is a bimodule over the quantum torus $\T( A)$.
\begin{lemma}\label{lem:right in SDH}
For any short exact sequence $0\rightarrow M\xrightarrow{f_1} L\xrightarrow{f_2} K\rightarrow 0$ in $\mod( A)$ with $K\in\fpr( A)$, we have
$[L]=[M\oplus K]$ in $(\ch( A)/I')[\cs_A^{-1}]$.
\end{lemma}

\begin{proof}
By the proof in Proposition \ref{prop:Def of MRH}, we have $q^{\langle K,M\rangle} [K]\diamond [M]=[L]$ in $(\ch( A)/I')[\cs_A^{-1}]$, where $\diamond$ denotes the multiplication here. It follows from \eqref{eq:bimodule of SDH} that $[K\oplus M]=[L]$.
\end{proof}

Note that $\cs\cd\ch( A)= \T( A)\otimes_{ \ch(\fpr( A)) } ( \ch( A)/I) \otimes_{\ch(\fpr( A))} \T( A)$ as a $\T( A)$-bimodule. Then there exists a natural epimorphism $(\ch( A)/I')[\cs_A^{-1}]\rightarrow \cs\cd\ch( A)$ as $\T( A)$-bimodules  induced by the natural epimorphism $\ch( A)/I'\rightarrow \ch( A)/I$; see Lemma \ref{lem:bimodule of MRH}. We shall prove that this epimorphism is an isomorphism, generalizing \cite[Proposition 3.19]{LP}. We first prepare two lemmas, which are inspired by \cite[Lemma 3.16, Lemma 3.17]{LP}.

\begin{lemma}\label{lemma left ideal}
Let $0\rightarrow K\xrightarrow{h_1} M\xrightarrow{h_2} N\rightarrow 0$ be a short exact sequence in $\mod( A)$ with $K\in\fpr( A)$. Then for any $L$ we have
\[
 \p\big([L]\diamond([M]-[K\oplus N] ) \big)=0,
\]
where $\p:\ch( A)\rightarrow  (\ch( A)/I')[\cs_A^{-1}]$ is the natural projection.
\end{lemma}

\begin{proof}
First, we have by definition
\[
\p([L]\diamond[M])=\sum_{[V]}\frac{|\Ext^1(L,M)_V|}{|\Hom(L,M)|}\p([V]).
\]
If $\Ext^1(L,M)_V\neq \emptyset$, then there exists a short exact sequence
\[
0 \longrightarrow M \stackrel{f_1}{\longrightarrow} V\stackrel{f_2}{\longrightarrow} L \longrightarrow0,
\]
which yields the following pushout diagram:
\begin{equation}
\label{eq:dag}
\xymatrix{ K\ar[r]^{h_1} \ar@{=}[d] & M\ar[r]^{h_2}\ar[d]^{f_1} & N\ar[d]\\
K \ar[r] & V\ar[r]\ar[d]^{f_2} &W\ar[d] &
\\
&L\ar@{=}[r]& L}
\end{equation}
For any $[W]\in\Iso(\mod( A))$ denote by $\cz_{[W]}$ the set formed by all $[V]$ such that there exists a diagram of the form \eqref{eq:dag}.
Note that $[V]=[K\oplus W]$ in $ (\ch( A)/I')[\cs_A^{-1}]$.
So  \begin{align*}
\p([L]\diamond[M])
&= \sum_{[V]}\frac{|\Ext^1(L,M)_V|}{|\Hom(L,M)|}\p([V])\\
&= \sum_{[W]} \sum_{[V]\in\cz_{[W]}} \frac{|\Ext^1(L,M)_V|}{|\Hom(L,M)|}[K\oplus W].
\end{align*}

Applying $\Hom(L,-)$ to the short exact sequence $0\rightarrow K\xrightarrow{h_1} M\xrightarrow{h_2}N\rightarrow0$, we obtain a long exact sequence
\begin{eqnarray*}
&&0\rightarrow \Hom(L,K)\rightarrow \Hom(L,M)\rightarrow \Hom(L,N)\rightarrow\Ext^1(L,K)\\
&&\;\;\;\rightarrow\Ext^1(L,M)\xrightarrow{\varphi} \Ext^1(L,N)\rightarrow\Ext^2(L,K)=0.
\end{eqnarray*}
So $\varphi$ is surjective. In particular, for any $[W]$, $\varphi^{-1}(\Ext^1(L,N)_W)=\bigcup_{[V]\in\cz_{[W]}}\Ext^1(L,M)_V$.
The cardinality of the fibre of $\varphi: \bigcup_{[V]\in\cz_{[W]}}\Ext^1(L,M)_V\rightarrow \Ext^1(L,N)_W$ is equal to
$|\ker(\varphi)|$, and then equal to $\frac{|\Ext^1(L,K)| |\Hom(L,M)|}{|\Hom(L,K)||\Hom(L,N)|}$.
So
\begin{align*}
\p([L]\diamond[M])&= \sum_{[W]}\frac{| \Ext^1(L,N)_W||\Ext^1(L,K)| |\Hom(L,M)|}{|\Hom(L,K)||\Hom(L,N)||\Hom(L,M)|} [K\oplus W]\\
&= \sum_{[W]}\frac{| \Ext^1(L,N)_W||\Ext^1(L,K)| }{|\Hom(L,K)||\Hom(L,N)|} [K\oplus W].
\end{align*}

On the other hand, we have $\p([L]\diamond[K\oplus N])=\sum_{[U]} \frac{|\Ext^1(L,K\oplus N)_U|}{|\Hom(L,K\oplus N)|}[U]$.
By applying $\Hom(L,-)$ to the split exact sequence $0\rightarrow K\rightarrow K\oplus N\rightarrow N\rightarrow0$, we obtain a short exact sequence
\[
0\longrightarrow \Ext^1(L,K)\longrightarrow \Ext^1(L,K\oplus N)\stackrel{\phi}{\longrightarrow} \Ext^1(L,N) \longrightarrow0.
\]
Then $\phi$ induces a surjective map $\phi:\bigcup_{[U]}\Ext^1(L,K\oplus N)_U \rightarrow \bigcup_{[W]} \Ext^1(L,N)_W$.
For any $\xi\in \Ext^1(L,N)_W$, the cardinality of $\phi^{-1}(\xi)$ is $|\Ext^1(L,K)|$, and for any $0\rightarrow K\oplus N\rightarrow U\rightarrow L\rightarrow0$ in $\phi^{-1}(\xi)$, we have $[U]=[K\oplus W]$ in $ (\ch( A)/I')[\cs_A^{-1}]$. So
\begin{eqnarray*}
\p([L]\diamond[K\oplus N])=\sum_{[W]} \frac{|\Ext^1(L,N)_W||\Ext^1(L,K)|}{|\Hom(L,K\oplus N)|}[K\oplus W].
\end{eqnarray*}
Therefore, $\p([L]\diamond[M])=\p([L]\diamond[K\oplus N])$.
\end{proof}

\begin{lemma}\label{lemma right ideal}
Let $0\rightarrow K\xrightarrow{h_1} M\xrightarrow{h_2} N\rightarrow 0$ be a short exact sequence in $\mod( A)$ with $K\in\fpr( A)$. Then for any $L$ we have
$$ \p(([M]-[K\oplus N] )\diamond[L])=0,$$
where $\p: { \ch( A) } \rightarrow  (\ch( A)/I')[\cs_A^{-1}]$ is the natural projection.
\end{lemma}

\begin{proof}
It follows from Lemma \ref{lemma approx of GP} that there exists a short exact sequence
\[
0\longrightarrow M \stackrel{h_1}{\longrightarrow} H^M \stackrel{h_2}{\longrightarrow} G^M \longrightarrow 0
\]
with
$H^M\in\fpr( A)$ and $G^M\in\Gproj( A)$.
Then we obtain the following pushout diagram
\[
\xymatrix{K \ar@{=}[r] \ar[d]^{f_1} & K\ar[d]\\
M\ar[r]^{h_1}\ar[d]^{f_2}& H^M \ar[r]^{h_2} \ar[d] & G^M\ar@{=}[d] \\
N\ar[r] &C\ar[r]& G^M}
\]
We have $C\in\fpr( A)$ by using the short exact sequence in the second column, and then
\begin{equation}\label{eqn:eqn in lemma right ideal 1}
[H^M]=q^{\langle C,K\rangle}[C]\diamond [K]
\end{equation}
in $\T( A)$.
The above pushout diagram implies that there exists
a short exact sequence
\begin{equation}\label{equation 2}
0\longrightarrow M\stackrel{g_1}{\longrightarrow} H^M\oplus N\stackrel{g_2}{\longrightarrow} C\longrightarrow0.
\end{equation}
By applying $\Hom(-,L)$ to (\ref{equation 2}), we obtain a long exact sequence
\begin{eqnarray*}
&&0\rightarrow \Hom(C,L)\rightarrow \Hom(H^M\oplus N,L)\rightarrow \Hom(M,L)\rightarrow\Ext^1(C,L)\\
&&\;\;\; \rightarrow\Ext^1(H^M\oplus N,L)\xrightarrow{\varphi} \Ext^1(M,L)\rightarrow\Ext^2(C,L)=0.
\end{eqnarray*}
In particular, $\varphi$ is surjective.
Similar to the proof of Lemma \ref{lemma left ideal}, in $(\ch( A)/I')[\cs_A^{-1}]$, we obtain that
\begin{align*}
[M]\diamond[L]&= \sum_{[V]} \frac{|\Ext^1(M,L)_V|}{|\Hom(M,L)|} [V]\\
&= \sum_{[W]} \frac{|\Ext^1(H^M\oplus N,L)_W||\Hom(C,L)|}{ |\Ext^1(C,L)||\Hom(H^M\oplus N,L)|} q^{-\langle C,V\rangle} [C]^{-1}\diamond[W]\\
&= \sum_{[W]} \frac{|\Ext^1(H^M\oplus N,L)_W||\Hom(C,L)|}{ |\Ext^1(C,L)||\Hom(H^M\oplus N,L)|} q^{-\langle C,L\oplus M\rangle} [C]^{-1}\diamond[W]\\
&= \sum_{[W]} \frac{|\Ext^1(H^M\oplus N,L)_W|}{ |\Hom(H^M\oplus N,L)|}q^{-\langle C, M\rangle} [C]^{-1}\diamond[W].
\end{align*}
By applying $\Hom(-,L)$ to the split exact sequence
\[
0\longrightarrow N\longrightarrow H^M\oplus N\longrightarrow H^M\longrightarrow0,
\]
we have a short exact sequence
\[
0 \longrightarrow \Ext^1(H^M,L) \longrightarrow \Ext^1(H^M\oplus N,L)\stackrel{\phi}{\longrightarrow} \Ext^1(N,L)\longrightarrow0.
\]
Then $\phi$ induces a surjective map $\phi:\bigcup_{[W]}\Ext^1(H^M\oplus N,L)_W \rightarrow \bigcup_{[X]} \Ext^1(N,L)_X$. For any $\xi\in \Ext^1(N,L)_X$, the cardinality of $\phi^{-1}(\xi)$ is $|\Ext^1(H^M,L)|$, and for any $0\rightarrow L\rightarrow W\rightarrow H^M\oplus N\rightarrow0$ in $\phi^{-1}(\xi)$, we have $[W]=[H^M\oplus X]=q^{\langle H^M,X\rangle}[H^M]\diamond [X]$ in $ (\ch( A)/I')[\cs_A^{-1}]$. So in $(\ch( A)/I')[\cs_A^{-1}]$,
\begin{align*}
[M]\diamond[L]&= \sum_{[X]} \frac{|\Ext^1( N,L)_X||\Ext^1(H^M,L)|}{|\Hom(H^M\oplus N,L)|}q^{-\langle C, M\rangle}q^{\langle H^M,X\rangle} [C]^{-1}\diamond[H^M]\diamond [X]\\
&= \sum_{[X]} \frac{|\Ext^1( N,L)_X| }{|\Hom(N,L)|}q^{-\langle H^M,L\rangle}q^{-\langle C,M\rangle}q^{\langle H^M,X\rangle} [C]^{-1}\diamond[H^M]\diamond [X]\\
&= \sum_{[X]} \frac{|\Ext^1( N,L)_X| }{|\Hom( N,L)|}q^{\langle K,N\rangle}[K]\diamond[X]\\
&= \sum_{[X]} \frac{|\Ext^1( N,L)_X| }{|\Hom( N,L)|}q^{-\langle K,L\rangle}[K\oplus X].
\end{align*}

Similarly, by applying $\Hom(-,L)$ to the split exact sequence $0\rightarrow K\rightarrow K\oplus N\rightarrow N\rightarrow0$, in $(\ch( A)/I')[\cs_A^{-1}]$, we obtain that
\begin{align*}
[K\oplus N]\diamond[L]&= \sum_{[U]} \frac{|\Ext^1(K\oplus N,L)_U|}{|\Hom(K\oplus N,L)|} [U]\\
&= \sum_{[W]} \frac{|\Ext^1(N,L)_X| |\Ext^1(K,L)|}{ |\Hom(K\oplus N,L)|}[K\oplus X]\\
&= \sum_{[X]} \frac{|\Ext^1( N,L)_X| }{|\Hom( N,L)|}q^{-\langle K,L\rangle}[K\oplus X]
= [M]\diamond [L].
\end{align*}
The proof is completed.
\end{proof}

\begin{proposition}
   \label{prop:isomorphism as T bimodule}
Let $ A$ be a finite-dimensional $1$-Gorenstein algebra over $\K$. Then
$\cs\cd\ch( A)$ is isomorphic to $(\ch( A)/I')[\cs_A^{-1}]$ as $\T( A)$-bimodules.
\end{proposition}

\begin{proof}
Having Lemma \ref{lemma left ideal} and Lemma \ref{lemma right ideal} available, the same proof as in \cite[Proposition ~3.19]{LP} works here.
\end{proof}

\subsection{Semi-derived Hall algebras}
   \label{subsec:SDH}

In this subsection, we shall prove that the semi-derived Ringel-Hall algebra $\cs\cd\ch( A)$ is isomorphic to the semi-derived Hall algebra of $\Gproj( A)$ defined in \cite{Gor2}.

First, let us recall the definition of semi-derived Hall algebras for Frobenius categories.
Let $\cf$ be a Frobenius category satisfying the following conditions:
\begin{itemize}
\item[(F1)] $\cf$ is essentially small, idempotent complete and linear over $\K=\F_q$;
\item[(F2)] $\cf$ is Hom-finite, and $\Ext^p$-finite for any $p>0$.
\end{itemize}
Denote by $\cp(\cf)$ the subcategory of $\cf$ consisting of projective-injective objects.
Let $\ch(\cf)$ be the Hall algebra of the exact category $\cf$. In \cite{Gor2}, as a generalization of Bridgeland's Hall algebra \cite{Br}, Gorsky defined the {\em semi-derived Hall algebra} $\cs\cd\ch(\cf)$ of the pair $(\cf,\cp(\cf))$ to be the localization of $\ch(\cf)$ at the classes of all projective-injective objects:
\[
\cs\cd\ch(\cf):= \ch(\cf) \big[[P]^{-1} :P\in\cp(\cf) \big].
\]
It is worth noting that the semi-derived Hall algebra $\cs\cd\ch(\cf)$ defined here is compatible with the one defined in Definition \ref{def:semi-derived} since the ideal $I$ defined in  \S\ref{subsection:Def of MRH} is trivial for the Frobenius category $\cf$. Therefore, there is no ambiguity by using the same terminology and notation.

Denote by $\T(\cp(\cf))$ the subalgebra of $\cs\cd\ch(\cf)$ generated by all $[P]\in\Iso(\cp(\cf))$. Then we have natural left and right actions of $\T(\cp(\cf))$ on $\cs\cd\ch(\cf)$ given by the Hall multiplication. Denote by $\cm(\cf)$ this bimodule structure on $\cs\cd\ch(\cf)$.

\begin{theorem}[\cite{Gor2}]
Assume that $\cf$ satisfies Conditions (F1)--(F2). Then $\cm(\cf)$ is a free right (respectively, left) module over $\T(\cp(\cf))$. Each choice of representatives in $\cf$ of the isoclasses of the stable category $\underline{\cf}$ yields a basis for this free module.
\end{theorem}

We shall always assume that $ A$ is a $1$-Gorenstein algebra. Since $\Gproj( A)$ is a Frobenius category,
we can define the semi-derived Hall algebra $\cs\cd\ch(\Gproj( A))$.

Denote by $\T(\proj( A))$ the subalgebra of $\cs\cd\ch(\Gproj( A))$ generated by all $P\in\proj(A)$. Then $\T(\proj( A))$ is also isomorphic to the $\Q$-group algebra of $K_0(\proj(A))$ with the multiplication twisted by $q^{-\langle\cdot,\cdot\rangle}$.
For any $K\in\fpr(\ca)$, take a projective resolution of $K$: $0\rightarrow Q_K\rightarrow P_K\rightarrow K\rightarrow0$. Define $\psi:\ch(\fpr( A))\rightarrow \T(\proj( A))$ given by $\psi([K])=q^{-\langle K,Q_K\rangle}  [P_K]\diamond [Q_K]^{-1}$. Then $\psi$ is a morphism of algebras, which induces an isomorphism $\tilde{\psi}: \T( A)\rightarrow \T(\proj( A))$ by noting that $K_0(\proj( A))=K_0(\mod( A))$. We identify them in the following. Then we have natural left and right actions of $\T( A)$ on $\cs\cd\ch(\Gproj( A))$ given by the Hall product.

\begin{theorem}    \label{theorem isomorphism of algebras}
Let $ A$ be a finite-dimensional $1$-Gorenstein $k$-algebra. Then there exists an isomorphism of algebras: $\cs\cd\ch( A) \cong \cs\cd\ch(\Gproj( A))$.
\end{theorem}

\begin{proof}
Clearly, $\ch(\Gproj( A))$ is a subalgebra of $\ch( A)$, and we denote the inclusion by
\[
\phi: \ch(\Gproj( A))\longrightarrow \ch( A).
\]
Then we obtain a composition of morphisms $\ch(\Gproj( A))\xrightarrow{\phi} \ch( A)\rightarrow \ch( A)/I$, which is compatible with the localization. Therefore, this induces a morphism of algebras $\tilde{\phi}:\cs\cd\ch(\Gproj( A))\rightarrow \cs\cd\ch( A)$.

For any $[M]\in \Iso(\mod( A))$, there exists a short exact sequence $0\rightarrow H_M\rightarrow G_M\rightarrow M\rightarrow0$ with $G_M\in\Gproj( A)$, $H_M$ projective; cf. \cite{AB}. So we obtain that
\begin{align*}
[M]=q^{-\langle M,H_M\rangle}[G_M]\diamond[H_M]^{-1}
\end{align*}
and then
$[M]\in\cs\cd\ch(\Gproj( A))$. So $\tilde{\phi}$ is surjective.

On the other hand,
define $\psi:\ch( A)\rightarrow \cs\cd\ch(\Gproj( A))$ to be
$$\psi([M])= q^{-\langle M,H_M\rangle}[G_M]\diamond[H_M]^{-1},$$ where
$H_M,G_M$ satisfy the short exact sequence as above.

In order to prove that $\psi$ is well defined, let $0\rightarrow H_M\xrightarrow{f_1} G_M\xrightarrow{f_2} M\rightarrow 0$ be a short exact sequence with $f_2$ a minimal right $\Gproj( A)$-approximation of $M$. 
The right $\Gproj( A)$-approximation is of the form
$$0\longrightarrow H_M\oplus U_1\longrightarrow G_M\oplus U_1\longrightarrow M\longrightarrow0$$
with $U_1\in\proj( A)$.
Then in $\cs\cd\ch(\Gproj( A))$, we have
\begin{align*}
q^{-\langle M, H_M\oplus U_1\rangle} & [G_M\oplus U_1]\diamond[H_M\oplus U_1]^{-1}\\
=& q^{-\langle M, H_M\oplus U_1\rangle+\langle G_M,U_1\rangle-\langle H_M,U_1\rangle}[G_M]\diamond[U_1]\diamond[U_1]^{-1}\diamond [H_M]^{-1}\\
=& q^{-\langle M,H_M\rangle}[G_M]\diamond[H_M]^{-1}.
\end{align*}
So $\psi$ is well defined.

Let $0\rightarrow K\rightarrow L\rightarrow M\rightarrow0$ be a short exact sequence with $K\in \fpr( A)$.
Denote by $0\rightarrow Q_K\rightarrow P_K\rightarrow K\rightarrow0$ a projective resolution of $K$.  By the Horseshoe Lemma, we have the following commutative diagram with all rows and columns short exact
\[\xymatrix{Q_K\ar[r]\ar[d]& Q_K\oplus H_M \ar[r]\ar[d]& H_M\ar[d] \\
P_K \ar[r]\ar[d]& P_K\oplus G_M\ar[r]\ar[d]& G_M\ar[d]\\
K\ar[r] &L\ar[r]&M}\]
So
$$\psi([L])= q^{-\langle L,H_M\oplus Q_K\rangle}[P_K\oplus G_M]\diamond[Q_K\oplus H_M]^{-1}.$$
By using the following resolution
\begin{align*}
0\longrightarrow Q_K\oplus H_M\longrightarrow P_K\oplus G_M\longrightarrow K\oplus M\longrightarrow0,
\end{align*}
one sees that $\psi([K\oplus M])= q^{-\langle L,H_M\oplus Q_K\rangle}[P_K\oplus G_M]\diamond[Q_K\oplus H_M]^{-1} = \psi([L])$. Therefore,
$\psi$ induces a map
$$\tilde{\psi}: \ch( A)/I'
\longrightarrow \cs\cd\ch(\Gproj( A))$$
which is a morphism of $\T( A)$-bimodules.
Since $\psi([K])$ is invertible in $ \cs\cd\ch(\Gproj( A))$ for any $K\in\fpr( A)$, $\psi$ induces a unique morphism of $\T( A)$-bimodules
$$\tilde{\psi}: (\ch( A)/I')[\cs_A^{-1}] \longrightarrow \cs\cd\ch(\Gproj( A)),$$
and then a morphism of $\T( A)$-bimodules $\cs\cd\ch( A)\rightarrow \cs\cd\ch(\Gproj( A))$, which is also denoted by $\tilde{\psi}$, by Proposition \ref{prop:isomorphism as T bimodule}.
Clearly, $\tilde{\psi}\tilde{\phi}=\id$, and then $\tilde{\phi}$ is injective.
Therefore, $\tilde{\phi}$ is an isomorphism of algebras.
\end{proof}

\begin{lemma}[\cite{Gor2}]
\label{lemma basis of semi-derived Hall algebra}
$\cs\cd\ch(\Gproj( A))$ is a free right (respectively, left) module over $\T( A)$. Each choice of representatives in $\Gproj( A)$ of the isoclasses of the stable category $\underline{\Gproj}( A)$ yields a basis for this free module.
\end{lemma}

Denote by $\Gproj^{\np}( A)$ the smallest subcategory of $\Gproj ( A)$ formed by all Gorenstein projective modules without any projective summands.

For any $\alpha\in K_0(\mod( A))$, there exist $U,V\in \mod( A)$ such that $\alpha=\widehat{U}-\widehat{V}$, we set
$\E_\alpha:=q^{-\langle\alpha,\widehat{V}\rangle}[U]\diamond[V]^{-1}$; this is well defined, see e.g. \cite[\S3.4]{LP}. Let $K_0^+(\mod( A))$ be positive cone of $K_0(\mod( A))$, that is the subset of $K_0(\mod( A))$ corresponding to classes of objects in $\mod( A)$. Then for any $\alpha\in K_0^+(\mod( A))$, $\E_\alpha=[U]$ for any $U\in\mod( A)$ with $\widehat{U}=\alpha$. For convenience, we view  $\E_\alpha$ as an (isoclass of) object (by identifying with $[U]$ such that $\widehat{U}=\alpha$) when considered in $\cs\cd\ch( A)$.

\begin{lemma}
\label{cor:basis of MRH as torus-module 1-Gor}
$\cs\cd\ch( A)$ has a basis given by
\[
\{ [M]\diamond \E_\alpha \mid [M]\in\Iso(\Gproj^{\np}( A)), \alpha\in K_0(\mod( A)) \}.
\]
\end{lemma}

\begin{proof}
It follows from Theorem \ref{theorem isomorphism of algebras} and Lemma \ref{lemma basis of semi-derived Hall algebra} immediately.
\end{proof}

As $ A$ is $1$-Gorenstein, it follows by Theorem \ref{thm:Buchweitz-Happel} that $\underline{\Gproj}( A)\simeq D_{sg}(\mod( A))$; in particular, each representative of the isoclasses of $D_{sg}(\mod( A))$ can be chosen to be an $ A$-module, where every $ A$-module is viewed as a stalk complex concentrated at degree $0$. This yields the following result.
\begin{theorem}
\label{theorem basis of MRH 1-Gor}
$\cs\cd\ch( A)$ is a free right (respectively, left) module over $\T( A)$. Each choice of representatives in $\mod( A)$ of the isoclasses of $D_{sg}(\mod( A))$ gives a basis for this free module.
\end{theorem}

\subsection{Tilting invariance}

Let $ A$ be a $1$-Gorenstein algebra over $\K$.
Recall that an $ A$-module $T$ is called  \emph{tilting} if
\begin{itemize}
\item[(T1)] $\pd T\leq1$;
\item[(T2)] $\Ext^i(T,T)=0$ for any $i>0$;
\item[(T3)] 
there exists a short exact sequence $0\rightarrow  A \rightarrow T_0\rightarrow T_1\rightarrow0$ with $T_0,T_1\in \add T$.
\end{itemize}
Let $\Gamma=\End_ A(T)^{op}$.
It is well known that if $T$ is a tilting module, then there is a derived equivalence $\RHom_ A(T,-):D^b( A)\xrightarrow{\simeq} D^b(\Gamma)$. 
In this subsection, we shall prove that $\cs\cd\ch( A)$ is isomorphic to $\cs\cd\ch(\Gamma)$ if $\Gamma$ is also $1$-Gorenstein.



Let $\Fac T$ be the full subcategory of $\mod( A)$ of epimorphic images of objects in $\add T$.
The following lemma is well known, see e.g. \cite[Chapter VI.2]{ASS}.

\begin{lemma}\label{lem: torision pair}
Let $ A$ be a $1$-Gorenstein algebra with a tilting module $T$. Let $\cu= \Fac T$, and $\cv=\{M\in\mod( A)\mid\Hom(T,M)=0\}$. Then
\begin{itemize}
\item[(a)] $(\cu,\cv)$ is a torsion pair in $\mod( A)$;
\item[(b)] $\Ext^i(T,-)|_{\cu}=0$ for any $i>0$;
\item[(c)] for any $M\in \mod( A)$, there exists a short exact sequence
\[
0 \longrightarrow M \longrightarrow X_M \longrightarrow T_M \longrightarrow 0
\]
with $X_M\in\cu$ and $T_M\in\add T$.
\end{itemize}
\end{lemma}
Let $T$ be a tilting $ A$-module. Then $\cu$ is an exact category as a subcategory of $\mod( A)$. Furthermore, $\cu$ has enough projective objects with $\add T$ as the subcategory of projective objects of $\cu$.

\begin{lemma}
$\cu$ is an exact category satisfying the conditions (E-a)--(E-d) in \S\ref{subsection:Def of MRH}.
\end{lemma}

\begin{proof}
It is enough to verify that $\cu$ is weakly $1$-Gorenstein.
For any $M\in\mod( A)$, it follows by Lemma~ \ref{lem: torision pair}  there exists a short exact sequence
\begin{align}
\label{eqn: T-resolution of M}
0 \longrightarrow M \longrightarrow X_M \longrightarrow T_M \longrightarrow 0
\end{align}
 with $X_M\in\cu$ and $T_M\in\add T$.

For any $L\in\cu$ with $\Ext\text{-}\pd_\cu L<\infty$, by applying $\Hom(L,-)$ to \eqref{eqn: T-resolution of M}, we have $\Ext_{ A}^i(L,M)=0$ for $i$ large enough by noting that $\ind T_M\leq1$. Since $M$ is arbitrary, we obtain that $\pd_ A L<\infty$ and then $\pd_ A L\leq 1$. Since $ A$ is $1$-Gorenstein, this implies that $\ind_ A L\leq1$. Since $\cu$ is the subcategory of $\mod( A)$ closed under taking extensions, we have $\Ext\text{-}\ind_\cu L\leq1$.

For any $L\in\cu$ with $\Ext\text{-}\ind_\cu L<\infty$, by applying $\Hom(-,L)$ to \eqref{eqn: T-resolution of M}, dually, one can show that $\Ext\text{-}\ind_\cu L\leq1$, and $\Ext\text{-}\pd_\cu L\leq1$.

From the above, we obtain that $\cp^{< \infty}(\cu)= \fpr(\cu)=\mathcal{I}^{\leq 1}(\cu)=\mathcal{I}^{<\infty}(\cu)$. Then $\cu$ is weakly $1$-Gorenstein.
\end{proof}

In fact, from the proof, we obtain that $\cp^{<\infty}(\cu)\subseteq \cp^{<\infty}( A)$, and $\mathcal{I}^{<\infty}(\cu)\subseteq \mathcal{I}^{<\infty}( A)$.
So we can define the semi-derived Ringel-Hall algebra $\cs\cd\ch(\cu)$ of $\cu$. Furthermore, it follows from $K_0(\fpr(\cu))\cong K_0(\add T)\cong K_0(\proj(A))\cong K_0(\fpr( A))$ that $\cs\cd\ch(\fpr(\cu))\cong \T( A)$ by noting that $\cu$ is full subcategory of $\mod( A)$ which is closed under taking extensions.


\begin{proposition} \label{prop:derived invariant of MRH}
Let $ A$ be a $1$-Gorenstein algebra with a tilting module $T$, and $\cu=\Fac T$.
Then the natural embedding $\Phi:\ch(\cu)\rightarrow \ch( A)$ induces an algebra isomorphism
$$\tilde{\Phi}: \cs\cd\ch(\cu) \stackrel{\simeq}{\longrightarrow} \cs\cd\ch( A).$$
Furthermore, the inverse morphism of $\tilde{\Phi}$ is given by
$\tilde{\Psi}:[M]\mapsto q^{-\langle M, T_M\rangle} [T_M]^{-1}\diamond[X_M]$,
where $M$, $X_M\in\cu$ and $T_M\in\add T$ fits into a short exact sequence
\[
0 \longrightarrow M \longrightarrow X_M \longrightarrow T_M \longrightarrow 0.
\]
\end{proposition}

\begin{proof}
The proof is similar to Theorem \ref{theorem isomorphism of algebras} by using \eqref{eqn: T-resolution of M}, and we omit it here.
\end{proof}

\begin{theorem}    \label{theorem:derived equivalence of MRH}
Let $ A$ be a $1$-Gorenstein algebra with a tilting module $T$. If $\Gamma:=\End_{\La}(T)^{op}$ is a $1$-Gorenstein algebra, then there exists an isomorphism of algebras
\begin{align*}
\Xi: & \cs\cd\ch( A) \stackrel{\simeq}{\longrightarrow} \cs\cd\ch(\Gamma)
\\
 & [M]\mapsto q^{-\langle M,T_M\rangle} [F(T_M)]^{-1}\diamond [F(X_M)],
 \end{align*}
 where $F=\Hom_ A(T,-)$. Here $M$, $X_M\in\cu$ and $T_M\in\add T$ fit into a short exact sequence
$$0 \longrightarrow M \longrightarrow X_M \longrightarrow T_M \rightarrow 0.$$
\end{theorem}

\begin{proof}
Theorem \ref{theorem isomorphism of algebras} shows that $\cs\cd\ch(\Gamma)\xrightarrow{\simeq}\cs\cd\ch(\Gamma)$ with the isomorphism induced by
the natural embedding $\ch(\Gproj(\Gamma))\hookrightarrow \ch(\Gamma)$.
Denote by $G=T\otimes_\Gamma -:\mod(\Gamma)\rightarrow \mod( A)$.
Then $(\cu,\cv)$ induces a torsion pair $(\cx,\cy)$ in $\mod(\Gamma)$, where
\[
\cx=\{X\in\mod(\Gamma)\mid T\otimes_\Gamma X=0\},
\qquad
\cy=\{Y\in\mod(\Gamma)\mid \Tor_1^\Gamma(T,Y)=0\}.
\]
We claim that $\Gproj(\Gamma)\subseteq \cy$. In fact, $T_\Gamma$ is a right tilting $\Gamma$-module by classical tilting theory. It follows that $\pd T_\Gamma\leq1$, and then $\ind_\Gamma DT\leq1$. Since $\Gamma$ is $1$-Gorenstein, we also have $\pd_\Gamma DT\leq1$. For any $Y\in \Gproj(\Gamma)$, we have by Lemma \ref{lemma perpendicular of P GP} that $\Tor_1^\Gamma(T,Y)\cong D\Ext^{1}_\Gamma(Y,DT)=0$, so $Y\in\cy$.

By the Brenner-Butler tilting theorem \cite{BB80},
the functors $F$ and $G$ induce quasi-inverse equivalences between $\cu$ and $\cy$. In particular, $F:\cu\rightarrow \cy$ and $G:\cy\rightarrow\cu$ are exact and preserve projective objects. So $\cy$ also satisfies (E-a)--(E-d), and then $\cs\cd\ch(\cy)$ is well defined. Then $F$ and $G$ induce the equivalence $\cs\cd\ch(\cu)\cong \cs\cd\ch(\cy)$. We claim that $\cs\cd\ch(\cy)\cong \cs\cd\ch(\Gamma)$. If so, together with
$\cs\cd\ch(\cu) \cong \cs\cd\ch(\Lambda)$ by Proposition \ref{prop:derived invariant of MRH}, we have proved that
$\cs\cd\ch(\Lambda)\cong \cs\cd\ch(\Gamma)$.

It remains to prove the claim that $\cs\cd\ch(\cy)\cong \cs\cd\ch(\Gamma)$. Since $\Gproj(\Gamma)$ and $\cy$ are closed under taking extensions, from above, we have injective homomorphisms ${\phi}:\ch(\Gproj (\Gamma))\rightarrow\ch(\cy)$ and $\varphi: \ch(\cy)\rightarrow\ch(\Gamma)$. Then
the natural embeddings $\ch(\Gproj (\Gamma))\stackrel{\phi}{\rightarrow}\ch(\cy)\stackrel{\varphi}{\rightarrow} \ch(\Gamma)$ induce morphisms of algebras
$\cs\cd\ch(\Gproj(\Gamma))\stackrel{\tilde{\phi}}{\rightarrow} \cs\cd\ch(\cy)\stackrel{\tilde{\varphi}}{\rightarrow} \cs\cd\ch(\Gamma)$. By Theorem \ref{theorem isomorphism of algebras}, $\tilde{\varphi}\tilde{\phi}$ is an isomorphism. So $\tilde{\phi}$ is injective. However, similar to the proof of Theorem~ \ref{theorem isomorphism of algebras}, one sees that $\tilde{\phi}$ is surjective. Then both $\tilde{\phi}$ and $\tilde{\varphi}$ are isomorphisms. So the claim follows.
\end{proof}

\begin{corollary}
Let $ A$ be a $1$-Gorenstein algebra with a tilting module $T$. If $\Gamma=\End(T)^{op}$ is a $1$-Gorenstein algebra, then
 \[
 \cs\cd\ch(\Gproj( A))\cong\cs\cd\ch(\Gproj(\Gamma)).
 \]
\end{corollary}

\begin{proof}
Recall from Theorem \ref{theorem isomorphism of algebras} that $\cs\cd\ch(\Gamma)\xrightarrow{\simeq}\cs\cd\ch(\Gproj(\Gamma))$. The assertion now follows from Theorem~ \ref{theorem:derived equivalence of MRH}.
\end{proof}


\end{document}